\documentclass[journal]{IEEEtran}

\usepackage{comment}
\usepackage{cite}
\usepackage{amsmath,amssymb,amsfonts,mathrsfs}
\usepackage{algorithm,algorithmic}
\usepackage{graphicx}
\usepackage{textcomp}
\usepackage{xcolor}
\usepackage{mathtools}
\usepackage{dsfont}
\usepackage{subcaption}
\usepackage{algorithm}
\usepackage{algorithmic}
\usepackage{makecell}
 
\usepackage{epstopdf} 
\usepackage{graphics} % for pdf, bitmapped graphics files 
\usepackage{epstopdf} 
\usepackage{url}
\usepackage{color,soul}
\usepackage{amsthm}
\usepackage[colorlinks=true,linkcolor=blue, citecolor=green, urlcolor=magenta]{hyperref}

\theoremstyle{definition}

\newtheorem{theorem}{Theorem}
\newtheorem{assumption}{Assumption}
\newtheorem{definition}{Definition}
\newtheorem{proposition}{Proposition}
\newtheorem{lemma}{Lemma}
\newtheorem{remark}{Remark}

\AtEndEnvironment{remark}{\hfill\ensuremath{\square}}
\AtEndEnvironment{definition}{\hfill\ensuremath{\square}}
\makeatletter
\@ifundefined{proof}{% if \proof is not defined, do nothing
}{%
  \let\oldproof\proof
  \renewcommand{\proof}[1]{\oldproof{#1\hfill\ensuremath{\blacksquare}}}
}
\makeatother

\newcommand{\Lim}{\displaystyle\lim}
\allowdisplaybreaks

\title{Dynamic Constrained Stabilization on the $n$-sphere (Extended version)
} 
\author{Mayur Sawant and Abdelhamid Tayebi  
	\thanks{This work was supported by the Natural Sciences and Engineering Research Council of Canada (NSERC), under the grants RGPIN-2020-06270.}
	\thanks{M. Sawant and A. Tayebi are with the Department of Electrical and Computer Engineering, Lakehead University, Thunder Bay, ON P7B 5E1, Canada. (e-mail: {\tt\small msawant, atayebi@lakeheadu.ca}).}%
}%
\begin{document}

\maketitle
\begin{abstract}
We consider the constrained stabilization problem of second-order systems evolving on the $n$-sphere. 
We propose a control strategy with a constraint proximity-based dynamic damping mechanism that ensures safe and almost global asymptotic stabilization of the target point in the presence of star-shaped constraints on the $n$-sphere. 
It is also shown that the proposed approach can be used to deal with constrained rigid-body attitude stabilization.
The effectiveness of our approach is demonstrated through simulation results on the $2$-sphere and the $3$-sphere in the presence of star-shaped constraint sets.
\end{abstract}
%%%%%%%%%%%%%%%%%%%%%%%%%%%%%%
\section{Introduction}

Various mechanical systems have states that evolve on the $n$-sphere, such as spin-axis stabilization of rigid body systems~\cite{bullo1995control}, two-axis gimbal systems~\cite{osborne2008global}, thrust-vector control for quadrotor aircraft \cite{hua2015control}, and the spherical robot \cite{muralidharan2015geometric}.
In many practical scenarios, the attitude stabilization problem can also be recast as a stabilization on the $3$-sphere.
The stabilization problem on the $n$-sphere (without constraints) has been addressed in the literature using differential geometry and hybrid dynamical systems tools, see for instance \cite{bullo1995control}, \cite{casau2019hybrid}, \cite{casau2019robust}.

In \cite{lee2014feedback}, a logarithmic barrier function is used to design a quaternion-based feedback controller for rigid body attitude stabilization in the presence of multiple attitude-constrained zones, and a single attitude-mandatory zone, characterized by quadratic inequalities.
In \cite{nicotra2019spacecraft}, the authors employ an explicit reference governor to ensure trajectory tracking under conic constraints and actuator saturation limits.
In \cite{danielson2021spacecraft}, an invariant set motion planner is proposed to generate a sequence of quaternion waypoints that steer the spacecraft attitude to a desired configuration while enforcing conic constraints on a body-fixed direction evolving on the $2$-sphere.
In \cite{berkane2021constrained}, the authors addressed the stabilization problem on the $n$-sphere under conic constraints by leveraging the stereographic projection to transform the problem into a classical navigation problem in $\mathbb{R}^n$ with spherical obstacles, enabling the use of existing navigation function-based obstacle avoidance methods.

Although existing approaches address constrained stabilization for dynamical systems, the constraint representations are often limited to conic sets.
%As illustrated in Fig.~\ref{fig:conic_vs_star}a, 
Conic constraints are conservative and exclude a significant portion of the free space from the feasible region for stabilization compared to star-shaped constraints (which include conic constraints as a special case).
In our earlier work \cite{sawant2026constrained}, we introduced a kinematic-level constrained stabilization framework on the $n$-sphere that accommodates star-shaped constraint sets.
%Fig.~\ref{fig:conic_vs_star}b provides examples of star-shaped sets on the $2$-sphere.

In this work, we propose a control strategy for second-order systems on the $n$-sphere with star-shaped obstacles, relying on constraint-proximity-based dynamic damping, guaranteeing safety over the feasible state space and almost\footnote{Almost global convergence here refers to the convergence from all initial conditions except a set of zero Lebesgue measure.} global asymptotic stabilization of the target location.
%we extend the kinematic control schemes on the $n$-sphere to second-order systems.
%Given a kinematic controller that renders a desired point on the $n$-sphere almost globally asymptotically stable while ensuring safety over the free space, we propose a control strategy for second order systems on the $n$-sphere, relying on constraint-proximity-based dynamic damping, guaranteeing safety over the feasible state space and almost global asymptotic stabilization of the target location.

%{\color{blue}A preliminary version of this work has been accepted in \cite{sawant2026dynamicCDC26}. 
%Compared to that, the present paper provides complete theoretical analysis alongside non-trivial simulations on the $2$-sphere and the $3$-sphere.}

\section{Notations and preliminaries}\label{section:notations}
The sets of real numbers, non-negative real numbers, and positive integers are denoted by $\mathbb{R}$, $\mathbb{R}_{\geq 0}$, and $\mathbb{N}$, respectively.
Bold lowercase symbols denote vector quantities.
The notation $\mathbf{0}_n$ represents a zero column vector of dimension $n$.
The $n$-dimensional zero matrix and identity matrix are denoted by $\mathbf{O}_n$ and $\mathbf{I}_{n}$, respectively.
The Frobenius norm of a matrix $\mathbf{A}\in\mathbb{R}^{n\times n}$ is defined as $\|\mathbf{A}\|_{F} = \sqrt{\sum_i\sum_ka_{ik}^2}$, where $a_{ik}$ is the element in the $i$-th row and the $k$-th column of $\mathbf{A}$.
Given $\mathcal{A}\subset\mathbb{R}^n$ and $\mathcal{B}\subset\mathbb{R}^n$, the relative complement of $\mathcal{B}$ in $\mathcal{A}$ is given by $\mathcal{A}\setminus\mathcal{B} = \{\mathbf{a}\in\mathcal{A}\mid \mathbf{a}\notin\mathcal{B}\}$.
%Given $\mathcal{A}\subset\mathbb{R}^n$, the cardinality of $\mathcal{A}$ is denoted by $\mathrm{card}(\mathcal{A})$.
For a twice continuously differentiable scalar mapping $f:\mathbb{R}\to\mathbb{R}$, we denote its first and second derivatives by $f'(x) = \frac{df(x)}{dx}$ and $f''(x) = \frac{d^2f(x)}{dx^2}$, respectively.

The $n$-sphere $\mathbb{S}^n$ is an $n$-dimensional manifold embedded in the Euclidean space $\mathbb{R}^{n+1}$ and defined as \[\mathbb{S}^n:=\{\mathbf{x}\in\mathbb{R}^{n+1}\mid\|\mathbf{x}\| = 1\}.\]
Given a set $\mathcal{A}\subset\mathbb{S}^n$, the symbols $\overline{\mathcal{A}}, \mathcal{A}^{\circ}$, and $\partial\mathcal{A}$ represent the closure, interior, and the boundary of $\mathcal{A}$ on $\mathbb{S}^n$, where $\partial\mathcal{A} = \overline{\mathcal{A}}\setminus\mathcal{A}^{\circ}$.
In the following, we will provide the definitions of some concepts related to the $n$-sphere that will be used throughout the paper.

\noindent{\bf Tangent space on \texorpdfstring{$\mathbb{S}^n$}{}:} The tangent space to $\mathbb{S}^n$ at $\mathbf{x}\in\mathbb{S}^n$ is given by $\mathsf{T}_{\mathbf{x}}\mathbb{S}^n = \{\mathbf{a}\in\mathbb{R}^{n+1}\mid\mathbf{a}^\top\mathbf{x} = 0\}$, which represents all vectors in $\mathbb{R}^{n+1}$ that are perpendicular to $\mathbf{x}$.
Given $\mathbf{x}\in\mathbb{S}^n$ and $\mathbf{v}\in\mathbb{R}^{n+1}$, the orthogonal projection operator $\mathbf{P}(\mathbf{x})$, which is given by 
\begin{equation}\label{orthogonal_projection_operator_formula}
    \mathbf{P}(\mathbf{x}) = \mathbf{I}_{n+1} - \mathbf{x}\mathbf{x}^\top,
\end{equation}
projects $\mathbf{v}$ onto the tangent space $\mathsf{T}_{\mathbf{x}}\mathbb{S}^n$, \textit{i.e.}, $\mathbf{P}(\mathbf{x})\mathbf{v}\in\mathsf{T}_{\mathbf{x}}\mathbb{S}^n$.

\noindent{\bf Geodesic:} For any two distinct points $\mathbf{a}, \mathbf{b}\in\mathbb{S}^n$ with $\mathbf{a} \ne -\mathbf{b}$, the unique geodesic connecting $\mathbf{a}$ and $\mathbf{b}$ is given by 
\begin{equation}\label{geodesic_expression}
    \mathcal{G}(\mathbf{a}, \mathbf{b}) = \left\{ \mathbf{x} \in \mathbb{S}^n \mid \mathbf{x} = g(\lambda; \mathbf{a}, \mathbf{b}), \lambda\in[0, 1]
    \right\},
\end{equation}
where, motivated by \cite[Section 3.3]{shoemake1985animating}, the mapping $g:[0, 1]\to\mathbb{S}^n$ is defined as
\begin{equation*}
    g(\lambda; \mathbf{a}, \mathbf{b}) = \frac{\sin((1-\lambda)\theta) \mathbf{a} + \sin(\lambda\theta) \mathbf{b}}{\sin\theta},
\end{equation*}
where $\theta = \arccos(\mathbf{a}^\top\mathbf{b})\in(0, \pi)$. 
For $\lambda\in[0, 1]$, one has $\mathbf{P}(g(\lambda; \mathbf{a}, \mathbf{b}))\frac{d^2g(\lambda;\mathbf{a}, \mathbf{b})}{d\lambda^2} = \mathbf{0}_{n+1}$.
Therefore, by \cite[Chap. 3, Def. 2.1]{do1992riemannian}, one can confirm that $\mathcal{G}(\mathbf{a}, \mathbf{b})$ is a geodesic and is the curve on $\mathbb{S}^n$ with the smallest path length, connecting $\mathbf{a}$ and $\mathbf{b}$.
If $\mathbf{a} = \mathbf{b}$, then the geodesic $\mathcal{G}(\mathbf{a}, \mathbf{b})$ is trivially the point itself \textit{i.e.,} $\mathcal{G}(\mathbf{a}, \mathbf{a}) = \{\mathbf{a}\}$.

% \noindent{\bf Spherical cone over a set:} Given a set $\mathcal{A}\subset\mathbb{S}^n$ and a point $\mathbf{p}\in\mathbb{S}^n$, the spherical cone $\mathcal{S}_{\mathcal{A}}(\mathbf{x})$ over $\mathcal{A}$ with respect to $\mathbf{p}$ is defined as
% \begin{equation*}
%     \mathcal{S}_{\mathcal{A}}(\mathbf{x}) = \bigcup_{\mathbf{x}\in\mathcal{A}}\mathcal{G}(\mathbf{x}, \mathbf{p}).
% \end{equation*}

\noindent {\bf Star-shaped sets on \texorpdfstring{$\mathbb{S}^n$}{}:} A set $\mathcal{A} \subset \mathbb{S}^n$ is a star-shaped set on $\mathbb{S}^n$ if there exists $\mathbf{g} \in \mathcal{A}$ with $-\mathbf{g} \notin \mathcal{A}$ such that $\mathcal{G}(\mathbf{g}, \mathbf{x}) \subset \mathcal{A}$ for all $\mathbf{x} \in \mathcal{A}$.

Given a star-shaped set $\mathcal{A}$ on $\mathbb{S}^n$, let $\sigma(\mathcal{A})$ be the set of all points $\mathbf{g}$ in $\mathcal{A}$ such that $-\mathbf{g}\notin\mathcal{A}$ and $\mathcal{G}(\mathbf{g}, \mathbf{x})\subset\mathcal{A}$ for all $\mathbf{x}\in\mathcal{A}$, defined as follows:
\begin{equation}\label{sigma_set}
    \sigma(\mathcal{A}) = \{\mathbf{g}\in\mathcal{A}\mid-\mathbf{g}\notin\mathcal{A},~\forall \mathbf{x}\in\mathcal{A},~\mathcal{G}(\mathbf{g}, \mathbf{x})\subset\mathcal{A}\}.
\end{equation}
A few examples of star-shaped subsets of $\mathbb{S}^n$ are shown in Fig.~\ref{fig:star_set_examples}. 
For each disjoint set $\mathcal{A}$ in Fig.~\ref{fig:star_set_examples}, the magenta dot denotes a point $\mathbf{g}$ selected from $\sigma(\mathcal{A})$.  
Note that the conic constraints, widely used in the literature \cite{danielson2021spacecraft, berkane2021constrained}, constitute a subclass of star-shaped sets on $\mathbb{S}^n$.
%%%%%%%%%%%%%%%%%%%%%%%%%%%%%%

\begin{figure}
    \centering
    \includegraphics[width=0.7\linewidth]{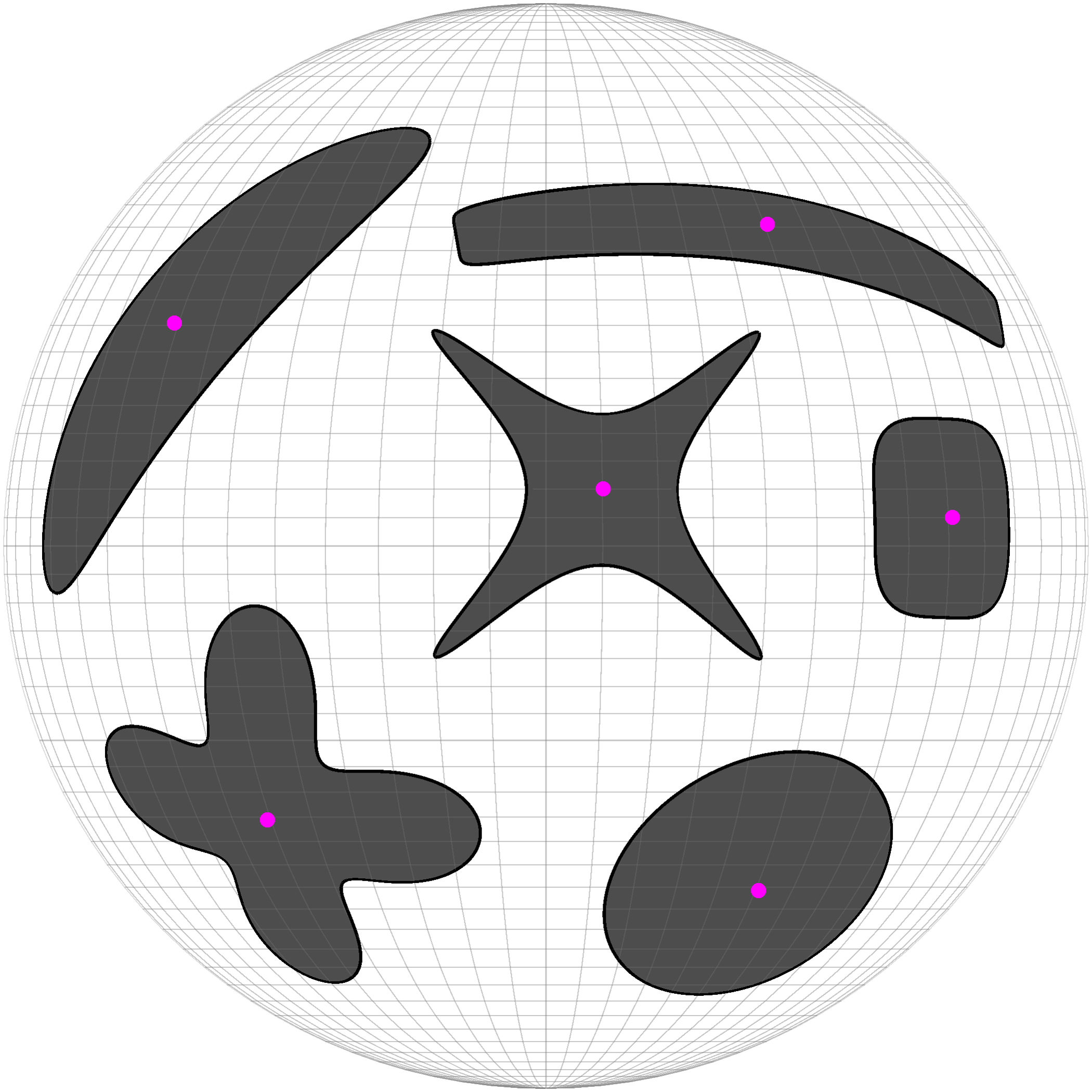}
    \caption{Star-shaped sets on $\mathbb{S}^n$.}
    \label{fig:star_set_examples}
\end{figure}

\section{Problem statement}\label{section:problem_statement}
Consider the following system:
\begin{equation}
\begin{aligned}\label{dynamics_motion_model_on_sphere}
        \dot{\mathbf{x}} &= \mathbf{P}(\mathbf{x})\mathbf{v},\\
        \dot{\mathbf{v}} &= \mathbf{u},
    \end{aligned}
\end{equation}
where $\mathbf{x}\in\mathbb{S}^n$, $\mathbf{v}\in\mathbb{R}^{n+1}$ and $\mathbf{u}\in\mathbb{R}^{n+1}$.
The orthogonal projection matrix $\mathbf{P}(\mathbf{x})$ is defined in \eqref{orthogonal_projection_operator_formula} and ensures that the velocity  vector $\mathbf{P}(\mathbf{x})\mathbf{v}$ belongs to the tangent space $\mathsf{T}_{\mathbf{x}}\mathbb{S}^n$ \textit{i.e.}, $\mathbf{P}(\mathbf{x})\mathbf{v}\in\mathsf{T}_{\mathbf{x}}\mathbb{S}^n$ for all $\mathbf{x}\in\mathbb{S}^n$ and for any $\mathbf{v}\in\mathbb{R}^{n+1}$.
This guarantees that if $\mathbf{x}(0)\in\mathbb{S}^{n}$ then $\mathbf{x}(t)\in\mathbb{S}^n$ for all $t\geq 0$.

The unsafe set $\mathcal{U}$ is the union of $m$ pairwise disjoint, closed, connected subsets $\mathcal{U}_i$ of $\mathbb{S}^n$, \textit{i.e.}, $\mathcal{U} = \bigcup_{i\in\mathbb{I}}\mathcal{U}_i$, where $i\in\mathbb{I}:=\{1, \ldots, m\}, m\in\mathbb{N}$.
The free space $\mathcal{M} = \mathbb{S}^n\setminus\mathcal{U}$ is an open subset of $\mathbb{S}^n$, and is assumed to be pathwise connected. 
We require that each constraint set $\mathcal{U}_i$, $i\in\mathbb{I}$, has a boundary that is twice continuously differentiable \textit{i.e.,} a smooth surface without corners, as mentioned in the next assumption.
\begin{assumption}\label{assumption:twice_differentiability}
    For each $i\in\mathbb{I}$, the boundary $\partial\mathcal{U}_i$ is an embedded submanifold of $\mathbb{S}^n$ of class $\mathcal{C}^2$, where a differentiable manifold of class $\mathcal{C}^2$ is defined using \cite[Definition 1.33]{lee2009manifolds}.
    %, and for every $i, j\in\mathbb{I}, i\ne j$, $\mathcal{U}_i\cap\mathcal{U}_j = \emptyset$.
\end{assumption}

We consider the continuous scalar mapping $d_{\mathcal{U}}:\overline{\mathcal{M}}\to\mathbb{R}_{\geq 0}$ that characterizes the separation between any $\mathbf{x}\in\overline{\mathcal{M}}$ and the unsafe set $\mathcal{U}$, and satisfies the following properties:
\begin{enumerate}
\renewcommand{\labelenumi}{D\arabic{enumi}}
\renewcommand{\theenumi}{D\arabic{enumi}}
    \item\label{property:distance1} $d_{\mathcal{U}}(\mathbf{x}) > 0, \;\forall\mathbf{x} \in \mathcal{M}$, and $d_{\mathcal{U}}(\mathbf{x}) = 0, \;\forall\mathbf{x} \in \partial\mathcal{U}$.
    \item \label{property:distance2}There exist $D_d > 0$ and a known parameter $\delta_u > 0$ such that
    for all $\mathbf{x}\in\mathcal{N}^{\delta_u}$, the function $d_{\mathcal{U}}$ is continuously differentiable and $\|\mathbf{P}(\mathbf{x})\nabla_{\mathbf{x}}d_{\mathcal{U}}(\mathbf{x})\|\leq D_d$,
\end{enumerate}
where the set $\mathcal{N}^{\delta_u}$ is defined as
\begin{equation}\label{neighborhood_definition}
    \mathcal{N}^{\delta_u} = \left\{\mathbf{x}\in\mathcal{M}\mid d_{\mathcal{U}}(\mathbf{x}) \in(0, \delta_u]\right\}.
\end{equation}
Examples of such separation functions are provided in Appendix \ref{remark:separation_function_examples}.

Let $\mathbf{v}_d:\mathcal{M}\to\mathbb{R}^{n+1}$ be a vector field such that the closed-loop system 
\begin{equation}\label{ideal_kinematics}\dot{\mathbf{x}} = \boldsymbol{\nu}_d(\mathbf{x}), \;\text{with}\;\boldsymbol{\nu}_d(\mathbf{x}):=\mathbf{P}(\mathbf{x})\mathbf{v}_d(\mathbf{x}),\end{equation} 
satisfies the following properties:
\begin{enumerate}
\renewcommand{\labelenumi}{V\arabic{enumi}}
\renewcommand{\theenumi}{V\arabic{enumi}}
\item\label{condition:Assump2:moveaway} There exist $\mu>0$ and $\delta_d \in (0, \delta_u]$ such that the inequality 
    \[\nabla_{\mathbf{x}}d_{\mathcal{U}}(\mathbf{x})^\top\boldsymbol{\nu}_d(\mathbf{x}) \geq \mu,\] 
    holds for all $\mathbf{x}\in\mathcal{N}^{\delta_d}$, where the existence of a known scalar $\delta_u > 0$ is guaranteed by Property \ref{property:distance2}, and the set $\mathcal{N}^{\delta_d}$ is obtained by replacing $\delta_u$ with $\delta_d$ in \eqref{neighborhood_definition}.

    \item \label{condition:set_of_equilibria}
    The chain recurrent set of \eqref{ideal_kinematics} is a subset of $\mathcal{E}\cup\{\mathbf{x}_d\}$, where $\mathcal{E}$ is a finite set of undesired isolated equilibrium points given by
    \[
    \mathcal{E} = \{\mathbf{x}\in\mathcal{M}\mid\boldsymbol{\nu}_d(\mathbf{x}) = \mathbf{0}_{n+1}, \mathbf{x}\ne \mathbf{x}_d\},
    \]
    with $\mathbf{x}_d\in\mathcal{M}$ being the desired equilibrium point satisfying $\boldsymbol{\nu}_d(\mathbf{x}_d) = \mathbf{0}_{n+1}$.

    \item \label{condition:Assump2:differentiability}
    The vector field $\boldsymbol{\nu}_d$ is continuously differentiable on $\mathcal{M}$.

    \item \label{condition:eigenvalues_of_equilibria}
    The Jacobian $\mathbf{J}_d(\mathbf{x}) = \frac{\partial\boldsymbol{\nu}_d(\mathbf{x})}{\partial\mathbf{x}}$ satisfies the following properties:
    \begin{enumerate}
        \item \label{condition:xd}Every eigenvalue of $\mathbf{J}_{d}(\mathbf{x}_d)$ associated with an eigenvector in the tangent space $\mathsf{T}_{\mathbf{x}_d}\mathbb{S}^n$ has negative real part. 
        \item \label{condition:undesired} For every $\mathbf{x}^*\in\mathcal{E}$, $\mathbf{J}_d(\mathbf{x}^*)$ has at least one eigenvalue with positive real part whose associated eigenvector belongs to the tangent space $\mathsf{T}_{\mathbf{x}^*}\mathbb{S}^n$. 
    \end{enumerate}

    \item\label{property:bounded_vector_field}
    There exists $D_1 > 0$ such that $\|\boldsymbol{\nu}_d(\mathbf{x})\| \leq D_1$ for all $\mathbf{x}\in\mathcal{M}$.

    % \item\label{property:bounded_jacobian}
    % The matrix $\mathbf{J}_d(\mathbf{x})$ is well-defined for all $\mathbf{x}\in\mathcal{M}$.
    % Moreover, there exists $D_2 > 0$ such that $\|\mathbf{J}_d(\mathbf{x})\|_F \leq D_2$ for all $\mathbf{x}\in\mathcal{M}\setminus\mathcal{N}^{\delta_d}$, where the set $\mathcal{N}^{\delta_d}$ is obtained by replacing $\delta_u$ with $\delta_d$ in \eqref{neighborhood_definition}.
    
    %\item\label{condition:Assump2:bounded_desired_velocity}
    %There exist $D_1 > 0$ and $D_2 > 0$ such that $\|\boldsymbol{\nu}_d(\mathbf{x})\| \leq D_1$ and $\|\mathbf{J}_d(\mathbf{x})\|_{F} \leq D_2$ for all $\mathbf{x}\in\mathcal{M}$, where the Frobenius norm operator $\|\cdot\|_{F}$ is defined in Section \ref{section:notations}.
\end{enumerate}

\begin{remark}
\label{remark:explanation_of_properties}
Property \ref{condition:Assump2:moveaway} ensures that, for the closed-loop system \eqref{ideal_kinematics}, the function $d_{\mathcal{U}}(\mathbf{x})$ strictly increases whenever $\mathbf{x}\in\mathcal{N}^{\delta_d}$.
Thus, the vector field $\boldsymbol{\nu}_d$ steers $\mathbf{x}$ away from the unsafe set $\mathcal{U}$ near its boundary.
In particular, the positive lower bound $\mu$ in this property implies that, for every initial condition $\mathbf{x}(0)\in\mathcal{M}$,
\begin{equation}
    d_{\mathcal{U}}(\mathbf{x}(t))\geq \underline{\delta}_d(\mathbf{x}(0)):= \min\{d_{\mathcal{U}}(\mathbf{x}(0)), \delta_d\}> 0, \;\forall t\geq 0.
\end{equation}
Consequently, $\mathbf{x}(t)$ remains in the set $\mathcal{K}(\mathbf{x}(0))$ for all times, 
where
\begin{equation}
    \mathcal{K}(\mathbf{x}(0)) := \{\mathbf{x}\in\mathcal{M}\mid d_{\mathcal{U}}(\mathbf{x}) \geq \underline{\delta}_d(\mathbf{x}(0))\}\subset\mathcal{M}.
\end{equation}
Since $d_{\mathcal{U}}$ is continuous on $\mathcal{M}$, $\mathcal{K}(\mathbf{x}(0))$ is a closed subset of the compact $n$-sphere $\mathbb{S}^n$ and is therefore compact in $\mathbb{S}^n$.
Consequently, for every initial condition $\mathbf{x}(0)\in\mathcal{M}$, the trajectory $\mathbf{x}(t)$ of the closed-loop system \eqref{ideal_kinematics} is precompact\footnote{A state trajectory is called precompact in its state space if its image
has compact closure in that space. Equivalently, its image is contained in
a compact subset of the state space.} in $\mathcal{M}$.

Property \ref{condition:set_of_equilibria} states that the chain recurrent set of the closed-loop system \eqref{ideal_kinematics} is subset of the finite set of isolated equilibrium points.
Loosely speaking, a point $\mathbf{x}_0\in\mathcal{M}$ is chain recurrent if one can return to $\mathbf{x}_0$ by following a finite sequence of trajectory segments connected by arbitrarily small jumps, with each segment governed by \eqref{ideal_kinematics} and lasting at least an arbitrarily prescribed time.
We refer readers to Appendix \ref{appendix:chain_recurrent_set_concept} for the formal definition and related terminology.
A vector field satisfying Property \ref{condition:set_of_equilibria} has only isolated equilibria as chain recurrent points, thereby excluding complex limiting behaviour such as nonisolated equilibrium sets, periodic orbits, homoclinic orbits, and heteroclinic cycles.

Together, the precompactness of state trajectories ensured by Property \ref{condition:Assump2:moveaway} and the characterization of the chain recurrent set in Property \ref{condition:set_of_equilibria} allow us to apply \cite[Proposition 1.2]{mischaikow1995asymptotically} to conclude global attractivity of the equilibrium set $\mathcal{E}\cup\{\mathbf{x}_d\}$ for the closed-loop system \eqref{ideal_kinematics} over $\mathcal{M}$. 
Additionally, the continuous differentiability of $\boldsymbol{\nu}_d$, stated in Property \ref{condition:Assump2:differentiability}, ensures uniqueness of solutions of \eqref{ideal_kinematics}.
Consequently, since state trajectories are continuous and $\mathcal{E}\cup\{\mathbf{x}_d\}$ is finite, global attractivity of this set implies that, for every initial condition $\mathbf{x}(0)\in\mathcal{M}$, there exists a unique equilibrium $\mathbf{x}^*\in\mathcal{E}\cup\{\mathbf{x}_d\}$ such that $\Lim_{t\to\infty}\mathbf{x}(t) = \mathbf{x}^*$.

In addition, Property \ref{condition:xd} ensures that $\mathbf{x}_d$ is asymptotically stable for the closed-loop system \eqref{ideal_kinematics} on $\mathcal{M}$.
In contrast, by Property \ref{condition:undesired} and the center manifold theorem \cite[Section 2.7, Pg 116]{perko2013differential}, every equilibrium point in $\mathcal{E}$ is unstable with a basin of attraction of measure zero on $\mathbb{S}^n$.
Together with the convergence of every solution of \eqref{ideal_kinematics} to an equilibrium and finiteness of $\mathcal{E}$, these properties establish almost global asymptotic stability of $\mathbf{x}_d$ relative to $\mathcal{M}$.
\end{remark}

Suppose we are given a separation function $d_{\mathcal{U}}$ satisfying Properties \ref{property:distance1} and \ref{property:distance2}, alongside a desired vector field $\mathbf{v}_d$ that inherits Properties \ref{condition:Assump2:moveaway}-\ref{property:bounded_vector_field} for the kinematic closed-loop system \eqref{ideal_kinematics}.
We aim to design a control law $\mathbf{u}$ for the dynamic model \eqref{dynamics_motion_model_on_sphere} such that for the resulting closed-loop system, the following objectives are satisfied:
\begin{enumerate}
    \item The set $\mathcal{M}\times\mathbb{R}^{n+1}$ is forward invariant.
    \item The desired equilibrium point $(\mathbf{x}_d, \mathbf{0}_{n+1})$ is almost globally asymptotically stable over $\mathcal{M}\times\mathbb{R}^{n+1}$.\footnote{An equilibrium point is almost globally asymptotically stable on $\mathcal{M}\times\mathbb{R}^{n+1}$ if it is Lyapunov stable and the set of initial conditions in $\mathcal{M}\times\mathbb{R}^{n+1}$ from which the solutions to the closed-loop system do not converge to the equilibrium point has zero Lebesgue measure on $\mathcal{M}\times\mathbb{R}^{n+1}$.}
\end{enumerate}

\section{Generic feedback control design}

%To address this trade-off, we propose a dynamic damping gain that adapts based on the proximity of $\mathbf{x}$ to the unsafe set $\mathcal{U}$ in the sense of the separation function $d_{\mathcal{U}}(\mathbf{x})$.

We propose the following control scheme:
\begin{equation}\label{proposed_feedback_control_input}
    \mathbf{u}(\boldsymbol{\xi}) = -k_d\beta(d_{\mathcal{U}}(\mathbf{x}))(\mathbf{v} - \boldsymbol{\nu}_d(\mathbf{x})) + \mathbf{J}_d(\mathbf{x})\mathbf{P}(\mathbf{x})\mathbf{v},
\end{equation}
where $k_d > 0$ and the composite state vector $\boldsymbol{\xi}$ is given by $\boldsymbol{\xi} := (\mathbf{x}, \mathbf{v})\in\mathcal{M}\times\mathbb{R}^{n+1}$. 
The vector $\boldsymbol{\nu}_d$ is such that Properties \ref{condition:Assump2:moveaway}-\ref{property:bounded_vector_field} hold for the closed-loop system \eqref{ideal_kinematics}. 
The matrix $\mathbf{J}_d(\mathbf{x}) = \frac{\partial\boldsymbol{\nu}_d(\mathbf{x})}{\partial\mathbf{x}}$ is defined in Property \ref{condition:eigenvalues_of_equilibria}.

The continuous, non-negative separation function $d_{\mathcal{U}}(\mathbf{x})$, which satisfies Properties \ref{property:distance1} and \ref{property:distance2}, characterizes the separation between any $\mathbf{x}\in\overline{\mathcal{M}}$ and the unsafe set $\mathcal{U}$.
Finally, the scalar function $\beta(d_{\mathcal{U}}(\cdot))$ is defined as
\begin{equation}\label{beta_function_definition}
    \beta(d_{\mathcal{U}}(\mathbf{x})) = \begin{cases}
        d_{\mathcal{U}}(\mathbf{x})^{-1}, & d_{\mathcal{U}}(\mathbf{x}) \leq \epsilon_1, \\
        \phi(d_{\mathcal{U}}(\mathbf{x})), &\epsilon_1 \leq d_{\mathcal{U}}(\mathbf{x}) \leq \epsilon_2,\\
        1, & d_{\mathcal{U}}(\mathbf{x}) \geq \epsilon_2,        
    \end{cases}
\end{equation}
where $\epsilon_1 \in (0, \delta_u)$,  $\epsilon_2 \in (\epsilon_1, \delta_u)$, and the existence of $\delta_u > 0$ is guaranteed in Property \ref{property:distance2}.
The continuously differentiable scalar mapping $\phi : [\epsilon_1, \epsilon_2]\to(0, \infty)$ satisfies $\phi(\epsilon_1) = \epsilon_1^{-1}$, $\phi(\epsilon_2) = 1$, $\phi'(\epsilon_1) = -\epsilon_1^{-2}$, $\phi'(\epsilon_2) = 0$ and $\phi(p) > 0, \;\forall p\in[\epsilon_1, \epsilon_2]$.\footnote{An example of such a function is $\phi(p) = (1 - b(s(p)))\frac{1}{p} + b(s(p))$, where $s(p) = \frac{p - \epsilon_1}{\epsilon_2 - \epsilon_1}$ and the blending function $b$ is given by $b(s) = 3s^2 - 2s^3$.}

Note that the control law \eqref{proposed_feedback_control_input} relies on a dynamic damping gain $k_d\beta(d_{\mathcal{U}}(\mathbf{x}))$ that adapts based on the proximity of $\mathbf{x}$ to the unsafe set $\mathcal{U}$ in the sense of the separation function $d_{\mathcal{U}}$. Specifically, the damping gain is low when the value of $d_{\mathcal{U}}(\mathbf{x})$ is high, and increases as $\mathbf{x}$ approaches the boundary of the unsafe set.\\
The control input is designed so that the scalar function
\begin{equation}\label{potential_function_used}
    V(\boldsymbol{\xi}) = \frac{1}{2}\|\mathbf{v}- \boldsymbol{\nu}_d(\mathbf{x})\|^2
\end{equation}
is non-increasing along the trajectories of the closed-loop system \eqref{dynamics_motion_model_on_sphere}-\eqref{proposed_feedback_control_input} when $\mathbf{x}\in\mathcal{M}$.
This suggests that, under the control input \eqref{proposed_feedback_control_input}, $\mathbf{v}(t)$ tends to align with $\boldsymbol{\nu}_d(\mathbf{x}(t))$ for all $t\geq 0$ as long as $\mathbf{x}(t)\in\mathcal{M}$.

From Property \ref{property:distance1}, $d_{\mathcal{U}}(\mathbf{x})>0, \;\forall \mathbf{x}\in\mathcal{M}$ and $d_{\mathcal{U}}(\mathbf{x}) = 0, \;\forall \mathbf{x}\in\partial\mathcal{U}$.
Hence, by \eqref{beta_function_definition}, $\beta(d_{\mathcal{U}}(\mathbf{x}))$ is positive on $\mathcal{M}$ and tends to $+\infty$ as $\mathbf{x}$ approaches $\partial\mathcal{U}$.
As a result, the magnitude of the damping term in \eqref{proposed_feedback_control_input} increases near the boundary, which, loosely speaking, enforces faster alignment of $\mathbf{v}$ with $\boldsymbol{\nu}_d(\mathbf{x})$ in the region close to $\partial\mathcal{U}$.

Moreover, by Property \ref{condition:Assump2:moveaway}, the vector $\boldsymbol{\nu}_d(\mathbf{x})$ points in a direction of increase of $d_{\mathcal{U}}(\mathbf{x})$ whenever $d_{\mathcal{U}}(\mathbf{x}) \in (0, \delta_d)$ for some $\delta_d > 0$. 
Therefore, under the control input \eqref{proposed_feedback_control_input}, there exists $T\geq 0$ such that for all $t\geq T$,
the velocity $\mathbf{v}(t)$ points in a direction along which $d_{\mathcal{U}}(\mathbf{x}(t))$ increases whenever $d_{\mathcal{U}}(\mathbf{x}(t))\in(0, \delta_d)$.
This allows us to show that if $\boldsymbol{\xi}(0)\in\mathcal{M}\times\mathbb{R}^{n+1}$, then $\boldsymbol{\xi}(t)\in\mathcal{M}\times\mathbb{R}^{n+1}$ for all $t\geq 0$, as stated in the next lemma.
\begin{lemma}\label{lemma:safety_lemma}
    Consider the closed-loop system \eqref{dynamics_motion_model_on_sphere}-\eqref{proposed_feedback_control_input} under Assumption \ref{assumption:twice_differentiability}.
    If $d_{\mathcal{U}}(\mathbf{x}(0)) > 0$, then the following statements hold:
    \begin{enumerate}
        \item \label{claim1:lemmaVTF}$d_{\mathcal{U}}(\mathbf{x}(t)) > 0$ for all $t \geq 0$.
        \item\label{claim2:lemmaVTF} There exists $t_d(\boldsymbol{\xi}(0)) \geq 0$ such that \[d_{\mathcal{U}}(\mathbf{x}(t)) \geq \delta_d \text{ for all } t\geq t_d(\boldsymbol{\xi}(0)),\] where the existence of $\delta_d > 0$ is assumed in Property \ref{condition:Assump2:moveaway}.
        \item \label{claim3:lemmaVTF} There exists a constant $D_{\mathbf{u}}(\boldsymbol{\xi}(0))> 0$ such that the control input satisfies\[\|\mathbf{u}(\boldsymbol{\xi}(t))\|\leq D_{\mathbf{u}}(\boldsymbol{\xi}(0))\;\text{for all}\;t\geq 0.\]
    \end{enumerate}
\end{lemma}
\proof{See Appendix \ref{proof:lemma:safety_lemma}.}

Although the damping gain grows unbounded as $\mathbf{x}$ tends to $\partial\mathcal{U}$, using Claims \ref{claim1:lemmaVTF} and \ref{claim2:lemmaVTF} of Lemma \ref{lemma:safety_lemma} one can show that for every initial condition $\boldsymbol{\xi}(0)\in\mathcal{M}\times\mathbb{R}^{n+1}$, there exists $\underline{\delta}(\boldsymbol{\xi}(0)) > 0$ such that the separation function satisfies $d_{\mathcal{U}}(\mathbf{x}(t))\geq \underline{\delta}(\boldsymbol{\xi}(0))$ for all $t\geq 0$.

Moreover, the scalar function $\beta(d_{\mathcal{U}}(\cdot))$ given in \eqref{beta_function_definition} is undefined if and only if $\mathbf{x}\in\partial\mathcal{U}$ and satisfies $\beta(d_{\mathcal{U}}(\mathbf{x})) > 0$ for all $\mathbf{x}\in\mathcal{M}$.
Therefore, there exists $\bar{\beta}(\boldsymbol{\xi}(0)) > 0$ such that 
\[\beta(d_{\mathcal{U}}(\mathbf{x}(t))) \leq \bar{\beta}(\boldsymbol{\xi}(0)), \;\forall t\geq 0,\]
which establishes an upper bound on the magnitude of the damping gain in \eqref{proposed_feedback_control_input}.
Combined with the boundedness of the remaining terms in \eqref{proposed_feedback_control_input}, this implies that the control input $\mathbf{u}(\boldsymbol{\xi}(t))$ is bounded for all $t\geq 0$, as stated in Claim \ref{claim3:lemmaVTF} of Lemma \ref{lemma:safety_lemma}.

By virtue of Lemma \ref{lemma:safety_lemma}, for any initial condition $\boldsymbol{\xi}(0)\in\mathcal{M}\times\mathbb{R}^{n+1}$, the scalar function $V$ defined in \eqref{potential_function_used} satisfies $\dot{V}(\boldsymbol{\xi}(t)) \leq 0$ for all $t\geq 0$.
Moreover, as stated earlier in Remark \ref{remark:explanation_of_properties}, for the system $\dot{\mathbf{x}} = \boldsymbol{\nu}_d(\mathbf{x})$, $\mathbf{x}_d$ is almost globally asymptotically stable over $\mathcal{M}$, where the desired vector field $\boldsymbol{\nu}_d$ satisfies Properties \ref{condition:Assump2:moveaway}-\ref{property:bounded_vector_field}.
This indicates that the closed-loop system \eqref{dynamics_motion_model_on_sphere}-\eqref{proposed_feedback_control_input} inherits the safety and convergence properties of the kinematic system $\dot{\mathbf{x}} = \boldsymbol{\nu}_d(\mathbf{x})$ and ensures that the desired point $(\mathbf{x}_d, \mathbf{0}_{n+1})$ is almost globally asymptotically stable over $\mathcal{M}\times\mathbb{R}^{n+1}$, as established in the next theorem.

\begin{theorem}\label{theorem:VTF}
    For the closed-loop system \eqref{dynamics_motion_model_on_sphere}-\eqref{proposed_feedback_control_input} under Assumption \ref{assumption:twice_differentiability}, the following statements hold:
    \begin{enumerate}
        \item \label{claim1:theorem} The state space $\mathcal{M}\times\mathbb{R}^{n+1}$ is forward invariant.
        \item \label{claim2:theorem} For all $t\geq 0$, $\|\mathbf{v}(t) - \boldsymbol{\nu}_d(\mathbf{x}(t))\|$ is 
        nonincreasing and converges exponentially to zero.
        %monotonically decreasing for all $t\geq 0$.
        \item \label{claim3:theorem} The set of equilibrium points is given by $\mathcal{S}\cup\{(\mathbf{x}_d, \mathbf{0}_{n+1})\}$, where the set $\mathcal{S}$ is defined as
        \begin{equation}\label{equilibrium_set_theorem}
            \mathcal{S} = \{(\mathbf{x}, \mathbf{0}_{n+1})\in\mathcal{M}\times\{\mathbf{0}_{n+1}\}\mid\mathbf{x}\in\mathcal{E}\},
        \end{equation}
        and the set $\mathcal{E}$ is defined in Property \ref{condition:set_of_equilibria}.
        \item \label{claim4:theorem} The equilibrium point $(\mathbf{x}_d, \mathbf{0}_{n+1})$ is almost globally asymptotically stable over $\mathcal{M}\times\mathbb{R}^{n+1}$.
    \end{enumerate}
\end{theorem}
\proof{See Appendix \ref{proof:theorem:VTF}.}

\section{Explicit feedback control design}\label{section:explicit_feedback_control}
The design of the vector field $\boldsymbol{\nu}_d$ is motivated by \cite[Section V]{sawant2026constrained}.
Furthermore, for any $\mathbf{x}\in\overline{\mathcal{M}}\subset\mathbb{S}^n$, we choose the spherical distance as the separation function $d_{\mathcal{U}}$, defined by
\begin{equation}\label{spherical_distance_definition}
    d_{\mathcal{U}}(\mathbf{x}) = \underset{\mathbf{a}\in\mathcal{U}}{\inf}\;\arccos(\mathbf{x}^\top\mathbf{a}),
\end{equation}
where for any $p\in[-1, 1]$, the inverse cosine function satisfies $\arccos(p)\in[0, \pi]$. 
See Appendix \ref{remark:separation_function_examples} for more details.

Before designing the desired vector field $\boldsymbol{\nu}_d$, we first elaborate on the restrictions \cite[Assumptions 1 and 2]{sawant2026constrained} imposed on the unsafe set characterization.
For each $i\in\mathbb{I}$, the constraint set $\mathcal{U}_i$ is assumed to be a star-shaped set on $\mathbb{S}^n$, where a star-shaped set on $\mathbb{S}^n$ is defined in Section \ref{section:notations}, and a few examples of such sets are provided in Fig.~\ref{fig:star_set_examples}.
We further require the constraint sets to be sufficiently separated from each other to avoid the presence of closed trajectories away from the desired point $\mathbf{x}_d$.
To expand further on this requirement, we define the following sets:

\begin{definition}\label{definition:delta_dilation}Given $\delta > 0$, the $\delta$-dilation $\mathcal{D}_i^{\delta}$ of the constraint set $\mathcal{U}_i$ on $\mathbb{S}^n$ is defined as
\begin{equation}\label{definition:dilation}
\mathcal{D}_i^{\delta} = \mathcal{U}_i\cup\mathcal{N}_i^{\delta},
\end{equation}
where the set $\mathcal{N}_i^{\delta}$ is defined analogously to \eqref{neighborhood_definition} as follows:
\begin{equation}\label{individual_neighborhood}
    \mathcal{N}_i^{\delta} = \left\{\mathbf{x}\in\mathcal{M}\mid d_{\mathcal{U}_i}(\mathbf{x})\in(0, \delta]\right\},
\end{equation}
where $d_{\mathcal{U}_i}(\mathbf{x})$ is the spherical distance between $\mathbf{x}$ and the constraint set $\mathcal{U}_i$, and is obtained by replacing $\mathcal{U}$ with $\mathcal{U}_i$ in \eqref{spherical_distance_definition}.
\end{definition}

\begin{definition}\label{definition:S_set}Given $\delta > 0$ and a constraint set $\mathcal{U}_i, i\in\mathbb{I}$, the spherical cone $\mathcal{S}_{\mathcal{D}_i^{\delta}}(\mathbf{x}_d)$ over the set $\mathcal{D}_i^{\delta}$ with respect to $\mathbf{x}_d$ is the collection of all geodesics $\mathcal{G}(\mathbf{x}, \mathbf{x}_d)$ connecting any $\mathbf{x}\in\mathcal{D}_i^{\delta}$ to $\mathbf{x}_d$, and is defined as
\begin{equation}\label{spherical_cone_over_a_set}
\mathcal{S}_{\mathcal{D}_i^{\delta}}(\mathbf{x}_d) = \bigcup_{\mathbf{x}\in\mathcal{D}_i^{\delta}}\mathcal{G}(\mathbf{x}, \mathbf{x}_d).
\end{equation}\end{definition}

\begin{figure}[t]
    \centering
    \includegraphics[width=0.7\linewidth]{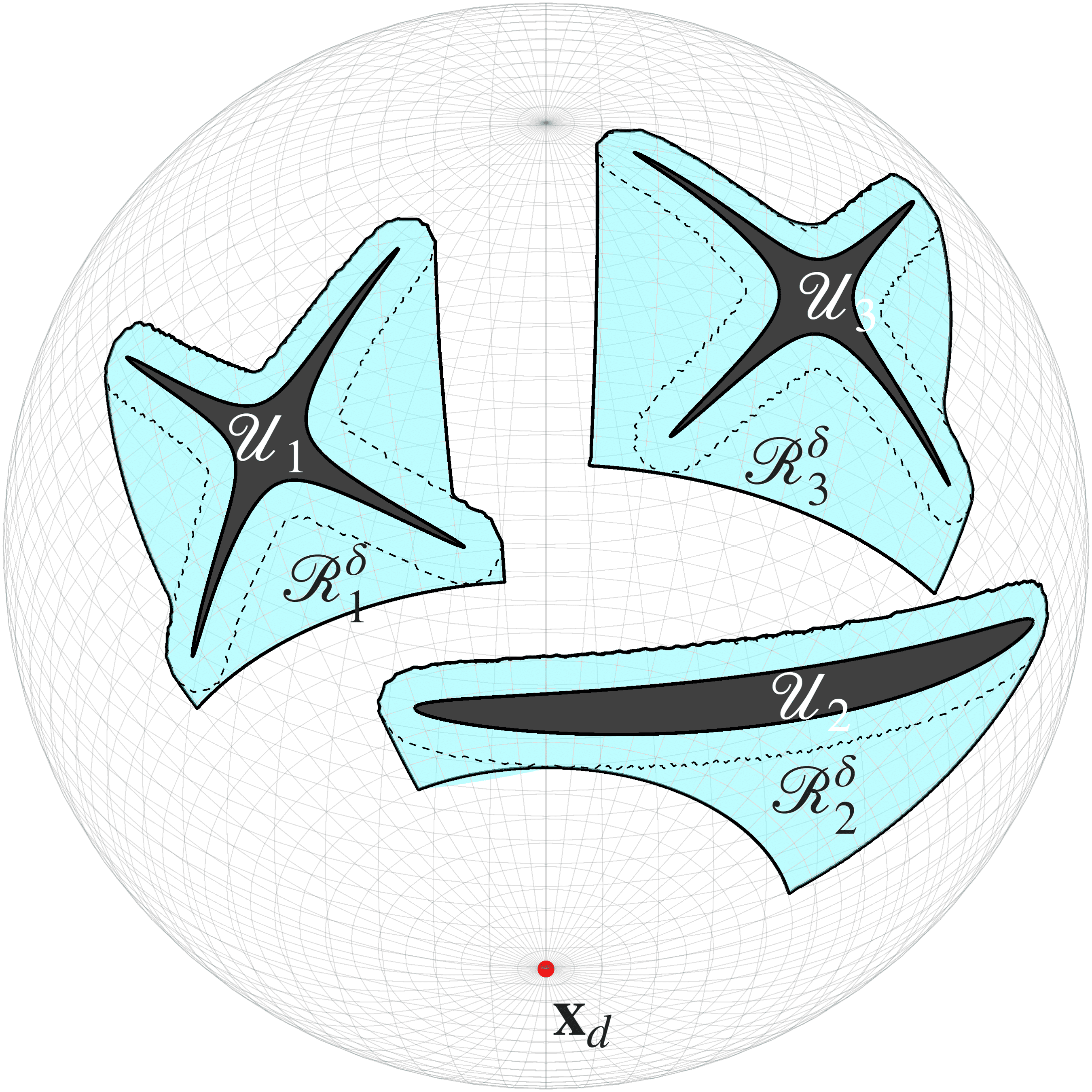}
    \caption{Example of three constraint sets satisfying Assumption \ref{assumption:sufficient_separation}. Note that the regions $\mathcal{R}_i^\delta$ do not represent unsafe regions; they remain part of the feasible state space that the system trajectories may freely enter. 
    %Their mutual separation provides a sufficient geometric condition to guarantee the absence of closed trajectories away from the target point $\mathbf{x}_d$.
    }\label{fig:sufficient_separation_representation}
\end{figure}

We assume the knowledge of $\delta > 0$ such that the target point $\mathbf{x}_d$ and its antipode $-\mathbf{x}_d$ do not belong to  $\mathcal{N}^{\delta}$,
where $\mathcal{N}^{\delta}$ is a $\delta$-neighborhood of $\mathcal{U}$, and is obtained by replacing $\delta_u$ with $\delta$ in \eqref{neighborhood_definition}.
Furthermore, for each $i\in\mathbb{I}$, we characterize the regions $\mathcal{R}_i^{\delta}$ based on the position $-\mathbf{x}_d$ as follows:
\begin{itemize}
    \item {\bf Case 1 $(-\mathbf{x}_d\notin\mathcal{U}_i)$:}
    \begin{equation}\label{definition-Ri-general}
    \mathcal{R}_i^{\delta} = \left\{\mathbf{x}\in\mathcal{S}_{\mathcal{D}_i^{\delta}}(\mathbf{x}_d)\setminus\mathcal{U}_i\mid d_{\{\mathbf{x}_d\}}(\mathbf{x})\geq d_{\mathcal{D}_i^{\delta}}(\mathbf{x}_d)\right\},
\end{equation}
where $d_{\{\mathbf{x}_d\}}(\mathbf{x})$ is the spherical distance between $\mathbf{x}$ and $\mathbf{x}_d$, and is obtained by replacing $\mathcal{U}$ with $\{\mathbf{x}_d\}$ in \eqref{spherical_distance_definition}.
Similarly, $d_{\mathcal{D}_i^{\delta}}(\mathbf{x}_d)$ is the spherical distance between $\mathbf{x}_d$ and $\mathcal{D}_i^{\delta}$, and is obtained by replacing $\mathcal{U}$ with $\mathcal{D}_i^{\delta}$ in \eqref{spherical_distance_definition}.

\item {\bf Case 2 $(-\mathbf{x}_d\in\mathcal{U}_i)$:}
\begin{equation}\label{definition-Ri-special}
    \mathcal{R}_i^{\delta} = \mathcal{S}_{\mathcal{D}_i^{\delta}}(-\mathbf{x}_d)\setminus\mathcal{U}_i,
\end{equation}
where $\mathcal{S}_{\mathcal{D}_i^{\delta}}(-\mathbf{x}_d)$ is obtained by replacing $\mathbf{x}_d$ with $-\mathbf{x}_d$ in \eqref{spherical_cone_over_a_set}.
\end{itemize}

We assume that the sets $\mathcal{R}_i^{\delta}$ and $\mathcal{R}_j^{\delta}$ are mutually exclusive for all $i, j\in\mathbb{I}$ with $i\ne j$, as illustrated in Fig.~\ref{fig:sufficient_separation_representation}.
The above-mentioned requirements are formally stated in the next assumption.
\begin{assumption}\label{assumption:sufficient_separation}
    There exists a known scalar $\delta > 0$ such that $\{\mathbf{x}_d, -\mathbf{x}_d\}\notin\mathcal{N}^{\delta}$ and $\mathcal{R}_i^{\delta}\cap\mathcal{R}_j^{\delta} = \emptyset$ for all $i, j\in\mathbb{I}$ with $i\ne j$.
\end{assumption}

\begin{remark}\label{remark:assumption_satisfaction}Since $\mathcal{D}_i^{\delta} \subset \mathcal{R}_i^{\delta}$ for each $i \in \mathbb{I}$, Assumption \ref{assumption:sufficient_separation} implies that $\mathcal{D}_i^{\delta} \cap \mathcal{D}_j^{\delta} = \emptyset$ for all $i, j \in \mathbb{I}$ with $i \ne j$. 
This disjointness directly ensures the satisfaction of \cite[Assumption 1]{sawant2026constrained}. 
Furthermore, because the construction of the regions $\mathcal{R}_i^{\delta}$ is analogous to the sets $\mathcal{R}_i$ defined in \cite{sawant2026constrained}, Assumption \ref{assumption:sufficient_separation} inherently guarantees the satisfaction of the separation condition in \cite[Assumption 2]{sawant2026constrained}.\end{remark}

\begin{remark}\label{remark:satisfaction-property-D2}
Since $\mathcal{D}_i^{\delta} \cap \mathcal{D}_j^{\delta} = \emptyset,\;\forall i, j \in \mathbb{I}, i \ne j$, using Assumption \ref{assumption:twice_differentiability}, one can show that there exists $\delta_u \in (0, \delta)$ such that, for each $i\in\mathbb{I}$ and for all $\mathbf{x} \in \mathcal{N}_i^{\delta_u}$, the minimizer \begin{equation}\label{unique_minimizer}\Pi_{\mathcal{U}_i}(\mathbf{x}) = \underset{\mathbf{a}\in\mathcal{U}_i}{\arg\min}\;\arccos(\mathbf{x}^\top\mathbf{a})\end{equation} is unique, where the set $\mathcal{N}_i^{\delta_u}$ is obtained by replacing $\delta$ with $\delta_u$ in \eqref{individual_neighborhood}.
Consequently, by virtue of Danskin's theorem \cite[Proposition B.25]{bertsekas1999nonlinear}, the spherical distance function $d_{\mathcal{U}}$ defined in \eqref{spherical_distance_definition} is continuously differentiable over $\mathcal{N}^{\delta_u} = \bigcup_{i\in\mathbb{I}}\mathcal{N}_i^{\delta_u}$ and satisfies Property \ref{property:distance2}.
\end{remark}

For each $i\in\mathbb{I}$ and for every $\mathbf{x}\in\mathcal{N}_i^{\delta_u}$, define 
\begin{equation}\label{unit_normal_vector_definition}
    \mathbf{n}_i(\mathbf{x}) = \frac{\mathbf{P}(\mathbf{x})(\mathbf{x} - \Pi_{\mathcal{U}_i}(\mathbf{x}))}{\|\mathbf{P}(\mathbf{x})(\mathbf{x} - \Pi_{\mathcal{U}_i}(\mathbf{x}))\|}
\end{equation}
as the unit vector at $\mathbf{x}$ pointing away from $\mathcal{U}_i$.
We assume that for each $i\in\mathbb{I}$, there exist $\mathbf{g}_i\in\mathcal{U}_i$ and $\delta_1^i\in(0, \delta_u]$ such that for all $\mathbf{x}\in\mathcal{N}_i^{\delta_1^i}$, the minimum non-negative angle between the unit vector $\mathbf{n}_{i}(\mathbf{x})$ and the vector $\mathbf{P}(\mathbf{x})(\mathbf{x}-\mathbf{g}_i)$ is bounded from above by an angle strictly less than $\frac{\pi}{2}$ rad.
This condition is formally stated in the following assumption:
\begin{assumption}\label{assumption:unit_normal}
    For each $i\in\mathbb{I}$, there exist $\delta_1^i\in(0, \delta_u]$, $\mu_i > 0$ and $\mathbf{g}_i\in\sigma(\mathcal{U}_i)\cap\mathcal{U}_i^{\circ}$ such that
    \[
    \mathbf{n}_i(\mathbf{x})^\top\mathbf{P}(\mathbf{x})(\mathbf{x} - \mathbf{g}_i) \geq \mu_i > 0, \;\forall\mathbf{x}\in\mathcal{N}_i^{\delta_1^i},
    \]
    where the set $\sigma(\mathcal{U}_i)$ is defined in \eqref{sigma_set} and the vector $\mathbf{n}_i(\mathbf{x})$ is defined in \eqref{unit_normal_vector_definition}.
\end{assumption}
Geometrically, Assumption \ref{assumption:unit_normal} implies that for any $i\in\mathbb{I}$ and $\mathbf{x}\in\partial\mathcal{U}_i$, the geodesic $\mathcal{G}(\mathbf{x}, \mathbf{g}_i)$ connecting $\mathbf{x}$ to $\mathbf{g}_i$ intersects the boundary $\partial\mathcal{U}_i$ only at $\mathbf{x}$.
In other words, for any $i\in\mathbb{I}$ and $\mathbf{x}\in\partial\mathcal{U}_i$, $\mathcal{G}(\mathbf{x}, \mathbf{g}_i)\cap\partial\mathcal{U}_i = \{\mathbf{x}\}$.
Assumption \ref{assumption:unit_normal} allows us to conclude that the desired vector field $\boldsymbol{\nu}_d$, defined next, satisfies Property \ref{condition:Assump2:moveaway}.

Inspired by \cite[Section V]{sawant2026constrained}, we construct the desired vector field
$\boldsymbol{\nu}_d$ as
\begin{equation}\label{nu_d}
\boldsymbol{\nu}_d(\mathbf{x}) = \mathbf{P}(\mathbf{x}) \mathbf{v}_d(\mathbf{x}), \quad \forall \mathbf{x}\in\overline{\mathcal{M}},
\end{equation}
where the vector field $\mathbf{v}_d$ is defined as
\begin{equation}\label{ideal_kinematic_planner_example}
    \mathbf{v}_d(\mathbf{x}) = \begin{cases}
        k_1\mathbf{v}_{d, i}(\mathbf{x}), & \mathbf{x}\in\overline{\mathcal{N}_i^{\epsilon}}, \\
        k_1\mathbf{x}_d, & \mathbf{x}\in\mathcal{M}\setminus\overline{\mathcal{N}^{\epsilon}},
    \end{cases}
\end{equation}
where $k_1 > 0$ and $\epsilon \in (0, \delta_u)$.
The set $\mathcal{N}_i^{\epsilon}$ is obtained by replacing $\delta$ with $\epsilon$ in \eqref{individual_neighborhood}. 
Similarly, $\mathcal{N}^{\epsilon}$ is obtained by replacing $\delta_u$ with $\epsilon$ in \eqref{neighborhood_definition}.
For each $i\in\mathbb{I}$, $\mathbf{v}_{d, i}(\mathbf{x})$ is defined as
\begin{equation}\label{ideal_kinematic_planner_individual}
    \mathbf{v}_{d, i}(\mathbf{x}) = \alpha(d_{\mathcal{U}_i}(\mathbf{x}))\mathbf{x}_d - \frac{1}{\kappa}\left(1 - \alpha(d_{\mathcal{U}_i}(\mathbf{x}))\right)\mathbf{g}_i,
\end{equation}
where $\kappa > 0$, and the spherical distance function $d_{\mathcal{U}_i}$ is obtained by replacing $\mathcal{U}$ with $\mathcal{U}_i$ in \eqref{spherical_distance_definition}.
The scalar mapping $\alpha:[0, \epsilon]\to[0, 1]$ is real analytic, and satisfies $\alpha(0) = \alpha'(0) = \alpha''(0) = 0, \alpha(\epsilon) = 1, \alpha'(\epsilon) = \alpha''(\epsilon) =0$ and $\alpha'(p) > 0, \forall p\in(0, \epsilon)$.\footnote{An example of such a function is $\alpha(p) = 6s(p)^5-15s(p)^4+10s(p)^3$, where $s(p) = \frac{p}{\epsilon}$. Since $\alpha'(p) = \frac{30}{\epsilon}s(p)^2(s(p)-1)^2 > 0$ for all $p\in(0, \epsilon)$, $\alpha(p)$ is strictly increasing over $[0, \epsilon]$.}

Since $\epsilon \in (0, \delta_u]$ and $\delta_u\in (0, \delta)$, it follows from \eqref{definition-Ri-general} and \eqref{definition-Ri-special} that $\mathcal{N}_i^{\epsilon}\subset\mathcal{R}_i^{\delta}, \forall i\in\mathbb{I}$, where the set $\mathcal{N}_i^{\epsilon}$ is obtained by replacing $\delta$ with $\epsilon$ in \eqref{individual_neighborhood}.
Therefore, it follows from Assumption \ref{assumption:sufficient_separation} that $\overline{\mathcal{N}_i^{\epsilon}}\cap\overline{\mathcal{N}_j^{\epsilon}} = \emptyset$ for all $i, j\in\mathbb{I}, i\ne j$.
Consequently, the choice of $\epsilon$ ensures that when $\mathbf{x}\in\overline{\mathcal{N}^{\epsilon}}$, there exists a unique $i\in\mathbb{I}$ such that $\mathbf{x}\in\overline{\mathcal{N}_i^{\epsilon}}$.

For the closed-loop system $\dot{\mathbf{x}} = \boldsymbol{\nu}_d(\mathbf{x})$, when $\mathbf{x}\in\mathcal{M}\setminus\overline{\mathcal{N}^{\epsilon}}$, the desired vector $\boldsymbol{\nu}_d(\mathbf{x})$ equals the attractive vector $k_1\mathbf{P}(\mathbf{x})\mathbf{x}_d$ and steers $\mathbf{x}$ along the geodesic $\mathcal{G}(\mathbf{x}, \mathbf{x}_d)$ towards the target point $\mathbf{x}_d$.
In contrast, when $\mathbf{x}\in\partial\mathcal{U}_i$ for some $i\in\mathbb{I}$, the desired vector $\boldsymbol{\nu}_d(\mathbf{x})$ equals the repulsive vector $-\frac{k_1}{\kappa}\mathbf{P}(\mathbf{x})\mathbf{g}_i$ and guides $\mathbf{x}$ along the geodesic $\mathcal{G}(\mathbf{x}, -\mathbf{g}_i)$ towards $-\mathbf{g}_i$, away from the constraint set $\mathcal{U}_i$.
Additionally, when $\mathbf{x}\in\mathcal{N}_i^{\epsilon}$ for some $i\in\mathbb{I}$, the desired vector $\boldsymbol{\nu}_d(\mathbf{x})$ is a convex combination of the attractive vector and the repulsive vector.

\begin{remark}\label{remark:continuously_differentiable_vector_field_planner}
Since $\mathbf{v}_d$ in \eqref{ideal_kinematic_planner_example} is constant for all $\mathbf{x}\in\mathcal{M}\setminus\overline{\mathcal{N}^{\epsilon}}$, $\boldsymbol{\nu}_d=\mathbf{P}(\mathbf{x})\mathbf{v}_d(\mathbf{x})$ is twice continuously differentiable over $\mathcal{M}\setminus\overline{\mathcal{N}^{\epsilon}}$.
Since $\epsilon \in (0, \delta_u]$ and $\delta_u \in (0, \delta)$, by Assumption \ref{assumption:sufficient_separation}, $\mathcal{N}_i^{\epsilon}\cap\mathcal{N}_j^{\epsilon} = \emptyset$ for all $i, j\in\mathbb{I}, i\ne j$.
Therefore, using Assumption \ref{assumption:twice_differentiability}, one can verify that $d_{\mathcal{U}}$ defined in \eqref{spherical_distance_definition} is twice continuously differentiable over $\mathcal{N}^{\epsilon}$.
Consequently, $\boldsymbol{\nu}_d$ is twice continuously differentiable for all $\mathbf{x}\in\left(\mathcal{N}^{\epsilon}\right)^{\circ}$.
Additionally, the scalar function $\alpha(d_{\mathcal{U}_i}(\cdot))$ used in \eqref{ideal_kinematic_planner_individual} is constructed to be twice continuously differentiable at points where $d_{\mathcal{U}_i}(\mathbf{x}) = \epsilon$. 
Therefore, the desired vector field $\boldsymbol{\nu}_d$ defined in \eqref{ideal_kinematic_planner_example} is twice continuously differentiable on $\mathcal{M}$.

The matrix $\mathbf{J}_d(\mathbf{x}) = \frac{\partial\boldsymbol{\nu}_d(\mathbf{x})}{\partial\mathbf{x}}$, which is used in \eqref{proposed_feedback_control_input}, is given by
\begin{equation*}
    \mathbf{J}_d(\mathbf{x}) = \mathbf{P}(\mathbf{x})\frac{\partial\mathbf{v}_d(\mathbf{x})}{\partial\mathbf{x}} -\mathbf{x}\mathbf{v}_d(\mathbf{x})^\top  - \mathbf{x}^\top\mathbf{v}_d(\mathbf{x})\mathbf{I}_{n+1},
\end{equation*}
with
\begin{equation*}
    \frac{\partial\mathbf{v}_d(\mathbf{x})}{\partial\mathbf{x}} = \begin{cases}
        \frac{-k_1\alpha'(d_{\mathcal{U}_i})}{\sin(d_{\mathcal{U}_i})}\left(\mathbf{x}_d + \frac{\mathbf{g}_i}{\kappa}\right)\Pi_{\mathcal{U}_i}(\mathbf{x})^\top, & \mathbf{x}\in\mathcal{N}_i^{\epsilon},\\
        \mathbf{O}_{n+1}, & \mathbf{x}\in\mathcal{M}\setminus\mathcal{N}^{\epsilon},
    \end{cases}
\end{equation*}
where $\Pi_{\mathcal{U}_i}(\mathbf{x})$ is the unique closest point to $\mathbf{x}$ on $\mathcal{U}_i$ defined in \eqref{unique_minimizer}.
\end{remark}

Following an argument similar to that in the proof of \cite[Theorem 2]{sawant2026constrained}, one can show that, under Assumption \ref{assumption:sufficient_separation}, $\overline{\mathcal{M}}$ is forward invariant for the closed-loop system \eqref{ideal_kinematics}-\eqref{ideal_kinematic_planner_example}.
Moreover, there exists $\bar{\kappa} > 0$ such that, for any $\kappa > \bar{\kappa}$, the desired point $\mathbf{x}_d\in\mathcal{M}$ is almost globally asymptotically stable over $\overline{\mathcal{M}}$.
Under Assumptions \ref{assumption:twice_differentiability}, \ref{assumption:sufficient_separation}, and \ref{assumption:unit_normal}, one can verify that the vector field $\boldsymbol{\nu}_d$ satisfies Properties \ref{condition:Assump2:moveaway}-\ref{property:bounded_vector_field}, as stated in the next lemma.

\begin{lemma}\label{lemma:planner-satisfies-properties}
    For the closed-loop system \eqref{ideal_kinematics}-\eqref{nu_d} under Assumptions \ref{assumption:twice_differentiability}, \ref{assumption:sufficient_separation}, and \ref{assumption:unit_normal}, there exists $\bar{\kappa} > 0$ such that, for any $\kappa > \bar{\kappa}$, Properties \ref{condition:Assump2:moveaway}-\ref{property:bounded_vector_field} are satisfied.
\end{lemma}
\proof{See Appendix \ref{proof:lemma:planner-satisfies-properties}.}
%We conjecture that Property \ref{condition:set_of_equilibria} also holds. 
%This conjecture is supported by numerical simulations, but a proof remains open.

\section{Application to constrained attitude stabilization}\label{section:application}
\label{section:full-attitude}

\subsection{Reduced attitude control using the rotation matrix}\label{section:reduced-attitude}
Let $\mathbf{x} = \mathbf{R}^\top\mathbf{e}_3\in\mathbb{S}^2$ be the pointing direction of a rigid body, corresponding to the inertial direction $\mathbf{e}_3 = [0, 0, 1]^\top$ in the body frame.
The matrix $\mathbf{R}\in\mathrm{SO}(3)$ represents the orientation of the rigid body attached frame with respect to the inertial frame.
The control objective is to align the pointing direction $\mathbf{x}$ with the desired direction $\mathbf{x}_d\in\mathbb{S}^2$ while avoiding unsafe directions represented as constraint sets on $\mathbb{S}^2$.
The rigid body dynamics are given by 
\begin{equation}\label{rigid_body_dynamics_SO3}
    \begin{aligned}
        \dot{\mathbf{R}} &= \mathbf{R}\boldsymbol{\omega}^{\times},\\
        \mathbf{J}_{m}\dot{\boldsymbol{\omega}}&=-\boldsymbol{\omega}\times\mathbf{J}_{m}\boldsymbol{\omega} + \boldsymbol{\tau},
    \end{aligned}
\end{equation}
where $\boldsymbol{\omega}\in\mathbb{R}^3$ is the angular velocity of the rigid body in the body frame and $\boldsymbol{\omega}^{\times}\in\mathbb{R}^{3\times 3}$ is a skew symmetric matrix associated with $\boldsymbol{\omega}$ such that $\boldsymbol{\omega}^{\times}\mathbf{y} = \boldsymbol{\omega}\times\mathbf{y}$ for any $\boldsymbol{\omega}, \mathbf{y}\in\mathbb{R}^3$ with $\times$ being the vector cross product.
The matrix $\mathbf{J}_m\in\mathbb{R}^{3\times3}$ is a constant positive definite inertia matrix, and $\boldsymbol{\tau}\in\mathbb{R}^3$ is the torque control input.

Taking the time derivative of $\mathbf{x} = \mathbf{R}^\top\mathbf{e}_3$ and using \eqref{rigid_body_dynamics_SO3}, one obtains the reduced attitude dynamics as follows:
\begin{equation}\label{reduced_rigid_body_dynamics_SO3}
    \begin{aligned}
        \dot{\mathbf{x}} &= \mathbf{x}^{\times}\boldsymbol{\omega},\\
        \mathbf{J}_{m}\dot{\boldsymbol{\omega}}&=-\boldsymbol{\omega}\times\mathbf{J}_{m}\boldsymbol{\omega} + \boldsymbol{\tau}.
    \end{aligned}
\end{equation}

Defining $\mathbf{v}:= \mathbf{x}^{\times}\boldsymbol{\omega}$ and using the fact that $\mathbf{P}(\mathbf{x})\mathbf{x}^{\times}\boldsymbol{\omega}=\mathbf{x}^{\times}\boldsymbol{\omega}$, one has $\dot{\mathbf{x}}=\mathbf{P}(\mathbf{x})\mathbf{v}$. Differentiating $\mathbf{v}$ with respect to time, one obtains
\begin{equation}\label{vdot_expression}
\begin{array}{rcl}
    \dot{\mathbf{v}}&=&  - \|\boldsymbol{\omega}\|^2\mathbf{x} + \mathbf{x}^\top\boldsymbol{\omega}\boldsymbol{\omega}+ \mathbf{x}^{\times}\dot{\boldsymbol{\omega}}\\
    %~&=&(\boldsymbol{\omega}^{\times})^2 \mathbf{x}+\mathbf{x}^{\times}\dot{\boldsymbol{\omega}}\\
     ~&=&(\boldsymbol{\omega}^{\times})^2 \mathbf{x}+\mathbf{x}^{\times}(-\mathbf{J}_m^{-1}(\boldsymbol{\omega}\times\mathbf{J}_{m}\boldsymbol{\omega}) + \mathbf{J}_m^{-1}\boldsymbol{\tau}).\\
    \end{array}
\end{equation}
Taking $\boldsymbol{\tau}=\boldsymbol{\omega} \times \mathbf{J}_m \boldsymbol{\omega}+ \mathbf{J}_m \bar{\boldsymbol{\tau}}$, one gets
\begin{equation}\label{vdot_expression2}
    \dot{\mathbf{v}}=(\boldsymbol{\omega}^{\times})^2 \mathbf{x}+\mathbf{x}^{\times}\bar{\boldsymbol{\tau}}.
\end{equation}
Consequently, system \eqref{reduced_rigid_body_dynamics_SO3} can be rewritten as \eqref{dynamics_motion_model_on_sphere} by defining  $\mathbf{v} := \mathbf{x}^{\times}\boldsymbol{\omega}$ and
\begin{equation}
\mathbf{u}:=(\boldsymbol{\omega}^{\times})^2 \mathbf{x}+\mathbf{x}^{\times}\bar{\boldsymbol{\tau}}.
\end{equation}
From the last equation, one has
\begin{equation}\label{expression1}
\mathbf{x}^\times \bar{\boldsymbol{\tau}}= \mathbf{u}-(\boldsymbol{\omega}^{\times})^2 \mathbf{x},
\end{equation}
where $\mathbf{u}$ is given in \eqref{proposed_feedback_control_input}. Now let us rewrite $\bar{\boldsymbol{\tau}}$ as
\begin{equation}
\bar{\boldsymbol{\tau}}=\bar{\boldsymbol{\tau}}^\bot+\bar{\boldsymbol{\tau}}^\parallel,
\end{equation}
where $\bar{\boldsymbol{\tau}}^\bot=\mathbf{P}(\mathbf{x})\bar{\boldsymbol{\tau}}$ and $\bar{\boldsymbol{\tau}}^\parallel= \mathbf{x}\mathbf{x}^\top \bar{\boldsymbol{\tau}}$.
Multiplying \eqref{expression1} by $-\mathbf{x}^{\times}$ on both sides and using the fact $-(\mathbf{x}^{\times})^2 = \mathbf{P}(\mathbf{x})$, one gets
\begin{equation}\label{compare1}
 \begin{array}{rcl}
    \bar{\boldsymbol{\tau}}^\bot &=& \mathbf{x}^{\times}(\boldsymbol{\omega}^{\times})^2\mathbf{x} + \mathbf{u}_r\\
    ~&=&\mathbf{x}^{\times}(\boldsymbol{\omega} \boldsymbol{\omega}^\top)\mathbf{x} + \mathbf{u}_r
 \end{array}   
\end{equation}
where $\mathbf{u}_r = -\mathbf{x}\times\mathbf{u}$, which can be rewritten using  \eqref{proposed_feedback_control_input} as
\begin{equation*}%\label{control_u_RSO3}
    \mathbf{u}_r = -k_d\beta(d_{\mathcal{U}}(\mathbf{x}))(\mathbf{P}(\mathbf{x})\boldsymbol{\omega} +\mathbf{x}^{\times}\boldsymbol{\nu}_d(\mathbf{x})) - \mathbf{x}^{\times}\mathbf{J}_d(\mathbf{x})\mathbf{x}^{\times}\boldsymbol{\omega},
\end{equation*}
where the scalar function $\beta$ is defined in \eqref{beta_function_definition}, $\boldsymbol{\nu}_d(\mathbf{x})$ is defined in \eqref{nu_d} and $\mathbf{J}_d(\mathbf{x})=\frac{\partial\boldsymbol{\nu}_d}{\partial\mathbf{x}}$.\\
Now, let us proceed with the design of $\bar{\boldsymbol{\tau}}^\parallel$. Let $\boldsymbol{\omega} = \boldsymbol{\omega}^{\perp} + \boldsymbol{\omega}^{\parallel}$, with $\boldsymbol{\omega}^{\perp}=\mathbf{P}(\mathbf{x})\boldsymbol{\omega}$ and $\boldsymbol{\omega}^{\parallel}=\mathbf{x}\mathbf{x}^\top \boldsymbol{\omega}$. Under the control $\bar{\boldsymbol{\tau}}^\bot$, the orthogonal component $\boldsymbol{\omega}^{\perp}$ tends to zero as $t$ tends to infinity. For the sake of simplicity\footnote{ Note that $\bar{\boldsymbol{\tau}}^\parallel$ can also be designed to guarantee that $\boldsymbol{\omega}^\parallel$ tends to some desired value $\boldsymbol{\omega}_d^\parallel$ other than zero.}, we will design $\bar{\boldsymbol{\tau}}^\parallel$ to guarantee that $\boldsymbol{\omega}^\parallel$, whose dynamics is given by
\begin{equation}\label{omega_parallel_dynamics}
\begin{array}{rcl}
\dot{\boldsymbol{\omega}}^\parallel&=&-\boldsymbol{\omega}^\times \mathbf{x} \mathbf{x}^\top \boldsymbol{\omega} +\mathbf{x}\mathbf{x}^\top \bar{\boldsymbol{\tau}}^\parallel\\
~&=& (\boldsymbol{\omega}^\parallel)^\times\boldsymbol{\omega}^\bot+\mathbf{x}\mathbf{x}^\top \bar{\boldsymbol{\tau}}^\parallel
\end{array}
\end{equation}
tends to zero as $t$ tends to infinity. 
Consider $V_{\omega} = \frac{1}{2} \|\boldsymbol{\omega}^{\parallel}\|^2$.
Differentiating $V_{\omega}$ with respect to time, along the dynamics of \eqref{omega_parallel_dynamics}, one obtains
\begin{equation*}
    \dot{V}_{\omega} = (\boldsymbol{\omega}^{\parallel})^\top\mathbf{x}\mathbf{x}^\top\bar{\boldsymbol{\tau}}^{\parallel}.
\end{equation*}
Taking $\bar{\boldsymbol{\tau}}^{\parallel} = -\gamma\boldsymbol{\omega}^{\parallel}$ with $\gamma > 0$ yields 
\begin{equation*}
\dot{V}_{\omega} = -\gamma(\mathbf{x}^\top\boldsymbol{\omega}^{\parallel})^2 = -\gamma\|\boldsymbol{\omega}^{\parallel}\|^2 = -2\gamma V_{\omega}\leq 0,
\end{equation*}
Consequently, one has
\begin{equation*}
    0 \leq V_{\omega}(t) = V_{\omega}(0)e^{-2\gamma t}, \forall t \geq 0.
\end{equation*}
It follows that $\|\boldsymbol{\omega}^{\parallel}(t)\| = \|\boldsymbol{\omega}^{\parallel}(0)\|e^{-\gamma t},\,\forall t\geq 0$, which implies that $\|\boldsymbol{\omega}^{\parallel}\|$ converges exponentially to zero.
% Consequently, one has
% \begin{equation*}
%     0 \leq V_{\omega}(t) = V_{\omega}(0)e^{-2\gamma t}, \forall t \geq 0.
% \end{equation*}
% which implies that the equilibrium $\boldsymbol{\omega}^{\parallel}=\mathbf{0}_3$ is stable for the closed-loop system \eqref{omega_parallel_dynamics} with $\bar{\boldsymbol{\tau}}^{\parallel} = -\gamma\boldsymbol{\omega}^{\parallel}$. Since $V_{\omega}$ is lower bounded and $\dot{V}_{\omega} \leq 0$, it follows that $\Lim_{t\to\infty}V_{\omega}$ exists.
% Furthermore, it can be verified that $\ddot{V}_{\omega}$ is bounded, which implies that $\dot{V}_{\omega}$ is uniformly continuous.
% Therefore, by virtue of Barbalat's lemma, $\Lim_{t\to\infty}\dot{V}_{\omega} = 0$, which ensures that $\Lim_{t\to\infty}\boldsymbol{\omega}^{\parallel} = \mathbf{0}_3$.
Finally, the control torque is given by
\begin{equation}\label{torque_control_RSO3}
\boldsymbol{\tau}=\boldsymbol{\omega} \times \mathbf{J}_m \boldsymbol{\omega}+ \mathbf{J}_m \left( \mathbf{x}^{\times}(\boldsymbol{\omega} \boldsymbol{\omega}^\top)\mathbf{x} + \mathbf{u}_r - \gamma\mathbf{x}\mathbf{x}^\top \boldsymbol{\omega}\right).
\end{equation}

\begin{proposition}\label{prop:SO3_system}
Consider system \eqref{reduced_rigid_body_dynamics_SO3} with the control input \eqref{torque_control_RSO3}. 
Under Assumptions \ref{assumption:twice_differentiability}, \ref{assumption:sufficient_separation} and \ref{assumption:unit_normal}, the following statements hold:
    \begin{enumerate}
        \item \label{claim1:propSO3} The set $\mathcal{M} \times \mathbb{R}^3=\{(\mathbf{x}, \boldsymbol{\omega})~|~\mathbf{x}\in \mathcal{M}\subset\mathbb{S}^2, \boldsymbol{\omega}\in\mathbb{R}^3\}$ is forward invariant.
        %\item \label{claim2:propSO3} $\|\mathbf{P}(\mathbf{x}(t))(\boldsymbol{\omega}(t) - \boldsymbol{\omega}_r(\mathbf{x}(t)))\|$ is monotonically decreasing for all $t\geq 0$, where $\boldsymbol{\omega}_r$ is defined in \eqref{reduced_desired_angular_velocity}.
        %\item \label{claim3:propSO3} The equilibrium set is given by $(\{\mathbf{x}_d\}\cup\mathcal{E})\times\mathcal{S}_w$, where $\mathcal{E}$ is defined in Property \ref{condition:set_of_equilibria} and $\mathcal{S}_w$ is given in \eqref{set_S_omega}
        \item \label{claim3:propSO3} The set of equilibrium points is given by $\mathcal{S}_r = \{(\mathbf{x}, \boldsymbol{\omega})\in\mathcal{M}\times\mathbb{R}^3\mid\boldsymbol{\nu}_d(\mathbf{x}) = \mathbf{0}_3, \boldsymbol{\omega} = \mathbf{0}_3\}$, where $\boldsymbol{\nu}_d$ is defined in \eqref{nu_d}.
        \item \label{claim4:propSO3} The equilibrium point $(\mathbf{x}, \boldsymbol{\omega})=(\mathbf{x}_d, \mathbf{0}_3)$ is almost globally asymptotically stable over $\mathcal{M} \times \mathbb{R}^3$.
        %\item If $\gamma$ in \eqref{reduced_torque_control_inputs} is designed as 
        %\begin{equation*}
        %    \gamma(\mathbf{x}, \boldsymbol{\omega}) = - k_{\gamma}(\mathbf{x}^\top\boldsymbol{\omega} - \omega_{\mathrm{ref}}),
        %\end{equation*}
        %where $k_{\gamma} > 0$ and $\omega_{\mathrm{ref}}\in\mathbb{R}$, then the equilibrium point $(\mathbf{x}_d, \omega_{\mathrm{ref}}\mathbf{x}_d)$ is %almost globally asymptotically stable over $\mathcal{M}\times\mathbb{R}^3$.
    \end{enumerate}
\end{proposition}
\proof{Under the transformation $\mathbf{v} = \mathbf{x}^{\times}\boldsymbol{\omega}$, the closed-loop system \eqref{reduced_rigid_body_dynamics_SO3}-\eqref{torque_control_RSO3} can be transformed into the closed-loop system \eqref{dynamics_motion_model_on_sphere}-\eqref{proposed_feedback_control_input}.
Therefore, the proof of Proposition \ref{prop:SO3_system} follows from Theorem \ref{theorem:VTF} and the developments preceding this proposition.}

\subsection{Full attitude control using the unit-quaternion}

The attitude of a rigid body with respect to the inertial frame can be described by a four-parameter representation, namely a unit quaternion.
To denote the unit-quaternion, we use $\mathbf{x} = [\eta, \mathbf{q}^\top]^\top\in\mathbb{S}^3$, where $\eta\in\mathbb{R}$ and $\mathbf{q}\in\mathbb{R}^3$.
The unit-quaternion-based dynamic attitude model is given by 
\begin{equation}\label{quaternion_dynamics}
    \begin{aligned}
    \dot{\mathbf{x}} &= \frac{1}{2}\mathbf{A}(\mathbf{x})\boldsymbol{\omega},\\
    \mathbf{J}_m \dot{\boldsymbol{\omega}}&=-\boldsymbol{\omega} \times \mathbf{J}_m \boldsymbol{\omega} + \boldsymbol{\tau},
    \end{aligned}
\end{equation}
where $\boldsymbol{\omega}\in\mathbb{R}^3$ is the angular velocity and $\mathbf{J}_m\in\mathbb{R}^{3\times 3}$ is a constant positive definite inertia matrix.
The matrix $\mathbf{A}(\mathbf{x})$ is defined as
\begin{equation}
    \mathbf{A}(\mathbf{x}) = \begin{bmatrix}-\mathbf{q}^\top\\\eta\mathbf{I}_3 + \mathbf{q}^{\times}\end{bmatrix},
\end{equation}
where $\mathbf{q}^{\times}\in\mathbb{R}^{3\times 3}$ is a skew symmetric matrix such that $\mathbf{q}^{\times}\mathbf{p} = \mathbf{q}\times\mathbf{p}$ for any $\mathbf{p}\in\mathbb{R}^3$ with $\times$ being the vector cross product.
%Given any initial condition $(\mathbf{x}(0), \boldsymbol{\omega}(0))\in\mathcal{M}\times\mathbb{R}^3$, the objective is to design the torque control input $\boldsymbol{\tau}\in\mathbb{R}^3$ in \eqref{quaternion_dynamics} that ensures $\mathbf{x}(t)\in\mathcal{M}, \;\forall t\geq 0$, and renders $(\mathbf{x}_d, \mathbf{0}_3)$ almost globally asymptotically stable over $\mathcal{M}\times\mathbb{R}^3$.
 Defining  $\mathbf{v} := \frac{1}{2}\mathbf{A}(\mathbf{x})\boldsymbol{\omega}$, one has 
\begin{equation}
\begin{array}{rcl}
\dot{\mathbf{v}}&=& \frac{1}{2}\dot{\mathbf{A}}(\mathbf{x})\boldsymbol{\omega}+\frac{1}{2}\mathbf{A}(\mathbf{x})\dot{\boldsymbol{\omega}}\\
~&=&-\frac{1}{4}\|\boldsymbol{\omega}\|^2\mathbf{x}+\frac{1}{2}\mathbf{A}(\mathbf{x})\dot{\boldsymbol{\omega}}\\
~&=&-\frac{1}{4}\|\boldsymbol{\omega}\|^2\mathbf{x}+\frac{1}{2}\mathbf{A}(\mathbf{x})(-\mathbf{J}_m^{-1}(\boldsymbol{\omega}\times\mathbf{J}_{m}\boldsymbol{\omega}) + \mathbf{J}_m^{-1}\boldsymbol{\tau}).
\end{array}
\end{equation}
Taking $\boldsymbol{\tau}=\boldsymbol{\omega} \times \mathbf{J}_m \boldsymbol{\omega}+ \mathbf{J}_m \bar{\boldsymbol{\tau}}$, one gets
\begin{equation}
\begin{array}{rcl}
\dot{\mathbf{v}}&=& -\frac{1}{4}\|\boldsymbol{\omega}\|^2\mathbf{x}+\frac{1}{2}\mathbf{A}(\mathbf{x})\bar{\boldsymbol{\tau}}.
\end{array}
\end{equation}
Defining $\mathbf{u}:=-\frac{1}{4}\|\boldsymbol{\omega}\|^2\mathbf{x}+\frac{1}{2}\mathbf{A}(\mathbf{x})\bar{\boldsymbol{\tau}}$, one gets $\dot{\mathbf{v}}=\mathbf{u}$. Using the facts $\mathbf{A}(\mathbf{x})^\top\mathbf{A}(\mathbf{x}) = \mathbf{I}_3$ and $\mathbf{P}(\mathbf{x}) = \mathbf{A}(\mathbf{x})\mathbf{A}(\mathbf{x})^\top$, it is clear that $\dot{\mathbf{x}}=\mathbf{P}(\mathbf{x})\mathbf{v}$. Therefore system \eqref{quaternion_dynamics} can be rewritten as \eqref{dynamics_motion_model_on_sphere} with $\mathbf{x}$, $\mathbf{v}$ and $\mathbf{u}$ as defined above.\\
Now, one can design $\mathbf{A}(\mathbf{x}) \bar{\boldsymbol{\tau}}$ as follows:
\[
\mathbf{A}(\mathbf{x}) \bar{\boldsymbol{\tau}}=\frac{1}{2}\|\boldsymbol{\omega}\|^2\mathbf{x}+2 \mathbf{u}.
\]
Multiplying the previous equality by $\mathbf{A}(\mathbf{x})^\top$ and using the facts $\mathbf{A}(\mathbf{x})^\top\mathbf{A}(\mathbf{x}) = \mathbf{I}_3$ and $\mathbf{A}(\mathbf{x})^\top \mathbf{x}=\mathbf{0}_3$, one has $\bar{\boldsymbol{\tau}}= \mathbf{u}_f$, where $\mathbf{u}_f:=2\mathbf{A}(\mathbf{x})^\top \mathbf{u}$, which can be obtained, using \eqref{proposed_feedback_control_input}, as follows:
{\small\begin{equation*}
    \mathbf{u}_{f} =  -k_d\beta(d_{\mathcal{U}}(\mathbf{x}))(\boldsymbol{\omega} - 2\mathbf{A}(\mathbf{x})^\top\boldsymbol{\nu}_d(\mathbf{x})) + \mathbf{A}(\mathbf{x})^\top\mathbf{J}_d(\mathbf{x})\mathbf{A}(\mathbf{x})\boldsymbol{\omega},
\end{equation*}}
where the scalar function $\beta$ is defined in \eqref{beta_function_definition}, $\boldsymbol{\nu}_d(\mathbf{x})$ is defined in \eqref{nu_d} and $\mathbf{J}_d(\mathbf{x}) = \frac{\partial\boldsymbol{\nu}_d}{\partial\mathbf{x}}$.
%and the desired angular velocity $\boldsymbol{\omega}_f(\mathbf{x})$ is given by
%\begin{equation}\label{desired_angular_velocity}
%\boldsymbol{\omega}_{f}(\mathbf{x}) = 2\mathbf{A}(\mathbf{x})^\top\boldsymbol{\nu}_d(\mathbf{x}),
%\end{equation}
%and $\boldsymbol{\nu}_d(\mathbf{x})$ is defined in \eqref{nu_d}.
Finally, the control torque is given by:
\begin{equation}\label{torque_control_input_quaternion}
    \boldsymbol{\tau} = \boldsymbol{\omega}\times\mathbf{J}_m\boldsymbol{\omega} + \mathbf{J}_{m}\mathbf{u}_{f}.
\end{equation}

\begin{proposition}\label{prop:quaternion_system}
For the closed-loop system \eqref{quaternion_dynamics}-\eqref{torque_control_input_quaternion} under Assumptions \ref{assumption:twice_differentiability}, \ref{assumption:sufficient_separation} and \ref{assumption:unit_normal}, the following statements hold:
    \begin{enumerate}
        \item \label{claim1:prop} The set $\mathcal{M} \times \mathbb{R}^3=\{(\mathbf{x}, \boldsymbol{\omega})~|~\mathbf{x}\in \mathcal{M}\subset\mathbb{S}^3, \boldsymbol{\omega}\in\mathbb{R}^3\}$ is forward invariant.
        %\item \label{claim2:prop} $\|\boldsymbol{\omega}(t) - \boldsymbol{\omega}_f(\mathbf{x}(t))\|$ is monotonically decreasing for all $t\geq 0$, where $\boldsymbol{\omega}_f$ is defined in \eqref{desired_angular_velocity}.
        \item \label{claim3:prop} The set of equilibrium points is given by $\mathcal{S}_f = \{(\mathbf{x}, \boldsymbol{\omega})\in\mathcal{M}\times\mathbb{R}^3\mid\boldsymbol{\nu}_d(\mathbf{x}) = \mathbf{0}_4, \boldsymbol{\omega} = \mathbf{0}_3\}$, where $\boldsymbol{\nu}_d$ is defined in \eqref{nu_d}.
        \item \label{claim4:prop} The equilibrium point $(\mathbf{x}_d, \mathbf{0}_3)$ is almost globally asymptotically stable over $\mathcal{M} \times \mathbb{R}^3$.
    \end{enumerate}
\end{proposition}
\proof{Under the transformation $\mathbf{v} = \frac{1}{2}\mathbf{A}(\mathbf{x})\boldsymbol{\omega}$, the closed-loop system \eqref{quaternion_dynamics}-\eqref{torque_control_input_quaternion} can be transformed into the closed-loop system \eqref{dynamics_motion_model_on_sphere}-\eqref{proposed_feedback_control_input}.
%Furthermore, using this transformation one gets $\|\mathbf{v} - \boldsymbol{\nu}_d(\mathbf{x})\| = \frac{1}{2}\|\boldsymbol{\omega} - \boldsymbol{\omega}_f(\mathbf{x})\|$.
Therefore, the proof of Proposition \ref{prop:quaternion_system} follows from Theorem \ref{theorem:VTF}.}

\section{Simulation results}

\subsection{Constrained stabilization on \texorpdfstring{$2$}{}-sphere}
We consider the $2$-sphere with six star-shaped obstacles, as shown in Fig.~\ref{fig:x-trajectories}.
The desired vector field $\boldsymbol{\nu}_d$ is given in \eqref{nu_d}, with parameters $k_1 = 1$, $\kappa = 1$, and $\epsilon = 0.13$ rad.
The separation function $d_{\mathcal{U}}$ is defined in \eqref{spherical_distance_definition}, and the parameters $\epsilon_1$ and $\epsilon_2$ in \eqref{beta_function_definition} are set to $0.087$ rad and $0.13$ rad, respectively.
The gain $k_d$ in \eqref{proposed_feedback_control_input} is set to $1$.
The location of the constant unit vectors $\mathbf{g}_i$ used in \eqref{ideal_kinematic_planner_individual} is denoted using magenta dots.

The $\mathbf{x}$-trajectories are initialized at ten distinct points in the free space $\mathcal{M}$, indicated by diamond markers.
For each initial condition $\mathbf{x}(0)$, the initial velocity $\mathbf{v}(0)$ is chosen as a unit vector directed toward the closest point on $\mathcal{U}$ in the sense of the spherical distance \eqref{spherical_distance_definition}.
%, denoted by $\Pi_{\mathcal{U}}(\mathbf{x}(0))$, where $\Pi_{\mathcal{U}}$ is defined analogously to \eqref{unique_minimizer} with $\mathcal{U}_i$ replaced by $\mathcal{U}$.

The $\mathbf{x}$-trajectories converge to the desired point $\mathbf{x}_d$ while remaining in the free space for all time $t\geq 0$, as shown in Fig.~\ref{fig:x-trajectories}.
Additionally, Fig.~\ref{fig:z_norm} shows that the norm $\|\mathbf{v} - \boldsymbol{\nu}_d(\mathbf{x})\|$ decreases monotonically along system trajectories, as stated in Claim \ref{claim2:theorem} of Theorem \ref{theorem:VTF}.
Furthermore, Fig.~\ref{fig:u_norm} illustrates that the control input remains bounded for all time $t\geq 0$, as established in Claim \ref{claim3:lemmaVTF} of Lemma \ref{lemma:safety_lemma}.

\begin{figure}[ht]
    \centering

    \begin{subfigure}{\linewidth}
        \centering
        \includegraphics[width=0.698\linewidth]{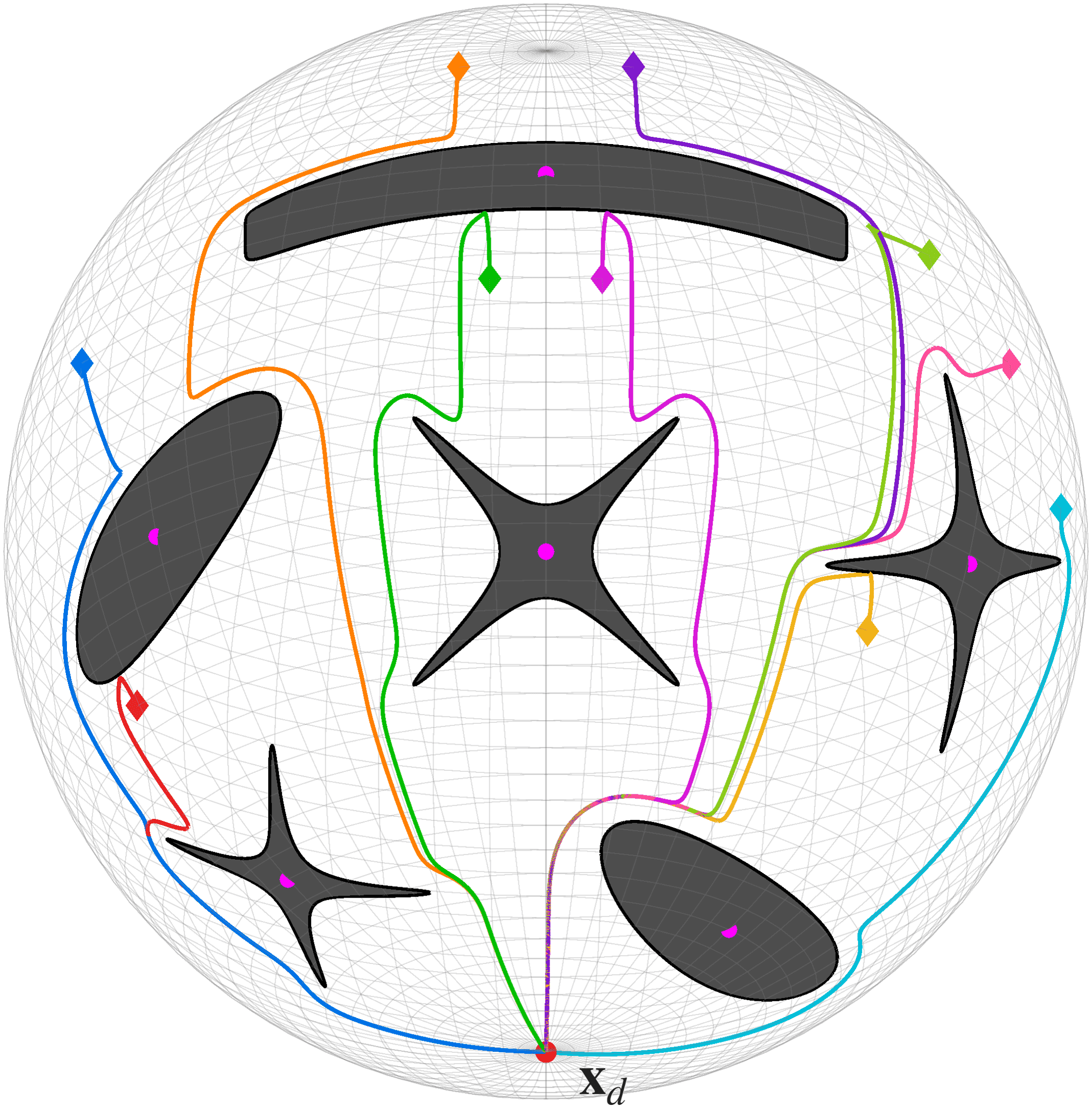}
        \caption{}
        \label{fig:x-trajectories}
    \end{subfigure}

    \vspace{0em}

    \begin{subfigure}{\linewidth}
        \centering
        \includegraphics[width=\linewidth]{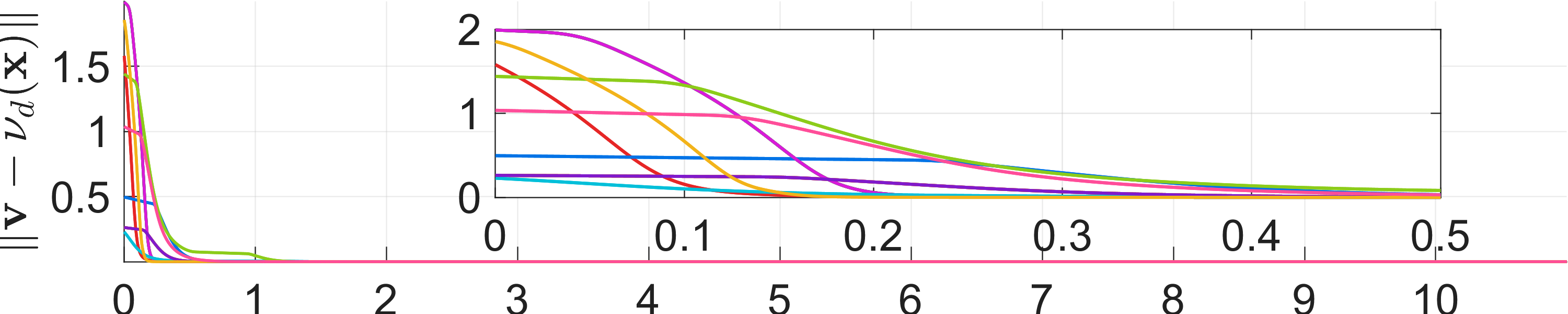}
        \caption{}
        \label{fig:z_norm}
    \end{subfigure}

    \vspace{0em}

    \begin{subfigure}{\linewidth}
        \centering
        \includegraphics[width=\linewidth]{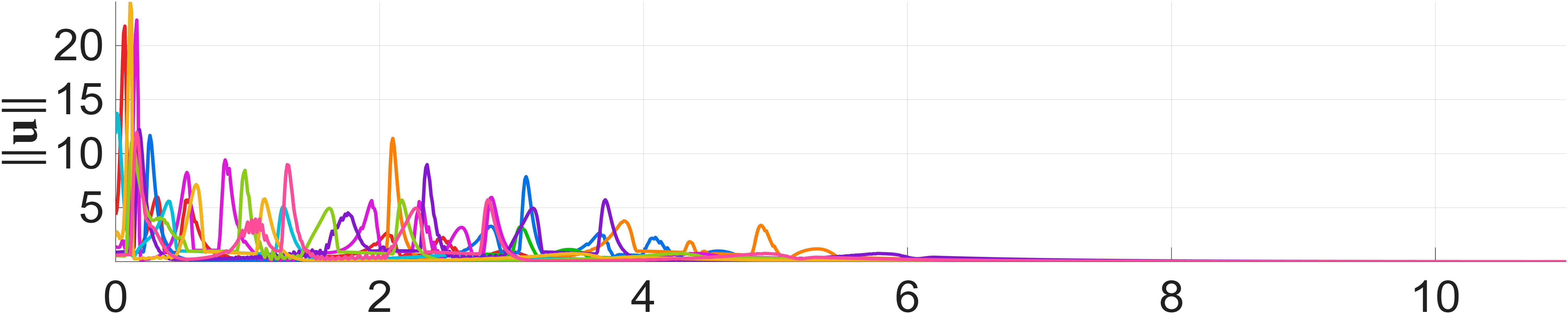}
        \caption{}
        \label{fig:u_norm}
    \end{subfigure}

    \vspace{0em}

    \begin{subfigure}{\linewidth}
        \centering
        \includegraphics[width=\linewidth]{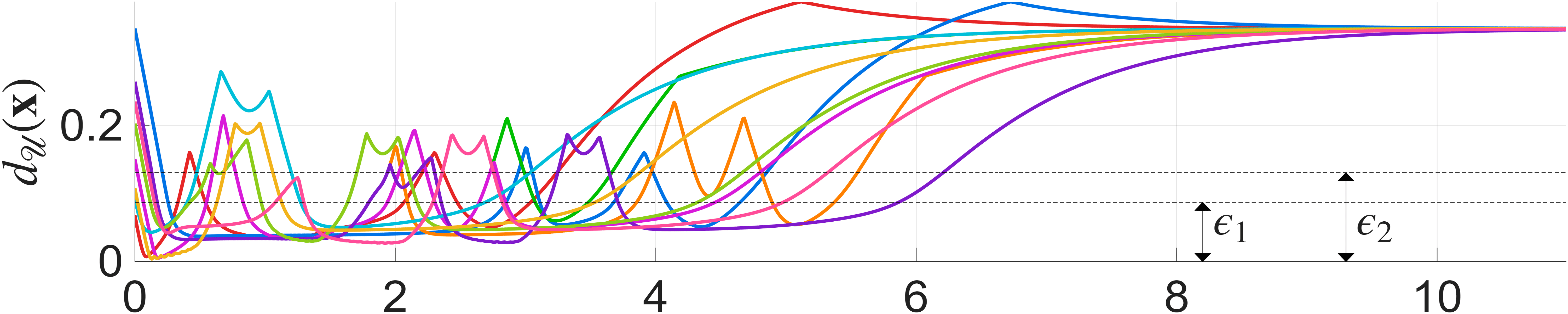}
        \caption{}
        \label{fig:d_profile}
    \end{subfigure}

    \caption{Simulation of the closed-loop system \eqref{dynamics_motion_model_on_sphere}-\eqref{proposed_feedback_control_input} with $\boldsymbol{\nu}_d$ defined in \eqref{nu_d}. (a) $\mathbf{x}$-trajectories safely converging to $\mathbf{x}_d$, (b) $\|\mathbf{v} - \boldsymbol{\nu}_d(\mathbf{x})\|$ versus time, (c) $\|\mathbf{u}\|$ versus time, (d) $d_{\mathcal{U}}(\mathbf{x})$ versus time.}
    \label{main_result}
\end{figure}

\subsection{Reduced attitude constrained stabilization}
The closed-loop system \eqref{reduced_rigid_body_dynamics_SO3}-\eqref{torque_control_RSO3} is simulated with the desired vector field $\boldsymbol{\nu}_d$ defined in \eqref{nu_d}.
We consider the $2$-sphere with a single star-shaped obstacle, as shown in Fig.~\ref{RSO3:x-trajectories}.
The controller gains $k_1$, $\kappa$ used in \eqref{ideal_kinematic_planner_example} and $k_d$ used in \eqref{proposed_feedback_control_input} are all set to $1$.
The parameters in \eqref{beta_function_definition} are chosen as $\epsilon_1 = 0.2$ rad and $\epsilon_2=0.4$ rad.
The scalar $\epsilon$ in \eqref{ideal_kinematic_planner_example} is set to $0.4$ rad.
The inertia matrix is set to $\mathbf{J}_m = \mathrm{diag}(0.01, 0.01, 0.002)$ kg$\cdot$m$^2$.

The system is simulated for two cases: $\gamma = 0$ (green trajectories) and $\gamma = 1$ (blue trajectories).
The initial angular velocity $\boldsymbol{\omega}(0)$ is set to $\mathbf{x}(0)$ rad/s. 
Therefore, in both cases, $\|\boldsymbol{\omega}^{\perp}(0)\| = 0$ rad/s and $\|\boldsymbol{\omega}^{\parallel}(0)\| = 1$ rad/s, where $\boldsymbol{\omega}^{\perp}$ and $\boldsymbol{\omega}^{\parallel}$ are defined in Section \ref{section:reduced-attitude}.

Since $\gamma$ only affects the torque control input \eqref{torque_control_RSO3} that controls the rotation of the rigid body about $\mathbf{x}$, the pointing direction $\mathbf{x}$ is safely steered to $\mathbf{x}_d$ irrespective of the value of $\gamma$, as illustrated in Fig.~\ref{RSO3:x-trajectories}.
When $\gamma = 0$, the rigid body continues to rotate about $\mathbf{x}$ with its initial angular velocity $\boldsymbol{\omega}(0)$.
In contrast, when $\gamma=1$, the magnitude of the angular velocity component $\boldsymbol{\omega}^{\parallel}$ parallel to $\mathbf{x}$ asymptotically converges to $0$, as shown in Fig.~\ref{RSO3:w_parallel}.

\begin{figure}[ht]
    \centering

    \begin{subfigure}{\linewidth}
        \centering
        \includegraphics[width=0.7\linewidth]{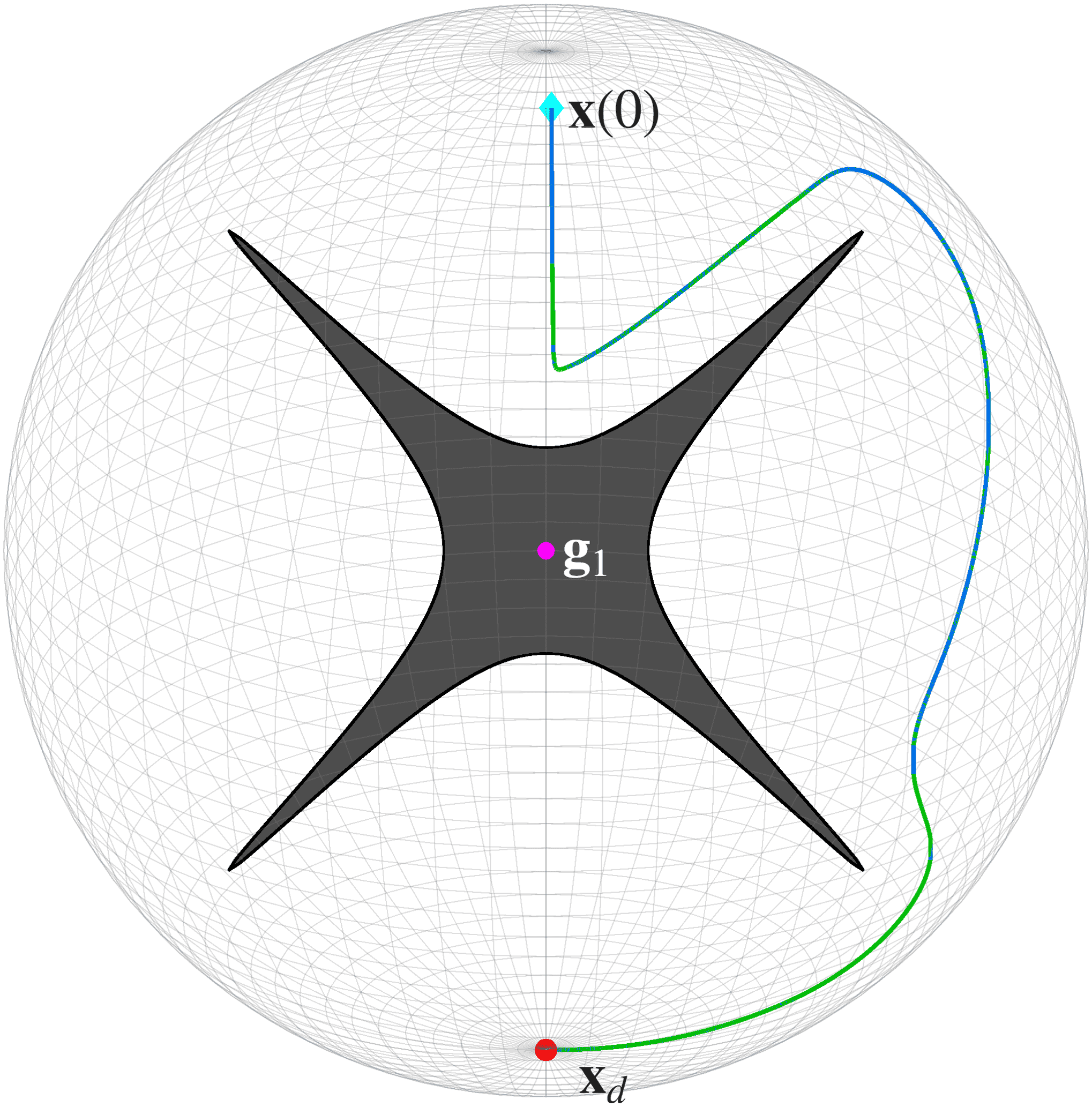}
        \caption{}
        \label{RSO3:x-trajectories}
    \end{subfigure}

    % \vspace{0em}

    % \begin{subfigure}{\linewidth}
    %     \centering
    %     \includegraphics[width=\linewidth]{Images/RSO3_images/monotonic_decrease_RSO3.png}
    %     \caption{}
    %     \label{RSO3:z_norm}
    % \end{subfigure}

    % \vspace{0em}

    % \begin{subfigure}{\linewidth}
    %     \centering
    %     \includegraphics[width=\linewidth]{Images/RSO3_images/distance_profile_RSO3.png}
    %     \caption{}
    %     \label{RSO3:d_profile}
    % \end{subfigure}

    \vspace{0em}

    \begin{subfigure}{\linewidth}
        \centering
        \includegraphics[width=\linewidth]{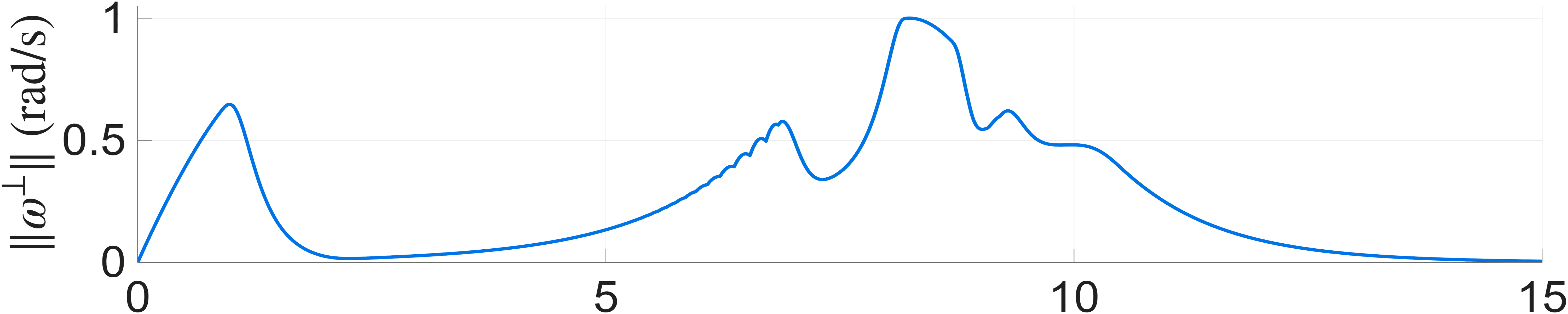}
        \caption{}
        \label{RSO3:w_perp}
    \end{subfigure}
    
    \vspace{0em}

    \begin{subfigure}{\linewidth}
        \centering
        \includegraphics[width=\linewidth]{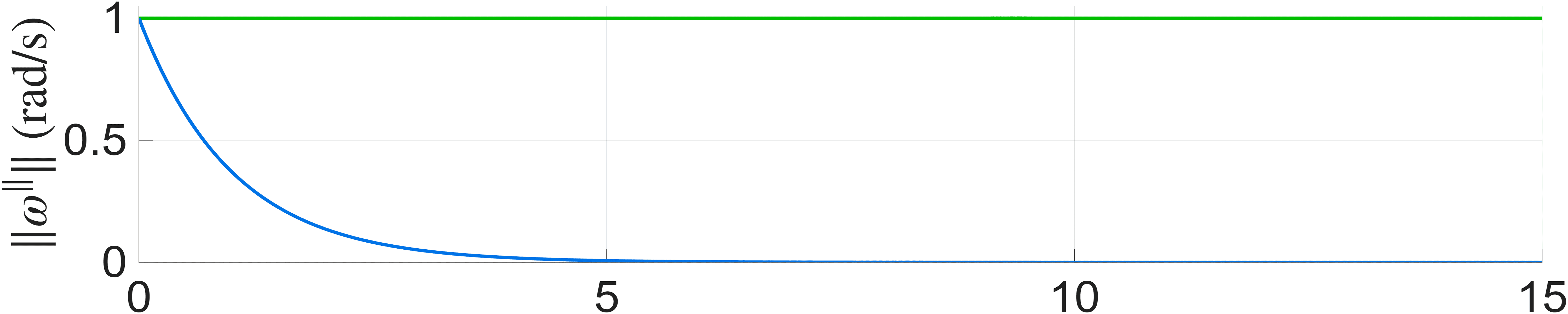}
        \caption{}
        \label{RSO3:w_parallel}
    \end{subfigure}

    \vspace{0em}

    \begin{subfigure}{\linewidth}
        \centering
        \includegraphics[width=\linewidth]{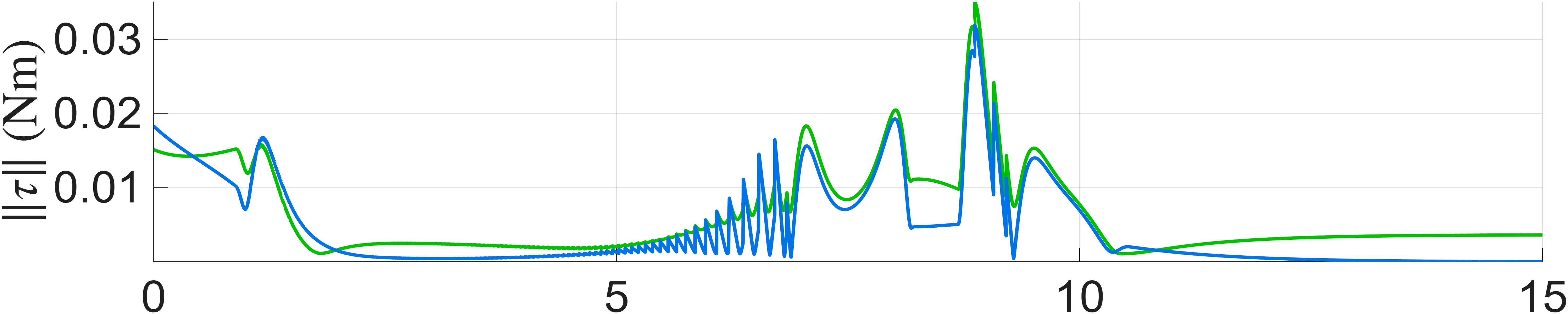}
        \caption{}
        \label{RSO3:u_norm}
    \end{subfigure}

    \caption{Simulation of the closed-loop system \eqref{reduced_rigid_body_dynamics_SO3}-\eqref{torque_control_RSO3} with $\boldsymbol{\nu}_d$ defined in \eqref{nu_d} for two different values of $\gamma$. Green trajectories are obtained with $\gamma = 0$, and the blue trajectories are obtained with $\gamma = 1$. (a) $\mathbf{x}$-trajectories safely converging to $\mathbf{x}_d$, %(b) $\|\mathbf{P}(\mathbf{x})(\boldsymbol{\omega} - \boldsymbol{\omega}_r)\|$ versus time, 
    (b) $\boldsymbol{\omega}^{\perp}$ versus time, (c) $\boldsymbol{\omega}^{\parallel}$ versus time, (d) $\|\boldsymbol{\tau}\|$ versus time.} 
    \label{RSO3_main_result}
\end{figure}

\subsection{Full attitude constrained stabilization}

The closed-loop system \eqref{quaternion_dynamics}-\eqref{torque_control_input_quaternion} is simulated with the desired vector field $\boldsymbol{\nu}_d$ defined in \eqref{nu_d}.
The $3$-sphere consists of a single star-shaped constraint $\mathcal{U}_1$.
The set $\mathcal{U}_1$ is constructed using $\mathcal{O}_1$ shown in Fig.~\ref{Qua:3D-star-obstacle} as follows:
\begin{equation*}
    \mathcal{U}_1 = \left\{\mathbf{x}\in\mathbb{S}^3\;\Bigg|\;\mathbf{x} = \frac{\mathbf{p} + \alpha\mathbf{g}_1}{\|\mathbf{p} + \alpha\mathbf{g}_1\|}, \mathbf{p}\in\mathcal{O}_1\right\}, 
\end{equation*}
where $\alpha = 2.5$ and $\mathbf{g}_1 = [0.5, 0.5, 0.5, 0.5]^{\top}$.
The boundary of the set $\mathcal{O}_1$ is a collection of all points $\mathbf{p} = [p_1, p_2, p_3, p_4]^\top\in\mathbb{R}^4$ such that 
\begin{equation*}
    (p_1^2 + p_2^2 + p_3^2)^3 + 3000(p_1^2p_2^2 + p_2^2p_3^2 + p_3^2p_1^2) - 100 = 0,
\end{equation*}
with $p_4 = 0$.

The controller gains $k_1$, $\kappa$ used in \eqref{ideal_kinematic_planner_example} are all set to $1$.
The parameter $k_d$ used in \eqref{proposed_feedback_control_input} is set to $0.5$.
The parameters in \eqref{beta_function_definition} are chosen as $\epsilon_1 = 0.25$ rad and $\epsilon_2=0.5$ rad.
The scalar $\epsilon$ in \eqref{ideal_kinematic_planner_example} is set to $0.5$ rad.
The inertia matrix is set to $\mathbf{J}_m = \mathrm{diag}(0.01, 0.01, 0.002)$ kg$\cdot$m$^2$. 
The $\mathbf{x}$-trajectories are initialized at 10 different initial conditions in $\mathcal{M}$, and the initial angular velocities $\boldsymbol{\omega}(0)$ are set to $\mathbf{0}_3$ rad/s.
The desired point is chosen as $\mathbf{x}_d = [1,0,0,0]^\top$.

The torque control input \eqref{torque_control_input_quaternion} steers all $\mathbf{x}$-trajectories to $\mathbf{x}_d$, as shown in Figs.~\ref{Qua:eta}-\ref{Qua:q3}.
Additionally, $\mathbf{x}(t)$ evolves in the free space $\mathcal{M}$ for all $t\geq 0$ \textit{i.e.}, $d_{\mathcal{U}}(\mathbf{x}(t)) > 0, \;\forall t\geq 0$ as illustrated in Fig.~\ref{Qua:du}.
Moreover, the torque control input guarantees that $\Lim_{t\to\infty}\|\boldsymbol{\omega}(t)\| = 0$, as shown in Fig.~\ref{Qua:w}.

\begin{figure*}[ht]
    \centering
    \begin{minipage}[t]{0.32\textwidth}
        \vspace{0pt} 
        \begin{subfigure}{\linewidth}
            \centering
            \includegraphics[width=0.76\linewidth]{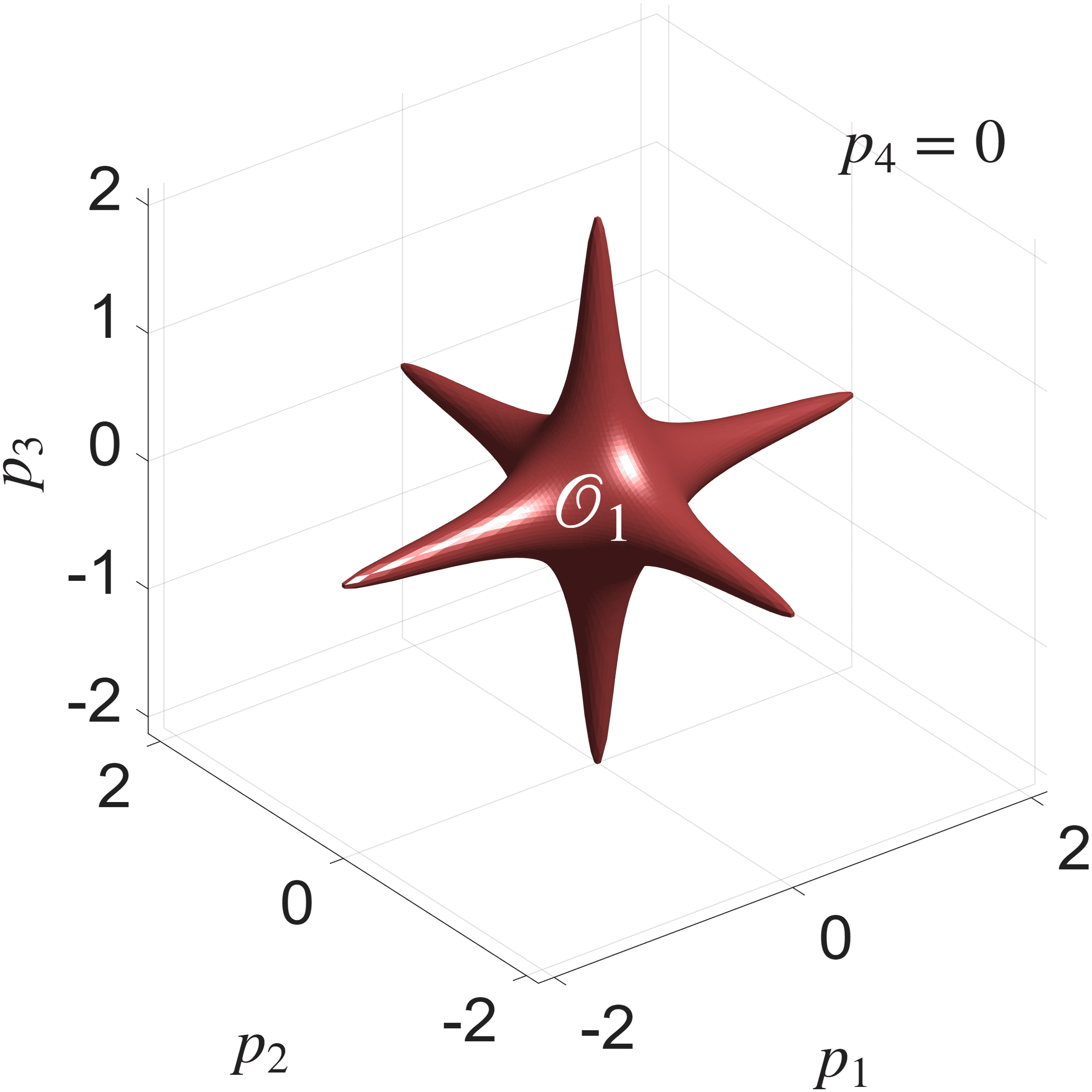}
            \caption{}
            \label{Qua:3D-star-obstacle}
        \end{subfigure}
        
        %\vspace{1em} 

        \begin{subfigure}{\linewidth}
            \centering
            \includegraphics[width=\linewidth]{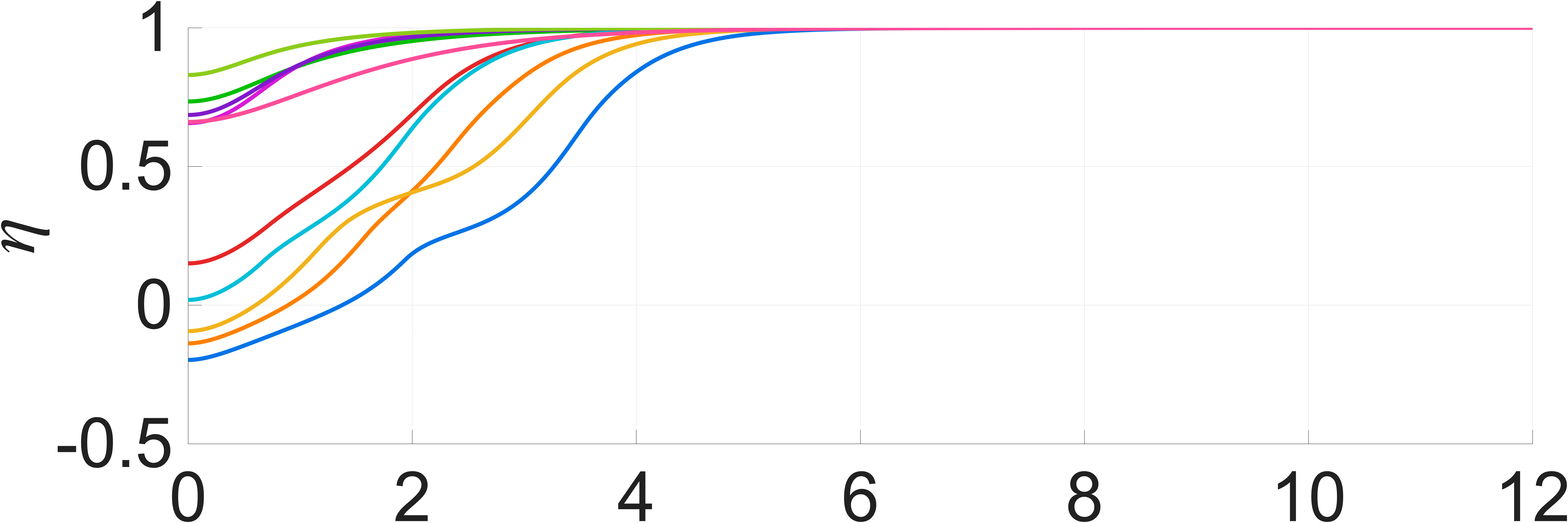}
            \caption{}
            \label{Qua:eta}
        \end{subfigure}
    \end{minipage}\hfill
    \begin{minipage}[t]{0.32\textwidth}
        \vspace{0pt} 
        \begin{subfigure}{\linewidth}
            \centering
            \includegraphics[width=\linewidth]{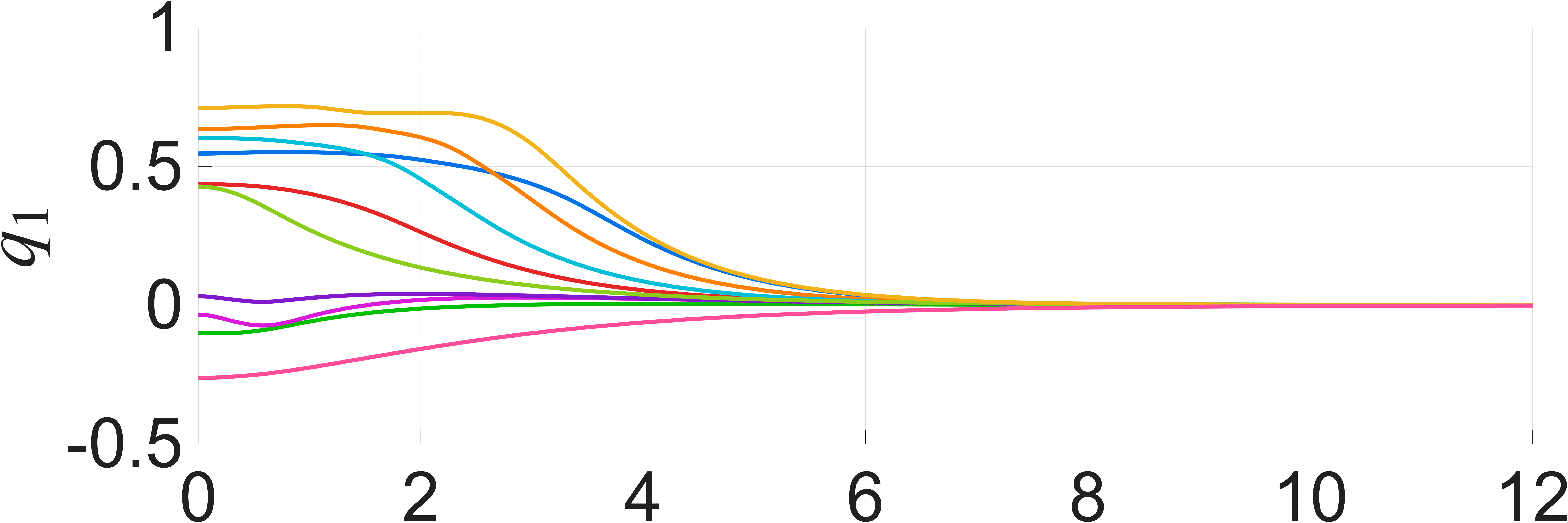}
            \caption{}
            \label{Qua:q1}
        \end{subfigure}
        
        %\vspace{1em}

        \begin{subfigure}{\linewidth}
            \centering
            \includegraphics[width=\linewidth]{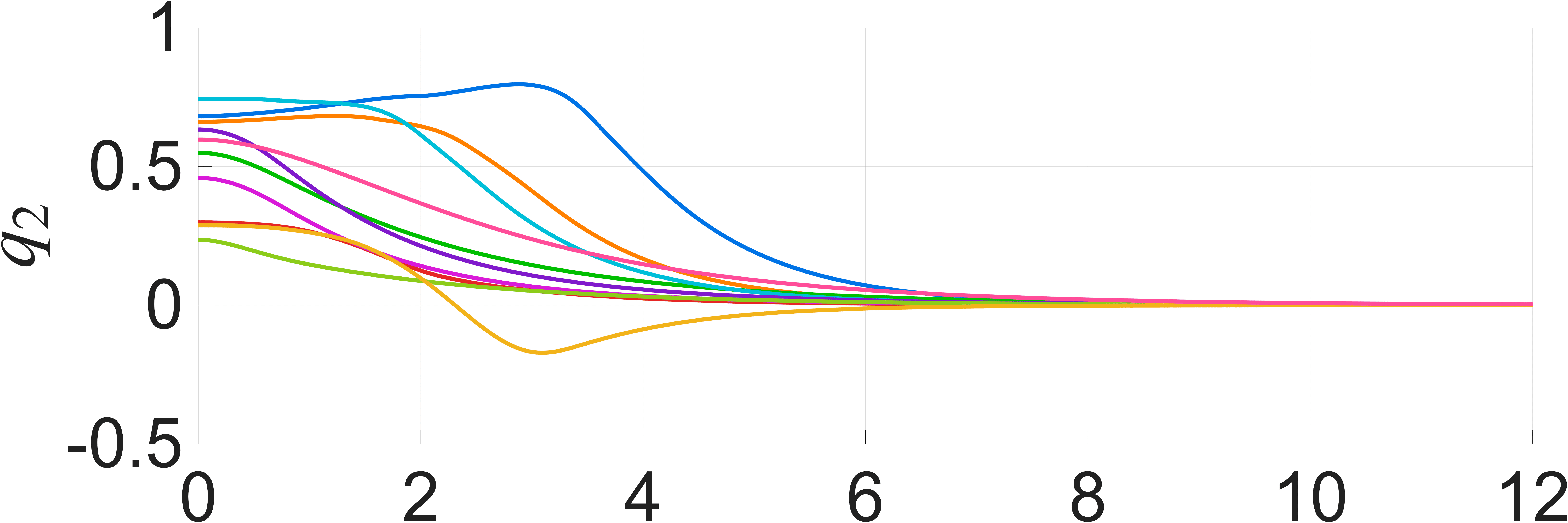}
            \caption{}
            \label{Qua:q2}
        \end{subfigure}
        
        %\vspace{1em}

        \begin{subfigure}{\linewidth}
            \centering
            \includegraphics[width=\linewidth]{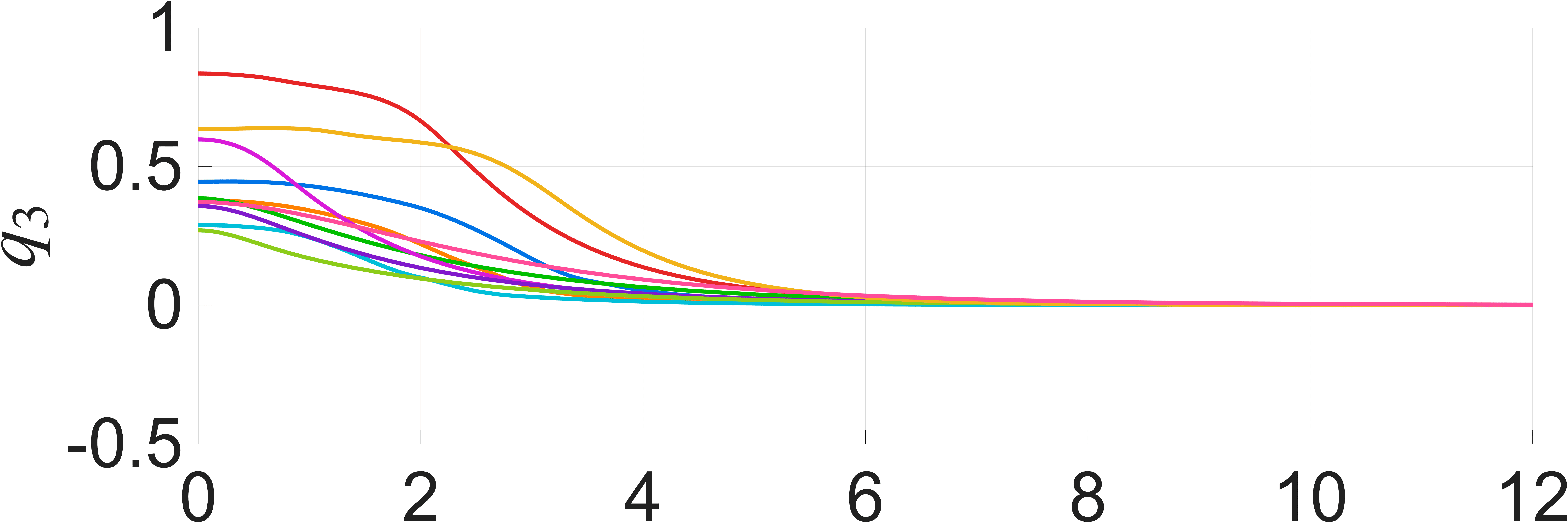}
            \caption{}
            \label{Qua:q3}
        \end{subfigure}
    \end{minipage}\hfill
    \begin{minipage}[t]{0.32\textwidth}
        \vspace{0pt} 
        \begin{subfigure}{\linewidth}
            \centering
            \includegraphics[width=\linewidth]{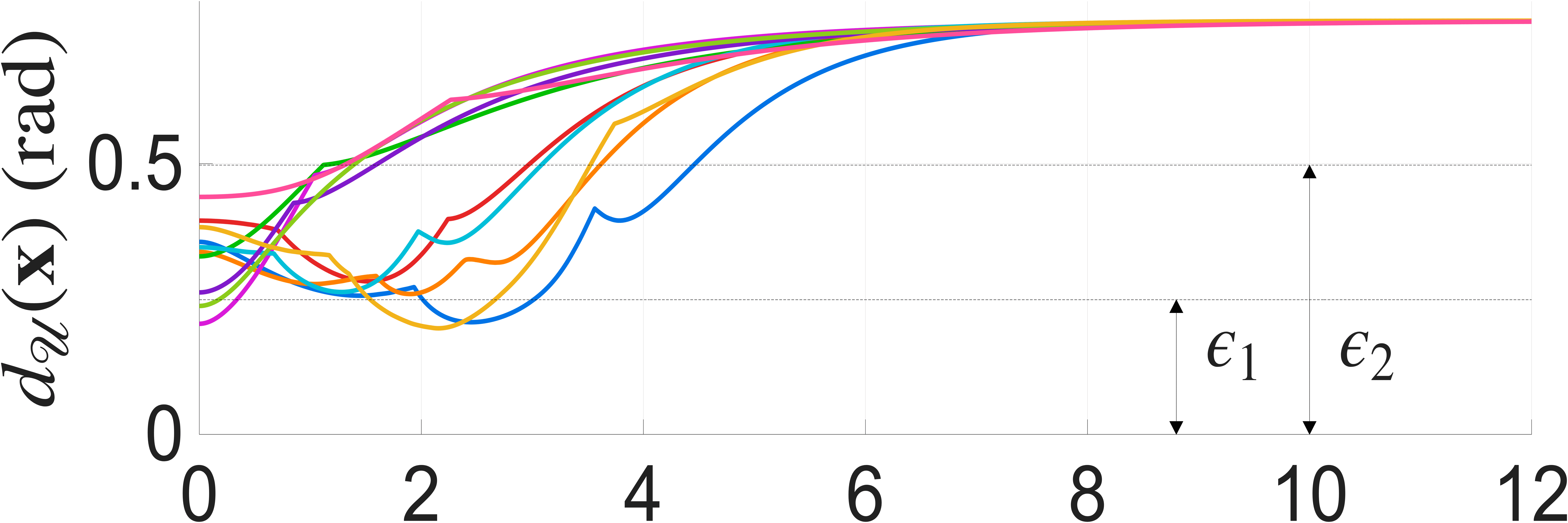}
            \caption{}
            \label{Qua:du}
        \end{subfigure}
               
        %\vspace{1em}

        \begin{subfigure}{\linewidth}
            \centering
            \includegraphics[width=\linewidth]{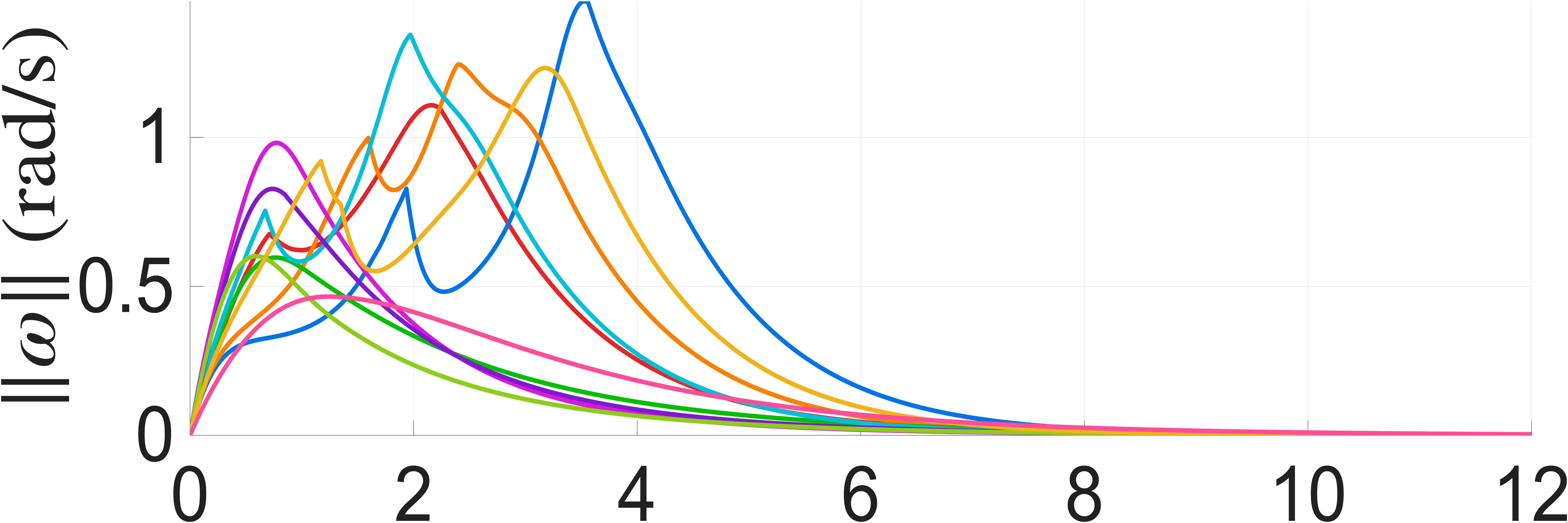}
            \caption{}
            \label{Qua:w}
        \end{subfigure}
        
        %\vspace{1em}

        \begin{subfigure}{\linewidth}
            \centering
            \includegraphics[width=\linewidth]{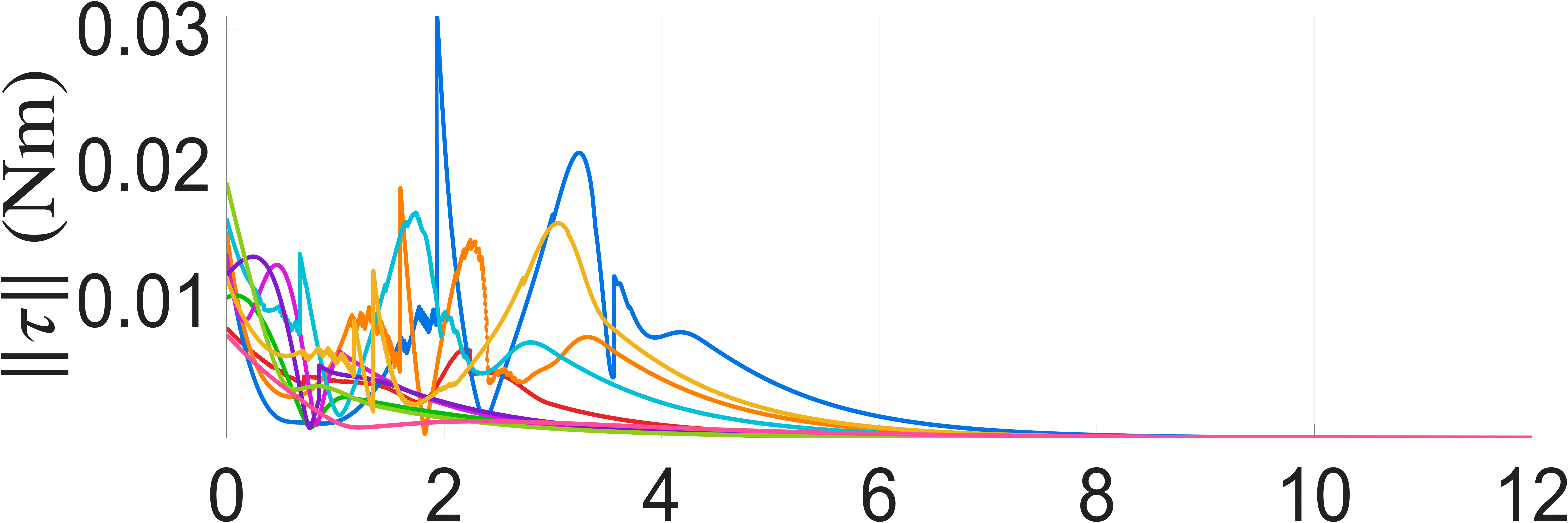}
            \caption{}
            \label{Qua:torque}
        \end{subfigure}
 
    \end{minipage}

    \caption{Simulation of the closed-loop system \eqref{quaternion_dynamics}-\eqref{torque_control_input_quaternion} with $\boldsymbol{\nu}_d$ defined in \eqref{nu_d}. (a) A star-shaped set $\mathcal{O}_1$, (b)-(e) $\mathbf{x}$-trajectories converging to $\mathbf{x}_d = [1, 0, 0, 0]^\top$, (f) $d_{\mathcal{U}}(\mathbf{x})$ versus time, (g) $\|\boldsymbol{\tau}\|$ versus time. 
    (h) $\|\boldsymbol{\omega}\|$ versus time.
    }
    \label{Qua_main_result}
\end{figure*}

\section{Conclusion}
In this work, we developed a feedback controller for the constrained stabilization of second-order systems evolving on the $n$-sphere.
Unlike much of the existing literature, where unsafe regions are typically represented by conic sets, the proposed approach is able to handle a more general class of obstacles represented by star-shaped sets on the $n$-sphere, which offers a more flexible characterization of the unsafe regions, potentially enabling a larger safe region for stabilization purposes.
The proposed feedback controller guarantees safety and almost global asymptotic stability of the equilibrium $(\mathbf{x}_d, \mathbf{0}_{n+1})$ on $\mathcal{M}\times\mathbb{R}^{n+1}$.

The safety analysis requires the constraint boundaries to be smooth, as stated in Assumption \ref{assumption:twice_differentiability}. 
In addition, the stability analysis requires the constraint sets to be sufficiently separated from one another, as specified in Assumption \ref{assumption:sufficient_separation}. 
Relaxing these requirements is an important direction for future work.
Another future direction would be to account for control input constraints while preserving safety and stability guarantees.

\begin{appendix}
    \subsection{Examples of separation function \texorpdfstring{$d_{\mathcal{U}}$}{}}\label{remark:separation_function_examples}
    
    For $\mathbf{x}\in\overline{\mathcal{M}}$, we provide three valid constructions for  $d_{\mathcal{U}}(\mathbf{x})$:
    \begin{itemize}
        \item {\bf Spherical distance:} Define $d_{\mathcal{U}} : \overline{\mathcal{M}}\to[0, \pi]$ by
        \begin{equation}\label{spherical_distance_appendix}
            d_{\mathcal{U}}(\mathbf{x}) = \underset{\mathbf{a}\in\mathcal{U}}{\inf}\;\arccos(\mathbf{x}^\top\mathbf{a}).
        \end{equation}
        This function is used as a separation function in Section \ref{section:explicit_feedback_control}.
        It represents the smallest non-negative angle between $\mathbf{x}\in\overline{\mathcal{M}}$ and elements of $\mathcal{U}$, and satisfies Property \ref{property:distance1}. 

        Using Assumption \ref{assumption:twice_differentiability}, one can show that
        there exists $\delta_u > 0$ such that for every $\mathbf{x}\in\mathcal{N}^{\delta_u}$, the minimizer \begin{equation}\label{unique_minimizer_union}\Pi_{\mathcal{U}}(\mathbf{x}) = \underset{\mathbf{a}\in\mathcal{U}}{\arg\min}\;\arccos(\mathbf{x}^\top\mathbf{a})\end{equation} is unique, where the set $\mathcal{N}^{\delta_u}$ is defined in \eqref{neighborhood_definition}.
        Consequently, $d_{\mathcal{U}}$ defined in \eqref{spherical_distance_appendix} is continuously differentiable on $\mathcal{N}^{\delta_u}$. 
        For all $\mathbf{x}\in\mathcal{N}^{\delta_u}$, the projection of the gradient of $d_{\mathcal{U}}(\mathbf{x})$ onto the tangent space $\mathsf{T}_{\mathbf{x}}\mathbb{S}^n$ is given by
         \begin{equation}
        \mathbf{P}(\mathbf{x})\nabla_{\mathbf{x}}d_{\mathcal{U}}(\mathbf{x}) = -\frac{\mathbf{P}(\mathbf{x})\Pi_{\mathcal{U}}(\mathbf{x})}{\|\mathbf{P}(\mathbf{x})\Pi_{\mathcal{U}}(\mathbf{x})\|}.
        \end{equation}
        Therefore, it is straightforward to verify that the spherical distance function satisfies Property \ref{property:distance2}.
        
        \item {\bf Scaled squared chordal distance:} Suppose $d_{\mathcal{U}}(\mathbf{x})$ is defined as follows:
        \begin{equation*}
            d_{\mathcal{U}}(\mathbf{x}) = \frac{1}{2}\underset{\mathbf{a}\in\mathcal{U}}{\inf}\;\|\mathbf{x} - \mathbf{a}\|^2 = \underset{\mathbf{a}\in\mathcal{U}}{\inf}\;1 - \mathbf{x}^\top\mathbf{a}.
        \end{equation*}
        This separation function satisfies Property \ref{property:distance1}. 
        Furthermore, using Assumption \ref{assumption:twice_differentiability}, one can show that this $d_{\mathcal{U}}(\mathbf{x})$ satisfies Property \ref{property:distance2}.
        
        \item {\bf Product of spherical distances:} Another example of a separation function is
        \[
        d_{\mathcal{U}}(\mathbf{x}) = \prod_{i\in\mathbb{I}}\underset{\mathbf{a}\in\mathcal{U}_i}{\inf}\;\arccos(\mathbf{x}^\top\mathbf{a}).
        \]
        This function is the product, over $i\in\mathbb{I}$, of the spherical distance between $\mathbf{x}$ and the set $\mathcal{U}_i$.
        Since it is strictly positive when $\mathbf{x} \notin \mathcal{U}$ and zero when $\mathbf{x} \in \partial\mathcal{U}$, it satisfies Property \ref{property:distance1}.
        
        Furthermore, if constraint sets $\mathcal{U}_i,i\in\mathbb{I}$ are such that $\Pi_{\mathcal{U}_i}(\mathbf{x})$ is unique for all $\mathbf{x}\in\mathbb{S}^n\setminus\mathcal{U}_i$, then one can ensure that this separation function satisfies Property \ref{property:distance2} for all $\mathbf{x}\in\mathcal{M}$, where $\Pi_{\mathcal{U}_i}(\mathbf{x})$ is obtained by replacing $\mathcal{U}$ with $\mathcal{U}_i$ in \eqref{unique_minimizer_union}.\end{itemize}

    \subsection{Chain recurrent set}\label{appendix:chain_recurrent_set_concept} Let $\Phi(s, \mathbf{p})$ denote the point at time $s\geq0$ along the unique solution of $\dot{\mathbf{x}} = \boldsymbol{\nu}_d(\mathbf{x})$ initialized at $\mathbf{p}\in\mathcal{M}$.
    By forward invariance of $\mathcal{M}$, as established in Remark \ref{remark:explanation_of_properties}, $\Phi(s, \mathbf{p})$ is defined and belongs to $\mathcal{M}$ for every $s\geq 0$ and for any $\mathbf{p}\in\mathcal{M}$.
    We recall the following definitions of $(\varepsilon, T)$-chains and the chain recurrent set from \cite[Definition 1.3]{mischaikow1995asymptotically}.

    \begin{definition}[$(\varepsilon, T)$-chain between two points]\label{definition:e-T_chain}
    Given $\varepsilon > 0$ and  $T > 0$, an $(\varepsilon, T)$-chain from $\mathbf{x}$ to $\mathbf{y}$ in $\mathcal{M}$ is a sequence $\{\mathbf{x}_1, \mathbf{x}_2, \ldots, \mathbf{x}_{N+1}; t_1, t_2, \ldots, t_N\}$ of points $\mathbf{x}_i\in\mathcal{M}$ and times $t_i\geq T$, where $N\in\mathbb{N}$, such that
    \begin{enumerate}
        \item $\mathbf{x}_1 = \mathbf{x}$ and $\mathbf{x}_{N+1} = \mathbf{y}$.
        \item $\|\Phi(t_i, \mathbf{x}_i) - \mathbf{x}_{i+1}\| < \varepsilon,\, i = 1, 2, \ldots, N$.
    \end{enumerate}
    \end{definition}

    \begin{definition}[Chain recurrent set]\label{definition:chain_recurrent_set}
        A point $\mathbf{p}$ in $\mathcal{M}$ is chain recurrent for $\dot{\mathbf{x}} = \boldsymbol{\nu}_d(\mathbf{x})$ if, for every $\varepsilon > 0$ and $T > 0$, there exists an $(\varepsilon, T)$-chain from $\mathbf{p}$ to itself.
        The set of all such points is the chain recurrent set associated with $\dot{\mathbf{x}} = \boldsymbol{\nu}_d(\mathbf{x})$.
    \end{definition}

    \subsection{Proof of Lemma \ref{lemma:safety_lemma}}\label{proof:lemma:safety_lemma}
    \subsubsection{Proof of Claim \ref{claim1:lemmaVTF}}
    The proof is by contradiction. 
    Assume that there exists a finite time $T > 0$ such that $d_{\mathcal{U}}(\mathbf{x}(T)) = 0$ and $d_{\mathcal{U}}(\mathbf{x}(t)) > 0$ for all $t\in [0, T)$.
    Since $\mathbf{x}(t)$ is continuous for all $t\in[0, T]$, by construction in Section \ref{section:problem_statement}, $d_{\mathcal{U}}(\mathbf{x}(t))$ is continuous for all $t\in[0, T]$.
    Consequently, there exists $t_1\in[0, T)$ such that
    \begin{equation}\label{belongs_to_rho_neighborhood}
    d_{\mathcal{U}}(\mathbf{x}(t))\in(0, \rho),\;\forall t\in[t_1,T),
    \end{equation}
    where $\rho=\min\{\delta_d, \epsilon_1\}$, the existence of $\delta_d > 0$ is guaranteed by Property \ref{condition:Assump2:moveaway}, and $\epsilon_1$ is defined in \eqref{beta_function_definition}.

    By Property \ref{condition:Assump2:moveaway}, $\delta_d \leq \delta_u$, where $\delta_u > 0$ is a known parameter, defined as in Property \ref{property:distance2}.
    Therefore, $\rho = \min\{\delta_d, \epsilon_1\}$ implies that $0 < \rho \leq \delta_u$.
    Consequently, since $d_{\mathcal{U}}(\mathbf{x}(t))\in(0, \rho)$ for all $t\in[t_1, T)$, by Property \ref{property:distance2}, $d_{\mathcal{U}}(\mathbf{x}(t))$ is continuously differentiable over the time interval $[t_1, T)$, and one has\begin{equation}\label{derivative_of_spherical_distance}
        \dot{d}_{\mathcal{U}}(\mathbf{x}(t)) = \nabla_{\mathbf{x}}d_{\mathcal{U}}(\mathbf{x}(t))^\top\mathbf{P}(\mathbf{x}(t))\mathbf{v}(t),
    \end{equation}
    for all $t\in[t_1, T)$.

Define $\mathbf{z} = \mathbf{v} - \boldsymbol{\nu}_d(\mathbf{x})$.
Therefore, substituting $\mathbf{v} = \boldsymbol{\nu}_d(\mathbf{x}) + \mathbf{z}$ in \eqref{derivative_of_spherical_distance}, one gets
\begin{equation}
\begin{aligned}
        \dot{d}_{\mathcal{U}}(\mathbf{x}(t)) &= \nabla_{\mathbf{x}}d_{\mathcal{U}}(\mathbf{x}(t))^\top\mathbf{P}(\mathbf{x}(t))\boldsymbol{\nu}_d(\mathbf{x}(t))\\ 
        &+\nabla_{\mathbf{x}}d_{\mathcal{U}}(\mathbf{x}(t))^\top\mathbf{P}(\mathbf{x}(t))\mathbf{z}(t).
        \end{aligned}
    \end{equation}
Since $\rho=\min\{\delta_d, \epsilon_1\}$, using Property \ref{condition:Assump2:moveaway} and \eqref{belongs_to_rho_neighborhood}, one obtains
\begin{equation*}
\begin{aligned}
        \dot{d}_{\mathcal{U}}(\mathbf{x}(t)) \geq \mu +\nabla_{\mathbf{x}}d_{\mathcal{U}}(\mathbf{x}(t))^\top\mathbf{P}(\mathbf{x}(t))\mathbf{z}(t),\;t\in[t_1, T),
        \end{aligned}
    \end{equation*}
    where $\mu > 0$.
Furthermore, since $\delta_d \in (0, \delta_u]$, by Property \ref{property:distance2} and \eqref{belongs_to_rho_neighborhood}, one has $\|\mathbf{P}(\mathbf{x}(t))\nabla_{\mathbf{x}}d_{\mathcal{U}}(\mathbf{x}(t))\| \leq D_d$, $\forall t\in[t_1, T)$, where $D_d > 0$.
Consequently, by applying Cauchy-Schwarz inequality, it follows that
\begin{equation}\label{equation:inequality_zt}
\begin{aligned}
        \dot{d}_{\mathcal{U}}(\mathbf{x}(t)) \geq \mu - D_d\|\mathbf{z}(t)\|, \; \forall t\in[t_1, T).
        \end{aligned}
    \end{equation}

Suppose $\Lim_{t\to T}\|\mathbf{z}(t)\| = 0$.
Therefore, there exists a time $t_s\in[t_1, T)$ such that for all $t\in[t_s, T)$, $\|\mathbf{z}(t)\| < \frac{\mu}{D_d}$.
Consequently,  $\dot{d}_{\mathcal{U}}(\mathbf{x}(t)) > 0,$ $\forall t\in[t_s, T)$.
In other words, there exists a time $t_s$ before $T$ such that $d_{\mathcal{U}}(\mathbf{x}(t_s)) > 0$ and between the time interval $[t_s, T)$, the value of $d_{\mathcal{U}}(\mathbf{x}(t))$ strictly increases.
As a result, since $d_{\mathcal{U}}(\mathbf{x}(t))$ is continuous over the time interval $[t_s, T]$, $d_{\mathcal{U}}(\mathbf{x}(T))\ne 0$, thereby leading to a contradiction.

We proceed to show that $\Lim_{t\to T}\|\mathbf{z}(t)\| = 0$.
Define 
\[V = \frac{1}{2}\|\mathbf{z}\|^2.\] 
Taking the time derivative and using \eqref{proposed_feedback_control_input}, one obtains
\begin{equation}\label{derivative_of_V}
    \dot{V} = -k_d\beta(d_{\mathcal{U}}(\mathbf{x}))\|\mathbf{z}\|^2.
\end{equation}
Since $d_{\mathcal{U}}(\mathbf{x}(t))\in(0, \rho)$ for all $t\in[t_1, T)$ and $\rho=\min\{\delta_d, \epsilon_1\}$, by \eqref{beta_function_definition}, one has $\beta(d_{\mathcal{U}}(\mathbf{x}(t))) = d_{\mathcal{U}}(\mathbf{x}(t))^{-1}, \forall t\in[t_1, T)$, and it follows from \eqref{derivative_of_V} that 
\begin{equation}\label{equation_of_V_to_integrate}
    \dot{V}(t) = \frac{-2k_d}{d_{\mathcal{U}}(\mathbf{x}(t))}V(t), \;\forall t\in[t_1, T).
\end{equation}

To show $\Lim_{t\to T}\|\mathbf{z}(t)\| = 0$, we proceed in the following manner.
First, we show that $\mathbf{v}(t)$ is bounded for all $t\in[t_1, T)$. 
This will imply that there exists $D_v > 0$ such that $\|\mathbf{v}(t)\|\leq D_v, \;\forall t\in[t_1, T)$.
Then, using \eqref{derivative_of_spherical_distance}, we will show that $d_{\mathcal{U}}(\mathbf{x}(t)) \leq D_dD_v(T - t)$ for all $t\in[t_1, T)$, where $D_d$ is defined in Property \ref{property:distance2}.
After substituting this bound into \eqref{equation_of_V_to_integrate} and integrating with respect to time from $t_1$ to $t\in[t_1, T)$, we will be able to conclude that $\Lim_{t\to T}V(t) = 0$, and the result will follow.

We proceed to show that there exists $D_v > 0$ such that $\|\mathbf{v}(t)\|\leq D_v, \forall t\in[t_1, T)$.
Since $d_{\mathcal{U}}(\mathbf{x}(t))>0, \;\forall t\in[t_1, T)$, $\dot{V}(t)$ in \eqref{equation_of_V_to_integrate} is well-defined $\forall t\in[t_1, T)$, and $V(t)$ is non-increasing on $[t_1, T)$.
Therefore, $\mathbf{z}(t)$ is bounded for all $t\in[t_1, T)$ and satisfies $\|\mathbf{z}(t)\|\leq \|\mathbf{z}(t_1)\|, \;\forall t\in[t_1, T)$.
Additionally, by Property \ref{property:bounded_vector_field}, $\|\boldsymbol{\nu}_d(\mathbf{x})\|\leq D_1$ for all $\mathbf{x}\in\mathcal{M}$.
Therefore, since $\mathbf{v}(t) = \mathbf{z}(t) + \boldsymbol{\nu}_d(\mathbf{x}(t))$, one has
\begin{equation}\label{bound_on_velocity}
\|\mathbf{v}(t)\|\leq D_v,\;\forall t\in[t_1, T),
\end{equation}
where $D_v = \|\mathbf{z}(t_1)\| + D_1$.

Now, we show that $d_{\mathcal{U}}(\mathbf{x}(t)) \leq D_dD_v(T-t), \;\forall t\in[t_1, T)$.
By Property \ref{property:distance2}, $\|\mathbf{P}(\mathbf{x})\nabla_{\mathbf{x}}d_{\mathcal{U}}(\mathbf{x})\|\leq D_d$ for all $\mathbf{x}\in\mathcal{N}^{\delta_u}$, where $D_d > 0$ and $\mathcal{N}^{\delta_u}$ is defined in \eqref{neighborhood_definition}.
Consequently, since $d_{\mathcal{U}}(\mathbf{x}(t))\in(0, \rho)$, $\rho = \min\{\delta_d, \epsilon_1\}$ and $\delta_d\leq \delta_u$, one has
\begin{equation*}
    \dot{d}_{\mathcal{U}}(\mathbf{x}(t)) \geq -D_dD_v, \;\forall t\in[t_1, T).
\end{equation*}
Integrating with respect to time from $t\in[t_1, T)$ to $T$, one obtains
\begin{equation*}
    d_{\mathcal{U}}(\mathbf{x}(T)) - d_{\mathcal{U}}(\mathbf{x}(t)) \geq -D_dD_v(T - t).
\end{equation*}
Substituting $d_{\mathcal{U}}(\mathbf{x}(T)) = 0$, and rearranging the terms, one gets
\begin{equation}\label{bound_on_distance}
    d_{\mathcal{U}}(\mathbf{x}(t)) \leq D_dD_v(T- t),\;\forall t\in[t_1, T).
\end{equation}

Using \eqref{equation_of_V_to_integrate} and \eqref{bound_on_distance}, one can write
\begin{equation*}
    \dot{V}(t) \leq \frac{-2k_d}{D_dD_v(T-t)}V(t), \;\forall t\in[t_1, T),
\end{equation*}
where $k_d, D_v > 0$.
Integrating with respect to time from $t_1$ to $t\in[t_1, T)$, one obtains
\begin{equation}\label{final_integration_of_V_from_t12T}
    V(t) \leq V(t_1)\left(\frac{T-t}{T-t_1}\right)^{\frac{2k_d}{D_dD_v}}.
\end{equation}
It follows from \eqref{final_integration_of_V_from_t12T} that $\Lim_{t\to T}V(t) = 0$, and hence $\Lim_{t\to T}\|\mathbf{z}(t)\| = 0$.
This completes the proof of Claim \ref{claim1:lemmaVTF} of Lemma \ref{lemma:safety_lemma}.

\subsubsection{Proof of Claim \ref{claim2:lemmaVTF}}
We show that for every $\boldsymbol{\xi}(0)\in\mathcal{M}\times\mathbb{R}^{n+1}$, there exist $t_s(\boldsymbol{\xi}(0))\geq0$ and $\zeta > 0$ such that for all $t\geq t_s(\boldsymbol{\xi}(0))$, $\dot{d}_{\mathcal{U}}(\mathbf{x}(t)) \geq \zeta $ whenever $d_{\mathcal{U}}(\mathbf{x}(t))\in(0, \delta_d]$.
This will imply that there exists $t_d(\boldsymbol{\xi}(0))\geq t_s(\boldsymbol{\xi}(0))$ such that $d_{\mathcal{U}}(\mathbf{x}(t))\geq \delta_d$ for all $t\geq t_d(\boldsymbol{\xi}(0))$. 

By Property \ref{condition:Assump2:moveaway}, $\delta_d \in (0, \delta_u]$.
Therefore, following arguments used to derive \eqref{equation:inequality_zt} in the proof of Claim \ref{claim1:lemmaVTF} of Lemma \ref{lemma:safety_lemma}, one can conclude that the inequality
\begin{equation}\label{expression2}
\dot{d}_{\mathcal{U}}(\mathbf{x}) \geq \mu - D_d\|\mathbf{z}\|,
\end{equation}
holds for all $\mathbf{x}\in\mathcal{N}^{\delta_d}$, where the existence of $\mu > 0$ is assumed in Property \ref{condition:Assump2:moveaway}, $\mathbf{z} = \mathbf{v} - \boldsymbol{\nu}_d(\mathbf{x})$, and the set $\mathcal{N}^{\delta_d}$ is obtained by replacing $\delta_u$ with $\delta_d$ in \eqref{neighborhood_definition}.
To proceed with the proof of Claim \ref{claim2:lemmaVTF}, we require the following lemma:
\begin{lemma}\label{lemma:monotonic_decrease_velocity_difference}
Consider the closed-loop system \eqref{dynamics_motion_model_on_sphere}-\eqref{proposed_feedback_control_input}, under Assumption \ref{assumption:twice_differentiability}.
Then, for all $t\geq 0$, there exists $\underline{\beta} > 0$ such that
\begin{equation}\label{z-converging-exponentially-to-zero}
    \|\mathbf{z}(t)\|\leq\|\mathbf{z}(0)\|e^{-k_d\underline{\beta}t},
\end{equation}
where $\mathbf{z}(t) = \mathbf{v}(t) - \boldsymbol{\nu}_d(\mathbf{x}(t))$.
\end{lemma}
\proof{Taking the time derivative of $V = \frac{1}{2}\|\mathbf{z}\|^2$ along the closed-loop dynamics \eqref{dynamics_motion_model_on_sphere}-\eqref{proposed_feedback_control_input}, one obtains
\begin{equation}
    \dot{V} = -2k_d\beta(d_{\mathcal{U}}(\mathbf{x})) V.
\end{equation}
Claim \ref{claim1:lemmaVTF} of Lemma \ref{lemma:safety_lemma} ensures that if $\mathbf{x}(0)\in\mathcal{M}$, then $\mathbf{x}(t)\in\mathcal{M}, \forall t\geq 0$.
In addition, by \eqref{beta_function_definition}, there exists $\underline{\beta} > 0$ such that $\beta(d_{\mathcal{U}}(\mathbf{x}))\geq \underline{\beta}, \forall \mathbf{x}\in\mathcal{M}$. 
Therefore,
\begin{equation}\label{monotonic-decrease-inequality}
    \dot{V}(t)\leq -2k_d\underline{\beta} V(t)\leq 0, \forall t\geq 0.
\end{equation}
It follows that
\begin{equation}
    V(t)\leq V(0)e^{-2k_d\underline{\beta}t}, \forall t\geq 0.
\end{equation}
Substituting $V(t) = \frac{1}{2}\|\mathbf{z}(t)\|^2$ and taking the square roots, one obtains \eqref{z-converging-exponentially-to-zero}.
This completes the proof.}

%The proof of Lemma \ref{lemma:monotonic_decrease_velocity_difference} is similar to the proof of \cite[Lemma 4]{sawant2025extendingfirstorderroboticmotion}.
According to Lemma \ref{lemma:monotonic_decrease_velocity_difference}, $\Lim_{t\to\infty}\|\mathbf{z}(t)\| = 0$.
Therefore, for any $s\in\left(0, \frac{\mu}{D_d}\right)$, there exists $t_s(\boldsymbol{\xi}(0)) \geq 0$ such that
\[\|\mathbf{z}(t)\| \leq s\; \text{for all}\; t\geq t_s(\boldsymbol{\xi}(0)).\]
Consequently, it follows from \eqref{expression2} that for all $t\geq t_s(\boldsymbol{\xi}(0))$, the inequality $\dot{d}_{\mathcal{U}}(\mathbf{x}(t))\geq \zeta > 0$ holds whenever $d_{\mathcal{U}}(\mathbf{x}(t))\in (0, \delta_d]$ with $\zeta = \mu - s$.
As a result, there exists a time $t_d(\boldsymbol{\xi}(0))\geq t_s(\boldsymbol{\xi}(0))$ such that $d_{\mathcal{U}}(\mathbf{x}(t)) \geq \delta_d$ for all $t\geq t_d(\boldsymbol{\xi}(0))$.
This completes the proof of Claim \ref{claim2:lemmaVTF} of Lemma \ref{lemma:safety_lemma}.

\subsubsection{Proof of Claim \ref{claim3:lemmaVTF}}
Taking the norm of both sides of \eqref{proposed_feedback_control_input}, one obtains
\begin{equation*}
    \begin{aligned}
        \|\mathbf{u}(\boldsymbol{\xi})\|&\leq k_d\beta(d_{\mathcal{U}}(\mathbf{x}))\|\mathbf{v} - \boldsymbol{\nu}_d(\mathbf{x})\|+\|\mathbf{J}_d(\mathbf{x})\|_{F}\|\mathbf{P}(\mathbf{x})\mathbf{v}\|,
    \end{aligned}
\end{equation*}where $\boldsymbol{\xi}=(\mathbf{x}, \mathbf{v})\in\mathcal{M}\times\mathbb{R}^{n+1}$ and we used the property $\|\mathbf{a} + \mathbf{b}\| \leq \|\mathbf{a}\| + \|\mathbf{b}\|$ for any $\mathbf{a}, \mathbf{b}\in\mathbb{R}^{n+1}$.
We have $\|\mathbf{P}(\mathbf{x})\mathbf{v}\|\leq \|\mathbf{v}\|$ for any $\mathbf{x}\in\mathbb{S}^n$ and for any $\mathbf{v}\in\mathbb{R}^{n+1}$.
It follows that
\begin{equation}\label{bounding_expression_1}
    \begin{aligned}
        \|\mathbf{u}(\boldsymbol{\xi})\|&\leq k_d\beta(d_{\mathcal{U}}(\mathbf{x}))\|\mathbf{v} - \boldsymbol{\nu}_d(\mathbf{x})\|+\|\mathbf{J}_d(\mathbf{x})\|_{F}\|\mathbf{v}\|.
    \end{aligned}
\end{equation}
According to Lemma \ref{lemma:monotonic_decrease_velocity_difference}, $\|\mathbf{v}(t) - \boldsymbol{\nu}_d(\mathbf{x}(t))\|\leq \bar{z}(\boldsymbol{\xi}(0))$, where $\bar{z}(\boldsymbol{\xi}(0)) = \|\mathbf{v}(0) - \boldsymbol{\nu}_d(\mathbf{x}(0))\|$ for all $t\geq 0$. 
Lemma \ref{lemma:monotonic_decrease_velocity_difference} also implies that $\|\mathbf{v}(t)\| \leq D_1 + \bar{z}(\boldsymbol{\xi}(0))$ for all $t\geq 0$, where the constant $D_1 >0$ is defined in Property \ref{property:bounded_vector_field}.
Consequently, by \eqref{bounding_expression_1}, for any $\boldsymbol{\xi}(0)\in\mathcal{M}\times\mathbb{R}^{n+1}$,
\begin{equation}\label{bounding_expression_2}
    \begin{aligned}
        \|\mathbf{u}(\boldsymbol{\xi}(t))\|&\leq k_d\beta(d_{\mathcal{U}}(\mathbf{x}(t)))\bar{z}(\boldsymbol{\xi}(0))\\
        &+(D_1 + \bar{z}(\boldsymbol{\xi}(0)))\|\mathbf{J}_d(\mathbf{x})\|_{F},
    \end{aligned}
\end{equation}
for all $t\geq 0$.

To show the existence of $D_{\mathbf{u}}(\boldsymbol{\xi}(0)) > 0$ such that $\|\mathbf{u}(\boldsymbol{\xi}(t))\|\leq D_{\mathbf{u}}(\boldsymbol{\xi}(0))$ for all $t\geq 0$, it is sufficient to show that for any initial condition $\boldsymbol{\xi}(0)$ in $\mathcal{M}\times\mathbb{R}^{n+1}$, there exist $\bar{\beta}(\boldsymbol{\xi}(0)) > 0$ and $\bar{J}(\boldsymbol{\xi}(0)) > 0$ such that the solution to the closed-loop system \eqref{dynamics_motion_model_on_sphere}-\eqref{proposed_feedback_control_input} satisfies $\beta(d_{\mathcal{U}}(\mathbf{x}(t)))\leq \bar{\beta}(\boldsymbol{\xi}(0))$ and $\|\mathbf{J}_d(\mathbf{x}(t))\|_F\leq \bar{J}(\boldsymbol{\xi}(0))$.

By Claim \ref{claim2:lemmaVTF} of Lemma \ref{lemma:safety_lemma}, for any solution to the closed-loop system \eqref{dynamics_motion_model_on_sphere}-\eqref{proposed_feedback_control_input} with initial condition  $\boldsymbol{\xi}(0)\in\mathcal{M}\times\mathbb{R}^{n+1}$, there exists a finite time $t_d(\boldsymbol{\xi}(0)) \geq 0$ such that \begin{equation}\label{delta_d_separation_condition}d_{\mathcal{U}}(\mathbf{x}(t))\geq \delta_d > 0, \;\forall t\geq t_d(\boldsymbol{\xi}(0)).
\end{equation}

Additionally, by \eqref{beta_function_definition}, $\beta(d_{\mathcal{U}}(\mathbf{x}))$ is strictly positive where $d_{\mathcal{U}}(\mathbf{x}) > 0$, and is undefined if and only if $d_{\mathcal{U}}(\mathbf{x}) = 0$.
Since $d_{\mathcal{U}}(\mathbf{x}(t))\geq \delta_d > 0$ for all $t\geq t_d(\boldsymbol{\xi}(0))$, and $\mathbb{S}^n$ is a compact subset of $\mathbb{R}^{n+1}$, there exists $\bar{\beta}_1(\boldsymbol{\xi}(0)) > 0$ such that $\beta(d_{\mathcal{U}}(\mathbf{x}(t))) \leq \bar{\beta}_1(\boldsymbol{\xi}(0)), \;\forall t\geq t_d(\boldsymbol{\xi}(0))$.
Furthermore, $\beta(d_{\mathcal{U}}(\mathbf{x}(t)))$ is continuous for all $t\geq 0$, and by Claim \ref{claim1:lemmaVTF} of Lemma \ref{lemma:safety_lemma}, $d_{\mathcal{U}}(\mathbf{x}(t)) > 0$ for all $t\geq 0$.
Therefore, over the compact time interval $[0, t_d(\boldsymbol{\xi}(0))]$, there exists an upper bound $\bar{\beta}_2(\boldsymbol{\xi}(0)) > 0$ such that $\beta(d_{\mathcal{U}}(\mathbf{x}(t))) \leq \bar{\beta}_2(\boldsymbol{\xi}(0))$.
As a result,
\begin{equation}\label{beta_bound}
    \beta(d_{\mathcal{U}}(\mathbf{x}(t))) \leq \bar{\beta}(\boldsymbol{\xi}(0)),\;\forall t\geq 0,
\end{equation}
where $\bar{\beta}(\boldsymbol{\xi}(0)) = \max\{\bar{\beta}_1(\boldsymbol{\xi}(0)), \bar{\beta}_2(\boldsymbol{\xi}(0))\}$.

Now, we show that there exists $\bar{J}(\boldsymbol{\xi}(0)) > 0$ such that $\|\mathbf{J}_d(\mathbf{x}(t))\|_F \leq \bar{J}(\boldsymbol{\xi}(0)), \;\forall t\geq 0$.
We know that $d_{\mathcal{U}}(\mathbf{x}(t)) \geq \delta_d \geq 0$ for all $t\geq t_d(\boldsymbol{\xi}(0))$, as stated in \eqref{delta_d_separation_condition}.
In other words, $\mathbf{x}(t)\in\mathcal{M}\setminus(\mathcal{N}^{\delta_d})^{\circ}, \forall t\geq t_d(\boldsymbol{\xi}(0))$.
%Therefore, by Property \ref{property:bounded_jacobian}, $\|\mathbf{J}_d(\mathbf{x}(t))\|_F\leq D_2, \;\forall t\geq t_d(\boldsymbol{\xi}(0))$.
Since $\mathcal{M}\setminus(\mathcal{N}^{\delta_d})^{\circ}$ is a compact subset of $\mathbb{S}^n$ and $\mathbf{J}_d$ is continuous on this set by Property \ref{condition:Assump2:differentiability}, there exists $D_2 > 0$ such that $\|\mathbf{J}_d(\mathbf{x}(t))\|_F\leq D_2, \;\forall t\geq t_d(\boldsymbol{\xi}(0))$.
In addition, by Claim \ref{claim1:lemmaVTF} of Lemma \ref{lemma:safety_lemma}, $\mathbf{x}(t)\in\mathcal{M}$ for all $t\in[0, t_d(\boldsymbol{\xi}(0))]$.
Moreover, by Property \ref{condition:Assump2:differentiability}, $\mathbf{J}_d(\mathbf{x}(t))$ is well-defined $\forall t\in[0, t_d(\boldsymbol{\xi}(0))]$.
Consequently, there exists $J_1(\boldsymbol{\xi}(0))>0$ such that $\|\mathbf{J}_d(\mathbf{x}(t))\|_F \leq J_1(\boldsymbol{\xi}(0)), \;\forall t\in[0, t_d(\boldsymbol{\xi}(0))]$.
As a result, 
\begin{equation}\label{J_d_bound}
    \|\mathbf{J}_d(\mathbf{x}(t))\|_F\leq \bar{J}(\boldsymbol{\xi}(0)), \forall t\geq 0,
\end{equation}
where $\bar{J}(\boldsymbol{\xi}(0))=\max\{D_2, J_1(\boldsymbol{\xi}(0))\}$.

It follows from \eqref{bounding_expression_2}, \eqref{beta_bound} and \eqref{J_d_bound} that for any initial condition $\boldsymbol{\xi}(0)\in\mathcal{M}\times\mathbb{R}^{n+1}$, $\|\mathbf{u}(\boldsymbol{\xi}(t))\| \leq D_{\mathbf{u}}(\boldsymbol{\xi}(0))$ for all $t\geq 0$, where
\[
D_{\mathbf{u}}(\boldsymbol{\xi}(0)) = k_d\bar{\beta}(\boldsymbol{\xi}(0))\bar{z}(\boldsymbol{\xi}(0)) + (D_1 + \bar{z}(\boldsymbol{\xi}(0)))\bar{J}(\boldsymbol{\xi}(0)).
\]
This completes the proof of Claim \ref{claim3:lemmaVTF} of Lemma \ref{lemma:safety_lemma}.

\subsection{Proof of Theorem \ref{theorem:VTF}}\label{proof:theorem:VTF}
For the closed-loop system \eqref{dynamics_motion_model_on_sphere}-\eqref{proposed_feedback_control_input}, the forward invariance of $\mathcal{M}\times\mathbb{R}^{n+1}$ can be easily established using Claim \ref{claim1:lemmaVTF} of Lemma \ref{lemma:safety_lemma}.
Claim \ref{claim2:theorem} of Theorem \ref{theorem:VTF} follows from the inequality \eqref{monotonic-decrease-inequality} established in the proof of Lemma \ref{lemma:monotonic_decrease_velocity_difference}.
%Furthermore, the monotonic decrease of $\|\mathbf{v}(t) - \boldsymbol{\nu}_d(\mathbf{x}(t))\|$ for all $t\geq 0$ follows directly from Lemma \ref{lemma:monotonic_decrease_velocity_difference}.

\subsubsection{Proof of Claim \ref{claim3:theorem}} For the closed-loop system \eqref{dynamics_motion_model_on_sphere}-\eqref{proposed_feedback_control_input}, by setting $\dot{\mathbf{x}} = \mathbf{0}_{n+1}$ and $\dot{\mathbf{v}} = \mathbf{0}_{n+1}$, and using Property \ref{condition:set_of_equilibria}, one can verify that the set of equilibrium points is $\mathcal{S}\cup\{(\mathbf{x}_d, \mathbf{0}_{n+1})\}$.

\subsubsection{Proof of Claim \ref{claim4:theorem}} 
Consider the mapping $T:\mathcal{M}\times\mathbb{R}^{n+1}\to\mathcal{M}\times\mathbb{R}^{n+1}$ given by $T(\mathbf{x}, \mathbf{v}) = (\mathbf{x}, \mathbf{v} - \boldsymbol{\nu}_d(\mathbf{x}))$.
Since, by Property \ref{condition:Assump2:differentiability}, $\boldsymbol{\nu}_d$ is continuously differentiable over $\mathcal{M}$, $T$ is continuously differentiable over its domain.
Furthermore, its unique inverse function is given by $T^{-1}(\mathbf{x}, \mathbf{z}) = (\mathbf{x}, \mathbf{z} + \boldsymbol{\nu}_d(\mathbf{x}))$ and is also continuously differentiable over its domain.
Consequently, $T$ is a diffeomorphism over its domain.
Under the transformation $T$, the closed-loop system \eqref{dynamics_motion_model_on_sphere}-\eqref{proposed_feedback_control_input} becomes
\begin{subequations}\label{transformed_closed_loop_system}
    \begin{align}
        \dot{\mathbf{x}}
        &= \boldsymbol{\nu}_d(\mathbf{x}) + \mathbf{P}(\mathbf{x})\mathbf{z},
        \label{transformed_x_subsystem}\\
        \dot{\mathbf{z}}
        &= -k_d\beta(d_{\mathcal{U}}(\mathbf{x}))\mathbf{z}.\label{transformed_z_subsystem}
    \end{align}
\end{subequations}
Since $k_d > 0$, and by \eqref{beta_function_definition}, $\beta(d_{\mathcal{U}}(\mathbf{x}))>0$ for all $\mathbf{x}\in\mathcal{M}$, every equilibrium point $(\mathbf{x}^*, \mathbf{z}^*)$ of \eqref{transformed_closed_loop_system} satisfies $\mathbf{z}^* = \mathbf{0}_{n+1}$.
Therefore, setting $\dot{\mathbf{x}} = \mathbf{0}_{n+1}$ yields $\boldsymbol{\nu}_d(\mathbf{x}^*) = \mathbf{0}_{n+1}$.
Furthermore, Property \ref{condition:set_of_equilibria} establishes that $\boldsymbol{\nu}_d(\mathbf{x}^*) = \mathbf{0}_{n+1}$ if and only if $\mathbf{x}^*\in\mathcal{E}\cup\{\mathbf{x}_d\}$. 
Consequently, the set of equilibrium points for the transformed system \eqref{transformed_closed_loop_system} is characterized by $\mathcal{S}\cup\{(\mathbf{x}_d, \mathbf{0}_{n+1})\}$, where $\mathcal{S}$ is defined in \eqref{equilibrium_set_theorem}.

Since $T(\mathbf{x}^*, \mathbf{v}^*) = (\mathbf{x}^*, \mathbf{v}^* -\mathbf{0}_{n+1}) = (\mathbf{x}^*, \mathbf{v}^*)$ for every $(\mathbf{x}^*, \mathbf{v}^*)\in\mathcal{S}\cup\{(\mathbf{x}_d, \mathbf{0}_{n+1})\}$, every equilibrium point in $\mathcal{S}\cup\{(\mathbf{x}_d, \mathbf{0}_{n+1})\}$ is a fixed point of the map $T$.
Furthermore, because $T$ is a diffeomorphism, the stability properties of equilibrium points are preserved under the state transformation \cite[Ex. 4.26]{khalil2002nonlinear}.
Therefore, we proceed to analyze the stability properties of $(\mathbf{x}^*, \mathbf{z}^*)\in\mathcal{S}\cup\{(\mathbf{x}_d, \mathbf{0}_{n+1})\}$ for the transformed system \eqref{transformed_closed_loop_system}.
The proof is separated into two parts as follows:

\noindent{\bf Part 1:} We show that the set $\mathcal{S}\cup\{(\mathbf{x}_d, \mathbf{0}_{n+1})\}$ is globally attractive for the closed-loop system \eqref{transformed_closed_loop_system} over $\mathcal{M}\times\mathbb{R}^{n+1}$.
Specifically, we show that $\Lim_{t\to\infty}(\mathbf{x}(t), \mathbf{z}(t))\in\mathcal{S}\cup\{(\mathbf{x}_d, \mathbf{0}_{n+1})\}$.

According to Lemma \ref{lemma:monotonic_decrease_velocity_difference}, for any initial condition $(\mathbf{x}(0), \mathbf{z}(0))\in\mathcal{M}\times\mathbb{R}^{n+1}$, the corresponding state trajectory $\mathbf{z}(t)$ converges exponentially to $\mathbf{0}_{n+1}$.
Consequently, the $\mathbf{x}$-subsystem \eqref{transformed_x_subsystem} is asymptotically autonomous in the sense of \cite{mischaikow1995asymptotically}, with limit equation \eqref{ideal_kinematics}.
Furthermore, it follows from Claims \ref{claim1:lemmaVTF} and \ref{claim2:lemmaVTF}, that for every $(\mathbf{x}(0), \mathbf{z}(0))\in\mathcal{M}\times\mathbb{R}^{n+1}$, there exists $\underline{\delta}(\mathbf{x}(0), \mathbf{z}(0)) > 0$ such that $d_{\mathcal{U}}(\mathbf{x}(t)) \geq \underline{\delta}(\mathbf{x}(0), \mathbf{z}(0)), \forall t\geq 0$.
Therefore, since $\mathcal{M}$ is bounded, $\mathbf{x}(t)$ is precompact\footnote{A state trajectory is called precompact
in its state space if its image has compact closure in that space.
Equivalently, its image is contained in a compact subset of the state
space.} in $\mathcal{M}$.
Consequently, by \cite[Theorem 1.8]{mischaikow1995asymptotically}, for any initial condition $(\mathbf{x}(0), \mathbf{z}(0))\in\mathcal{M}\times\mathbb{R}^{n+1}$, the corresponding state trajectory $\mathbf{x}(t)$ approaches the chain recurrent set associated with the autonomous system $\dot{\mathbf{x}} = \boldsymbol{\nu}_d(\mathbf{x})$.

Consequently, since, by Property \ref{condition:set_of_equilibria}, the chain recurrent set associated with $\dot{\mathbf{x}} = \boldsymbol{\nu}_d(\mathbf{x})$ is subset of $\mathcal{E}\cup\{\mathbf{x}_d\}$, it follows that
\begin{equation}\label{equilibria_distance_goes_to_zero}
    \Lim_{t\to\infty}d_{\mathcal{E}\cup\{\mathbf{x}_d\}}(\mathbf{x}(t)) = \Lim_{t\to\infty}\underset{\mathbf{y}\in\mathcal{E}\cup\{\mathbf{x}_d\}}{\inf}\|\mathbf{x}(t) - \mathbf{y}\| = 0.
\end{equation}
%where the distance function $d_{\mathcal{E}\cup\{\mathbf{x}_d\}}$ is defined similar to $d_{\mathcal{U}}$ provided in Section \ref{section:problem_statement}, where $\mathcal{U}$ is replaced with $\mathcal{E}\cup\{\mathbf{x}_d\}$.
Since, by Property \ref{condition:set_of_equilibria}, $\mathcal{E}\cup\{\mathbf{x}_d\}$ is a finite set of isolated points, and $\mathbf{x}(t)$ is continuous $\forall t\geq 0$, \eqref{equilibria_distance_goes_to_zero} implies that $\Lim_{t\to\infty}\mathbf{x}(t)\in\mathcal{E}\cup\{\mathbf{x}_d\}$.
Consequently, the definition of $\mathcal{S}$ in \eqref{equilibrium_set_theorem} and the fact that $\Lim_{t\to\infty}\mathbf{z}(t) = \mathbf{0}_{n+1}$ imply that $\Lim_{t\to\infty}(\mathbf{x}(t), \mathbf{z}(t))\in\mathcal{S}\cup\{(\mathbf{x}_d, \mathbf{0}_{n+1})\}$.

\noindent{\bf Part 2:} We show that for the closed-loop system \eqref{transformed_closed_loop_system}, the desired equilibrium $(\mathbf{x}_d, \mathbf{0}_{n+1})$ is asymptotically stable, and every equilibrium point in $\mathcal{S}$ has a stable manifold of zero Lebesgue measure on $\mathcal{M}\times\mathbb{R}^{n+1}$.

To analyze the properties of the equilibrium point in $\mathcal{S}\cup\{(\mathbf{x}_d, \mathbf{0}_{n+1})\}$, we examine the eigenvalues of the Jacobian matrix of the transformed system \eqref{transformed_closed_loop_system} at these points.
Since $\mathbf{z}^* = \mathbf{0}_{n+1}$ for each $(\mathbf{x}^*, \mathbf{z}^*)\in\mathcal{S}\cup\{(\mathbf{x}_d, \mathbf{0}_{n+1})\}$, the Jacobian matrix $\mathbf{J}(\mathbf{x}^*, \mathbf{z}^*)$ is given by

\begin{equation}
    \mathbf{J}(\mathbf{x}^*, \mathbf{z}^*) = \begin{bmatrix}
        \mathbf{J}_d(\mathbf{x}^*) & \mathbf{P}(\mathbf{x}^*)\\
        \mathbf{O}_{n+1} & -k_d\beta(d_{\mathcal{U}}(\mathbf{x}^*))\mathbf{I}_{n+1}
    \end{bmatrix}.
\end{equation}

Since $\mathbf{J}(\mathbf{x}^*, \mathbf{z}^*)$ is an upper block triangular matrix, by \cite[Proposition 5.5.13]{bernstein2009matrix}, the set of eigenvalues of $\mathbf{J}(\mathbf{x}^*, \mathbf{z}^*)$ is equal to the union of the set of eigenvalues of $\mathbf{J}_d(\mathbf{x}^*)$ and $n+1$ eigenvalues of value $-k_d\beta(d_{\mathcal{U}}(\mathbf{x}^*))$.
Since $\mathbf{x}^*\in\mathcal{E}\cup\{\mathbf{x}_d\}\subset\mathcal{M}$, one has $d_{\mathcal{U}}(\mathbf{x}^*) > 0$.
Therefore, since $k_d > 0$, it follows from the definition of $\beta$ in \eqref{beta_function_definition} that $-k_d\beta(d_{\mathcal{U}}(\mathbf{x}^*)) < 0$ for all $\mathbf{x}^*\in\mathcal{E}\cup\{\mathbf{x}_d\}$.
Consequently, to characterize the stability properties of the equilibrium points $(\mathbf{x}^*, \mathbf{z}^*)\in\mathcal{S}\cup\{(\mathbf{x}_d, \mathbf{0}_{n+1})\}$, it is sufficient to analyze the eigenvalues of $\mathbf{J}_d(\mathbf{x}^*)$.

According to Property \ref{condition:xd}, every eigenvalue of $\mathbf{J}_d(\mathbf{x}_d)$ whose associated eigenvector belongs to $\mathsf{T}_{\mathbf{x}_d}\mathbb{S}^n$ has negative real parts. 
Consequently, $(\mathbf{x}_d, \mathbf{0}_{n+1})$ is asymptotically stable for the closed-loop system \eqref{transformed_closed_loop_system} on $\mathcal{M}\times\mathbb{R}^{n+1}$.

Furthermore, according to Property \ref{condition:undesired}, for each $\mathbf{x}^*\in\mathcal{E}$, at least one eigenvalue whose associated eigenvector belongs to $\mathsf{T}_{\mathbf{x}^*}\mathbb{S}^n$ has a positive real part.
Consequently, by virtue of the center manifold theorem \cite[Section 2.7, Pg 116]{perko2013differential}, every undesired equilibrium point $(\mathbf{x}^*, \mathbf{0}_{n+1})$ in $\mathcal{S}$ has stable manifold of zero Lebesgue measure in $\mathcal{M}\times\mathbb{R}^{n+1}$.

Since, by Property \ref{condition:set_of_equilibria}, $\mathcal{E}$ consists of isolated points, $\mathcal{S}$ is a set of isolated equilibrium points.
Therefore, since the stable manifold associated with each equilibrium point in $\mathcal{S}$ has zero Lebesgue measure in $\mathcal{M}\times\mathbb{R}^{n+1}$, the union of these stable manifolds also has zero Lebesgue measure in $\mathcal{M}\times\mathbb{R}^{n+1}$.

As proved earlier, $\mathcal{S}\cup\{(\mathbf{x}_d, \mathbf{0}_{n+1})\}$ is globally attractive for the closed-loop system \eqref{transformed_closed_loop_system} over $\mathcal{M}\times\mathbb{R}^{n+1}$.
Furthermore, $(\mathbf{x}_d, \mathbf{0}_{n+1})$ is an asymptotically stable equilibrium point, and the union of the stable manifolds of equilibria in $\mathcal{S}$ has zero Lebesgue measure in $\mathcal{M}\times\mathbb{R}^{n+1}$.
Consequently, the desired equilibrium point $(\mathbf{x}_d, \mathbf{0}_{n+1})$ is almost globally asymptotically stable for the closed-loop system \eqref{transformed_closed_loop_system} over $\mathcal{M}\times\mathbb{R}^{n+1}$. 
This completes the proof of Claim \ref{claim4:theorem} of Theorem \ref{theorem:VTF}.

\subsection{Proof of Lemma \ref{lemma:planner-satisfies-properties}}
\label{proof:lemma:planner-satisfies-properties}

From \eqref{individual_neighborhood}, \eqref{definition-Ri-general}, and \eqref{definition-Ri-special}, it follows that $\mathcal{N}_i^\delta\subset\mathcal{R}_i^\delta, \forall i\in\mathbb{I}$.
Therefore, by Assumption \ref{assumption:sufficient_separation}, $\overline{\mathcal{N}_i^\delta}\cap\overline{\mathcal{N}_j^\delta} = \emptyset$.
Consequently, by Assumption \ref{assumption:twice_differentiability}, there exists $\delta_u\in (0, \delta)$ such that $d_{\mathcal{U}}$, defined in \eqref{spherical_distance_definition}, is continuously differentiable over $\mathcal{N}_i^{\delta_u}, i\in\mathbb{I}$.
Assumption \ref{assumption:sufficient_separation} further implies that if $\mathbf{x}\in\partial\mathcal{U}$, then there exists a unique index $i\in\mathbb{I}$ such that $\mathbf{x}\in\partial\mathcal{U}_i$, and by \eqref{nu_d}, $\boldsymbol{\nu}_d(\mathbf{x}) = -\frac{1}{\kappa}\mathbf{P}(\mathbf{x})\mathbf{g}_i$.
Since $\mathbf{P}(\mathbf{x})\mathbf{x} = \mathbf{0}, \forall \mathbf{x}\in\mathbb{S}^n$, the term $\mathbf{P}(\mathbf{x})(\mathbf{x} - \mathbf{g}_i)$ in Assumption \ref{assumption:unit_normal} is equal to $-\mathbf{P}(\mathbf{x})\mathbf{g}_i$.
Therefore, since $d_{\mathcal{U}}$ is continuously differentiable over $\mathcal{N}_i^{\delta_u}$ and $\boldsymbol{\nu}_d$ is continuous in $\overline{\mathcal{N}_i^{\delta_u}}$, by Assumption \ref{assumption:unit_normal}, there exists $\delta_d^i\in (0, \delta_u)$ and $\mu_d^i > 0$ such that
\[
\nabla_{\mathbf{x}}d_{\mathcal{U}}(\mathbf{x})^\top\boldsymbol{\nu}_d\geq \mu_d^i, \;\forall \mathbf{x}\in\mathcal{N}_i^{\delta_d^i}.
\]
Since $\mathbb{I}$ is finite and  $\overline{\mathcal{N}_i^{\delta_u}}\cap\overline{\mathcal{N}_j^{\delta_u}} = \emptyset, \forall i,j\in\mathbb{I}, i\ne j$, it follows that the closed-loop system \eqref{ideal_kinematics}-\eqref{nu_d} satisfies Property \ref{condition:Assump2:moveaway} with $\delta_d = \min_{i\in\mathbb{I}} \{\delta_d^i\}$ and $\mu = \min_{i\in\mathbb{I}}\mu_d^i$.

Since $\epsilon\in(0,\delta_u)$, it follows from Remark \ref{remark:satisfaction-property-D2} that $d_{\mathcal{U}}$, defined in \eqref{spherical_distance_definition}, is continuously differentiable on $\mathcal{N}^{\epsilon}$.
Together with the properties of $\alpha$ stated in Section~\ref{section:explicit_feedback_control}, this ensures that the vector field $\boldsymbol{\nu}_d$ in \eqref{nu_d} is continuously differentiable on $\mathcal{M}$.
Hence, $\boldsymbol{\nu}_d$ satisfies
Property~\ref{condition:Assump2:differentiability}.
This guarantees the uniqueness of the solution of the closed-loop system \eqref{ideal_kinematics}-\eqref{nu_d}.

It can be verified that the equilibrium points of the closed-loop system \eqref{ideal_kinematics}-\eqref{nu_d} are isolated.
Furthermore, by evaluating the eigenvalues of the Jacobian matrix at the equilibrium points, it can be established that Property \ref{condition:eigenvalues_of_equilibria} holds for the closed-loop system \eqref{ideal_kinematics}-\eqref{nu_d}.
Since $\boldsymbol{\nu}_d$ is continuous on $\overline{\mathcal{M}}$, and $\overline{\mathcal{M}}$ is a compact subset of $\mathbb{S}^n$, it follows that Property \ref{property:bounded_vector_field} holds.

Finally, we show that the closed-loop system \eqref{ideal_kinematics}-\eqref{nu_d} satisfies Property \ref{condition:set_of_equilibria}. 
In particular, we show that its chain recurrent set, denoted by $\mathcal{C}_{\boldsymbol{\nu}_d}$, satisfies \[\mathcal{C}_{\boldsymbol{\nu}_d}\subset\mathcal{E}\cup\{\mathbf{x}_d\},\]where the definition of the chain recurrent set is provided in Appendix \ref{appendix:chain_recurrent_set_concept}.

We first show that $\mathcal{C}_{\boldsymbol{\nu}_d}$ is internally\footnote{A chain recurrent set $\mathcal{C}_{\boldsymbol{\nu}_d}$ is said to be internally chain recurrent if for every $\mathbf{x}\in\mathcal{C}_{\boldsymbol{\nu}_d}$ and for every $\varepsilon, T > 0$, there exists an $(\varepsilon, T)$-chain from $\mathbf{x}$ to itself entirely contained in $\mathcal{C}_{\boldsymbol{\nu}_d}$, where the notion of chain recurrence and $(\varepsilon, T)$-chain are explained in Appendix \ref{appendix:chain_recurrent_set_concept}.} chain recurrent.
Since $\overline{\mathcal{M}}$ is a compact subset of $\mathbb{S}^n$ and is forward invariant for the closed-loop system \eqref{ideal_kinematics}-\eqref{nu_d}, it follows from 
% By Property \ref{condition:Assump2:moveaway}, the set \[\mathcal{K}_{\delta_d} = \{\mathbf{x}\in\mathcal{M}\mid d_{\mathcal{U}}(\mathbf{x})\geq \delta_d\}\] is compact and forward invariant for the closed-loop system \eqref{ideal_kinematics}-\eqref{nu_d}.
% Property \ref{condition:Assump2:moveaway} also guarantees that for every $\mathbf{x}(0)\in\mathcal{M}\setminus\mathcal{K}_{\delta_d}$, there exists $t(\mathbf{x}(0)) > 0$ such that $\Phi(t(\mathbf{x}(0)), \mathbf{x}(0))\in\mathcal{K}_{\delta_d}$.
% Consequently, it can be shown that $\mathcal{C}_{\boldsymbol{\nu}_d}\subset\mathcal{K}_{\delta_d}$. It then follows from 
\cite[Theorem 5.5]{Benaim1999Dynamics} that $\mathcal{C}_{\boldsymbol{\nu}_d}$ is internally chain recurrent.

An argument analogous to that used in the proof of \cite[Theorem 2]{sawant2026constrained} establishes almost global asymptotic stability of $\mathbf{x}_d$ for the closed-loop system \eqref{ideal_kinematics}-\eqref{nu_d} over $\overline{\mathcal{M}}$.
Accordingly, $\overline{\mathcal{M}}$ can be decomposed as \
\begin{equation}\label{eq:Xd_Xu_decomposition}\overline{\mathcal{M}} = \mathcal{X}_d \cup\mathcal{X}_u \text{ with } \mathcal{X}_d\cap\mathcal{X}_u = \emptyset,\end{equation}
where $\mathcal{X}_d$ is the basin of attraction of the desired equilibrium $\mathbf{x}_d$, whereas $\mathcal{X}_u$ is one dimensional subset of $\overline{\mathcal{M}}$ and is the union of the basins of attraction of the undesired equilibria in $\mathcal{E}$.

Since $\mathbf{x}_d$ is an asymptotically stable equilibrium point with the basin of attraction $\mathcal{X}_d$, no point in $\mathcal{X}_d\setminus\{\mathbf{x}_d\}$ is chain recurrent.
Consequently,
\begin{equation}\label{chain_recurrent_in_Xd}
\mathcal{C}_{\boldsymbol{\nu}_d}\cap\mathcal{X}_d \subset\{\mathbf{x}_d\}.
\end{equation}

Moreover, $\mathbf{x}_d$ is isolated in $\mathcal{C}_{\boldsymbol{\nu}_d}$. Define
\begin{equation}\label{Cu-definition}
    \mathcal{C}_u = \mathcal{C}_{\boldsymbol{\nu}_d}\setminus\{\mathbf{x}_d\}.
\end{equation}
Since $\mathcal{C}_{\boldsymbol{\nu}_d}$ is compact and internally chain recurrent, and $\mathbf{x}_d$ is isolated in $\mathcal{C}_{\boldsymbol{\nu}_d}$, $\mathcal{C}_u$ is itself internally chain recurrent.
From \eqref{eq:Xd_Xu_decomposition}, \eqref{chain_recurrent_in_Xd} and \eqref{Cu-definition}, it follows that
\begin{equation}\label{Cu_in_Xu}
    \mathcal{C}_u\subset\mathcal{X}_u.
\end{equation}

It remains to characterize the chain recurrent points contained in $\mathcal{X}_u$.
Following \cite[(20)]{sawant2026constrained}, define, for each $i\in\mathbb{I}$,
\begin{equation}
    \mathcal{Z}_i = \mathcal{G}(\mathbf{g}_i, -\mathbf{x}_d) \cup \mathcal{G}(-\mathbf{g}_i, -\mathbf{x}_d)
\end{equation}
where $\mathcal{G}(\mathbf{g}_i, -\mathbf{x}_d)$ denotes the unique geodesic connecting $\mathbf{g}_i$ to $-\mathbf{x}_d$, defined as in \eqref{geodesic_expression}.
One can verify that the undesired equilibria satisfy
\begin{equation}\label{gamma_definition}
    \mathcal{E}\setminus\{-\mathbf{x}_d\} \subset \bigcup_{i\in\mathbb{I}}\Gamma_i \subset\mathcal{X}_u, \; \Gamma_i = \overline{\mathcal{R}_i^{\delta}}\cap\mathcal{Z}_i, i\in\mathbb{I}.
\end{equation}
Moreover, since $\epsilon \in (0, \delta)$, by Assumption \ref{assumption:sufficient_separation}, it holds that $-\mathbf{x}_d$ is an equilibrium point of the closed-loop system \eqref{ideal_kinematics}-\eqref{nu_d} and is a repeller.

Let $\mathbb{I}_a = \{i\in\mathbb{I}\mid -\mathbf{x}_d\notin\mathcal{U}_i\}$.
Using Assumption \ref{assumption:sufficient_separation}, one can verify that, for each $i\in\mathbb{I}_a$, the set $\Gamma_i$ is forward invariant for the closed-loop system \eqref{ideal_kinematics}-\eqref{nu_d}. 
On the other hand, an argument analogous to that used in the proof of \cite[Lemma 8]{sawant2026constrained} shows that, if $\mathbf{x}(0)\in\overline{\mathcal{R}_i^{\delta}}$ with $i\in\mathbb{I}\setminus\mathbb{I}_a$, then there exists $t_1 > 0$ such that \[\mathbf{x}(t)\notin\overline{\mathcal{R}_i^{\delta}}, \;\forall t \geq t_1.\]
Furthermore, it can be established that for every $\mathbf{x}(0)\in\mathcal{X}_u$, there exists $i\in\mathbb{I}_a$ and $t_1\geq 0$ such that 
\[\mathbf{x}(t_1)\in\Gamma_i.\]

For each $i\in\mathbb{I}$, $\Gamma_i$ is a one-dimensional curve with two distinct end points and is homeomorphic to the compact interval $[0, 1]$.
Moreover, since $\boldsymbol{\nu}_d$ is locally Lipschitz continuous on $\mathcal{M}$, for every initial condition in $\mathcal{M}$, the solutions of \eqref{ideal_kinematics}-\eqref{nu_d} are unique.
Therefore, any non-equilibrium trajectory entirely contained in $\Gamma_i$ must evolve in a single direction along this curve and cannot return to a previous point. 
Consequently, the closed-loop system admits no nontrivial periodic orbit, homoclinic orbit, or heteroclinic cycle entirely contained in $\Gamma_i$ for any $i\in\mathbb{I}_a$.
Hence, every chain recurrent point contained in $\bigcup_{i\in\mathbb{I}_a}\Gamma_i$ is an equilibrium point of \eqref{ideal_kinematics}-\eqref{nu_d}. 
In other words,
\begin{equation}\label{Cu_in_Gamma}
    \mathcal{C}_u\cap\bigcup_{i\in\mathbb{I}_a}\Gamma_i \subset\mathcal{E}.
\end{equation}

It remains to show that there are no chain recurrent points in $\mathcal{X}_u\setminus\bigcup_{i\in\mathbb{I}_a}\Gamma_i$.
For any $\mathbf{x}\in\mathcal{X}_u\setminus\bigcup_{i\in\mathbb{I}_a}\Gamma_i$, there exists $t_1 > 0$ such that $\Phi(t_1, \mathbf{x})\in\Gamma_i$ for some $i\in\mathbb{I}_a$, where
$\Phi(t_1, \mathbf{x})$ denotes the point at time $t_1\geq0$ along the unique solution of $\dot{\mathbf{x}} = \boldsymbol{\nu}_d(\mathbf{x})$ initialized at $\mathbf{x}\in\overline{\mathcal{M}}$.
In addition, as stated earlier, $\Gamma_i, i\in\mathbb{I}_a$, is a compact subset of $\overline{\mathcal{M}}$ and forward invariant for the closed-loop system \eqref{ideal_kinematics}-\eqref{nu_d}.
Consequently, it can be shown that there exist $\varepsilon>0$ and $T>0$ such that no $(\varepsilon,T)$-chain entirely contained in $\mathcal{X}_u$ can return from $\mathbf{x}$ to itself.
Therefore, it follows from \eqref{Cu_in_Gamma} that
\begin{equation}\label{chain_recurrent_in_Xu}
    \mathcal{C}_u\subset\mathcal{E}.
\end{equation}

Finally, since $\overline{\mathcal{M}} = \mathcal{X}_d\cup\mathcal{X}_u$, combining \eqref{chain_recurrent_in_Xd}, \eqref{Cu-definition}, and \eqref{chain_recurrent_in_Xu} gives
\begin{equation}
    \mathcal{C}_{\boldsymbol{\nu}_d}\subset\mathcal{E}\cup\{\mathbf{x}_d\}.
\end{equation}
Therefore, the closed-loop system \eqref{ideal_kinematics}-\eqref{nu_d} satisfies Property \ref{condition:set_of_equilibria}.

\end{appendix}
%%%%%%%%%%%%%%%%%%%%%%%%%%%%%%

\bibliographystyle{IEEEtran}
\bibliography{reference}

\end{document}